\documentclass[10pt]{amsart}

\title[A LG mirror theorem without concavity]{A Landau--Ginzburg mirror theorem \\ without concavity}

\author[Gu\'er\'e]{J\'er\'emy Gu\'er\'e}
\address{Institut de Math\'ematiques de Jussieu\\
Paris\\
France}
\email{jeremy.guere@gmail.com}

\usepackage{amsrefs}
\usepackage{amsfonts}
\usepackage{amsmath}
\usepackage{amssymb}
\usepackage{amscd}
\usepackage{array}
\usepackage{subfigure}
\usepackage{tikz}

\usepackage[all]{xy}

\allowdisplaybreaks[3]

\newcommand{\SL}{\underline{\mathrm{SL}}}
\newcommand{\e}{\mathfrak{e}}
\newcommand{\grj}{\mathfrak{j}}
\newcommand{\gru}{\mathfrak{u}}
\newcommand{\grv}{\mathfrak{v}}
\newcommand{\kK}{\mathbf{K}}
\newcommand{\fc}{\mathfrak{c}}
\newcommand{\PV}{\mathbf{PV}}

\newcommand{\MF}{matrix factorization }
\newcommand{\MFs}{matrix factorizations }
\newcommand{\bw}{\mathbf{w}}
\newcommand{\cE}{\mathcal{E}}
\newcommand{\ccE}{\textrm{e}}
\newcommand{\du}{\vec{u}} 
\newcommand{\dv}{\vec{v}} 
\newcommand{\dw}{\vec{e}} 
\newcommand{\AB}{\mathfrak{R}}
\newcommand{\aA}{\mathbb{A}}
\newcommand{\bB}{\mathfrak{B}}
\newcommand{\rR}{\mathfrak{C}}
\newcommand{\ci}{\mathrm{i}}
\newcommand{\PP}{\mathbb{P}}
\newcommand{\CC}{\mathbb{C}}
\newcommand{\ZZ}{\mathbb{Z}}
\newcommand{\NN}{\mathbb{N}}

\newcommand{\QQ}{\mathbb{Q}}

\newcommand{\cO}{\mathcal{O}}

\renewcommand{\(}{\left(}
\renewcommand{\)}{\right)}

\newcommand{\ii}{\mathrm{i}}
\newcommand{\cC}{\mathcal{C}}
\newcommand{\cD}{\mathcal{D}}
\newcommand{\cF}{\mathcal{F}}
\newcommand{\cH}{\mathcal{H}}

\newcommand{\cL}{\mathcal{L}}
\newcommand{\cM}{\mathcal{M}}
\newcommand{\bP}{\mathbf{P}}
\newcommand{\cP}{\mathcal{P}}
\newcommand{\cQ}{\mathcal{Q}}

\newcommand{\fq}{\mathfrak{q}}

\newcommand{\cS}{\mathrm{Sym}}
\newcommand{\sS}{\mathcal{S}}

\newcommand{\Ch}{\mathrm{Ch}}
\newcommand{\Td}{\mathrm{Td}}
\newcommand{\ww}{\underline{w}}

\newcommand{\st}{\mathbf{H}}

\newcommand{\cvir}{c_{\mathrm{vir}}}
\newcommand{\cvirPV}{c_{\mathrm{vir}}^{\mathrm{PV}}}
\newcommand{\cvirFJRW}{c_{\mathrm{vir}}^{\mathrm{FJRW}}}

\newcommand{\Spec}{\textrm{Spec}}

\newcommand{\bs}{{\boldsymbol{s}}}

\newcommand{\rank}{\operatorname{rank}}
\DeclareMathOperator{\Aut}{Aut}

\theoremstyle{plain}

\newtheorem{thm}{Theorem}[section]
\newtheorem{pro}[thm]{Proposition}
\newtheorem{lem}[thm]{Lemma}
\newtheorem*{lem*}{Lemma}

\newtheorem{cor}[thm]{Corollary}

\theoremstyle{definition}

\newtheorem{dfn}[thm]{Definition}

\newtheorem{rem}[thm]{Remark}
\newtheorem*{rem*}{Remark}
\newtheorem*{rems*}{Remarks}
\newtheorem{exa}[thm]{Example}
\newtheorem*{exa*}{Example}
\newtheorem*{exait*}{\rm \em Example}
\newtheorem*{exadefit*}{\rm \em Example/Definition}
\newtheorem*{cla*}{\rm \em Claim}

\newtheorem*{dfn*}{Definition}

\usepackage{amsxtra}

\def\<{\left\langle}
\def\>{\right\rangle}

\begin{document}
\begin{abstract}
We provide a mirror symmetry theorem in a range of cases where the
state-of-the-art techniques relying on concavity or convexity do not
apply. More specifically, we work on a family of FJRW potentials named
after Fan, Jarvis, Ruan, and Witten's quantum singularity theory and
viewed as the counterpart of a non-convex Gromov--Witten potential via
the physical LG/CY correspondence.
The main result provides an explicit formula for Polishchuk and Vaintrob's virtual cycle in genus zero.
In the non-concave case of the so-called chain invertible polynomials, it yields a compatibility theorem with the FJRW virtual cycle
and a proof of mirror symmetry for FJRW theory.
\end{abstract}

\maketitle

\tableofcontents

\setcounter{section}{-1}

\section{Introduction}
In the last two decades, mirror symmetry has been a central statement in theoretical physics and a fundamental driving force for several developments in mathematics. For instance it can be phrased mathematically as a prediction on Gromov--Witten invariants, namely the intersection numbers attached to curves traced on a Calabi--Yau variety. In this form, it has been proven in a vast range of concrete cases: the most famous example provides a full computation of the genus-zero invariants enumerating rational curves on the quintic threefold \cites{Gi,LLY}.
Even in this case, it is often pointed out how we still lack a complete computation in higher genus (only the genus-one case was completely proven by Zinger \cite{Zi}).

But even in genus zero the problem of computing Gromov--Witten invariants of projective varieties is far from being completely solved; indeed, most known techniques focus on computing Gromov--Witten invariants attached to cohomology classes which lie in the so-called ambient part of cohomology: the restriction to classes from the ambient projective space.
For the quintic threefold, working with ambient cohomology classes turns out to determine the entire theory; however, in general, this scheme covers only a tiny portion of quantum cohomology.
Remarkably, even the ambient cohomology classes may pose problem as soon as we work with orbifolds.

It is interesting to notice that these gaps in Gromov--Witten computation all arise from the same phenomena: as we argue below, certain positivity or negativity conditions named convexity and concavity are not always satisfied, making the virtual cycle\footnote{It is a crucial cycle of the moduli space of the theory, since Gromov--Witten invariants are intersection numbers of cohomological classes against the virtual cycle.} challenging to compute.
In genus one, this difficulty was overcome by Zinger after a great deal of hard work, but we still lack a comprehensive approach for higher genus.
Guided by the frame of ideas of mirror symmetry and the Landau--Ginzburg/Calabi--Yau correspondence, we switch to the quantum theory of singularities introduced by Fan, Jarvis, and Ruan \cites{FJRW,FJRW2} based on ideas of Witten \cite{Witten} (FJRW theory).
Single non-concave quantum invariants were inferred from concave ones using tautological relation (e.g.~using WDVV equation as in \cites{FJRW,KShen,Shen2}), but so far no systematic approach tackling directly the virtual cycle has been taken.
Polishchuk and Vaintrob recently opened the way to an algebraic computation: their construction \cite{Polish1} of a virtual cycle is given by applying the Chern character to a universal object, the fundamental matrix factorization. 

In this paper we build upon this work an effective method for computing this algebraic cycle in genus zero.
In a range of non-concave singularities --- the so-called chain potentials --- we provide a systematic computation of the genus-zero intersection numbers within Givental's quantum Riemann--Roch formalism.
We also prove a compatibility theorem between Polishchuk and Vaintrob's cycle and FJRW virtual cycle.
Altogether, we end with a mirror symmetry theorem for FJRW theory.
By casting this non-concave computation within Givental formalism, this paper takes a step forward in FJRW theory computing the virtual cycle and ultimately quantum invariants even beyond ambient cohomology. 
Next step would be to seek an extension of this genus-zero non-concave result to higher genus, and possibly to Gromov--Witten theory.

\subsection{Non-convexity in genus zero}
Consider a degree-$d$ hypersurface $X$ within the projective space $\PP^{d-1}$. Convexity holds in genus-$g$ if for every stable map $f \colon C \to \PP^{d-1}$, from a genus-$g$ curve $C$, we have $H^1(C,f^*\cO(d))=0$ (see \cite[Example B]{Coates3}).
When satisfied, this condition allows us to deduce the Gromov--Witten theory of $X$ from the known Gromov-Witten theory of $\PP^{d-1}$ via the quantum Lefschetz principle \cite{Coates2}.
For Calabi--Yau hypersurfaces, the condition holds in genus-zero as soon as the ambient projective space is Gorenstein, which is of course always the case for ordinary projective spaces.
But convexity can fail in genus-zero for (Calabi--Yau) hypersurfaces in the orbifold setting, as the weighted projective space $\PP(w_1,\dotsc,w_N)$ is Gorenstein if and only if $w_j$ divides $\sum_k w_k$ for all $j$.
Coates \emph{et al} discuss extensively the possibility of stable genus-zero maps with nontrivial obstruction in
``\textit{The Quantum Lefschetz Hyperplane Principle Can Fail for Positive Orbifold Hypersurfaces}'' \cite{Coates1}.

\begin{equation*}
\begin{array}{|cc|cc|}
\hline
\forall ~ j ~, ~ w_j \mid \sum\limits_{k=1}^N w_k & \quad & \quad & \exists ~ j \textrm{ s.t. } w_j \not \vert \sum\limits_{k=1}^N w_k \\[0.05cm]
\hline
\PP(w_1,\dotsc,w_N) \textrm{ is Gorenstein} & \quad & \quad & \PP(w_1,\dotsc,w_N) \textrm{ is not Gorenstein} \\[0.05cm]
\hline
\textrm{Convexity holds in genus-zero} & \quad & \quad & \textrm{Convexity can fail in genus-zero} \\[0.05cm]
\textrm{for CY hypersurfaces in } \PP(\ww) & \quad & \quad & \textrm{for CY hypersurfaces in } \PP(\ww)\\[0.05cm]
\hline
\end{array}
\end{equation*}

\subsection{Invertible polynomials}
Berglund--H\"ubsch's and Krawitz's treatment \cites{Hubsch,Krawitz} of hypersurfaces via the so-called invertible polynomials provides illuminating examples of non-Gorenstein ambient space. 
Invertible polynomials are characterized for having as many monomials as variables
\begin{equation*}
W(x_1, \dotsc ,x_N)=x_1^{m_{1,1}} \dotsm x_N^{m_{1,N}}+ \dotsb +x_1^{m_{N,1}}\dotsm x_N^{m_{N,N}}
\end{equation*}
and are naturally encoded by an invertible matrix of exponents $E_W=(m_{k,j})$. 
In this way $W$ is quasi-homogeneous with respect to a unique choice of the weights, and 
we further impose non-degeneracy: $\partial W(\underline{x})/\partial x_j=0, \ 
\forall j   \implies   x_j=0,\ \forall j$.
Thus, the zero locus $\left\lbrace W=0 \right\rbrace$ in the corresponding weighted projective space is a smooth (stack-theoretic) hypersurface and is Calabi--Yau as soon as the degree of $W$ equals the sum of the weights.
These polynomials were completely classified by Kreuzer and Skarke \cite{Kreuzer} as Thom--Sebastiani sums\footnote{A Thom--Sebastiani sum is a sum of terms with disjoint sets of variables.} of terms of type Fermat, chain or loop
\begin{enumerate}
\item[(a)] $x^{a+1}$;
\item[(b)] $x_1^{a_1}x_2+\dotsb+x_{c-1}^{a_{c-1}} x_c+x_c^{a_c+1}$ \quad with $c \geq 2$;
\item[(c)] $x_1^{a_1}x_2+\dotsb+x_{l-1}^{a_{l-1}} x_l+x_l^{a_l}x_1$ \quad \ with $l \geq 2$.
\end{enumerate}
Clearly, if only Fermat terms occur, the weights $w_j$ divide the degree $d$. 
On the other hand, the Gorenstein condition may fail when chain terms appear, consider for instance 
the chain polynomial $W=x_1^2x_2+ x_2^3x_3+x_3^5x_4+x_4^{10}x_5+x_5^{11}$ with weights $(4,3,2,1,1)$ and degree $11$.

Under the Calabi--Yau assumption, these polynomials allow an explicit and elementary construction of the (conjectural) mirror, by Berglund--H\"ubsch \cite{Hubsch} and by Krawitz \cite{Krawitz}:
\begin{equation*}
\begin{array}{cc}
\{W=0\}     & \\
& \textrm{A-side}   \\
\hline
\hline
& \textrm{B-side} \\
\left[ \{W^\vee =0\} / \SL(W^\vee) \right]   &  \\
\end{array}
\end{equation*}
where the quotient on the B-side is a stack quotient, $W^\vee$ is the invertible polynomial encoded by the transposed matrix $(E_W)^\textrm{T} = (m_{j,k})$ and $\SL(W^\vee)$ is a group containing automorphisms of $W^\vee$ of determinant $1$, see Remark \ref{SL}.
Since mirror symmetry yields explicit Gromov--Witten theory predictions, 
these examples are 
ideal for developing new techniques in non-convex cases. 

\subsection{LG/CY correspondence}
As mentioned above, precisely as in Witten \cite[\S 3.1]{Witten}, we look at Calabi--Yau via singularities; namely we pass from a Calabi--Yau (CY) hypersurface $\{W=0\}\subset \PP(\ww)$ to the Landau--Ginzburg (LG) model $W \colon \CC^N \to \CC$ whose monodromy around the origin is given by the group $\mu_d = \langle \grj \rangle$ of order $d$ generated by the grading element
\begin{equation*}
\grj := \textrm{diag}(\exp(2 \ci \pi w_1 /d), \dotsc ,\exp(2 \ci \pi w_N /d)).
\end{equation*}

This quantum theory of singularities was recently introduced under the name of FJRW theory by Fan, Jarvis, and Ruan \cites{FJRW,FJRW2} building upon Witten's initial analytic construction \cite{Witten}.
Polishchuk and Vaintrob have provided in \cite{Polish1} an algebro-geometric counterpart.
The compatibility between \cite{FJRW} and \cite{Polish1} is only partly proven, for instance for simple singularities \cite[Theorem 7.6.1]{Polish1} or for invariants with so-called narrow entries \cite[Theorem 1.2]{Li2}.
In Theorem \ref{compat}, we establish it also for (almost) every invertible polynomials with the maximal group of symmetries.

This paper is the first proof of one of the numerous conjectures beyond concavity casted by Chiodo and Ruan \cite{Chiodo2} within the unified framework of global mirror symmetry, made precise in the case of Fermat polynomials by Chiodo, Iritani, and Ruan \cite{LG/CY}. In this framework the above LG/CY correspondence is pictured as the mirror of the pencil deforming $\{W^\vee=0\}$ along a weighted projective line $\PP(d,1)$. 
More precisely, we show mirror symmetry at the point $t=0$ for every chain polynomial.
We hope this technique will shed new light on other non-convex cases that occur systematically in the global mirror symmetry conjecture.

\begin{figure}[ht]
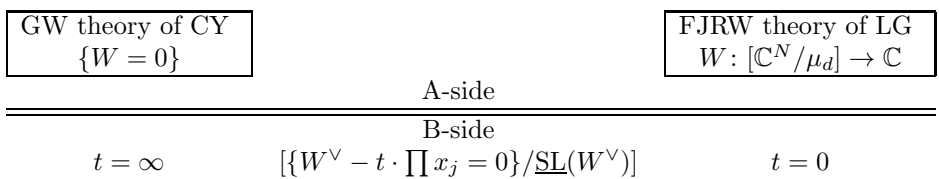

\begin{equation*}
\begin{array}{@{}ccc@{}}
\begin{array}{|c|}
\hline
\textrm{GW theory of CY }       \\
\{W=0\}  \\
\hline                       
\end{array} & & \begin{array}{|c|}
\hline
\textrm{FJRW theory of LG } \\
W \colon [\CC^N/\mu_d] \to \CC \\
\hline
\end{array} \\
 & \textrm{A-side} & \\
\hline
\hline
 & \textrm{B-side} & \\
t = \infty                      & \left[ \{ W^\vee -t \cdot \prod x_j =0 \} / \SL(W^\vee)\right]  & t=0  \\
                                &   & \\
\end{array}
\end{equation*}
\caption{Global Mirror Symmetry.}
\label{GMS}
\end{figure}


\subsection{Landau--Ginzburg side and concavity, a change of perspective}\label{intro4} 
The FJRW theory of $W\colon [\CC^N/\mu_d]\to \CC$ is radically different from the GW theory of $\{W=0\}$, but they have some aspects in common.
For instance, the dichotomy between ambient cohomology classes and primitive cohomology classes for $\{W=0\}$ emerges also on the LG side. Usually referred as states, the cohomology classes of an LG model are cohomology classes (relative to the Milnor fiber) of the inertia stack 
\begin{equation*}
I_{[\CC^N/\mu_d]}=\bigsqcup_{0\leq m <d} [\mathrm{Fix}(\CC^N,\grj^m)/\mu_d],
\end{equation*}
where $\mathrm{Fix}(\cdot,\gamma)$ is the fixed point set with respect to $\gamma$.
The counterpart of ambient classes simply corresponds to narrow states: cohomology classes on the zero-dimensional sectors of the inertia stack. The remaining higher dimensional sectors produce the counterpart to primitive cohomology: broad states (see \cite{statespace}). 

\begin{rem}
The entire paper deals with FJRW theory of invertible polynomials $W$ together with the maximal group of symmetries $\Aut(W)$; we only present the theory with $\mu_d$ to mention LG/CY correspondence.
Nevertheless, in genus zero and as far as each broad state involved comes from the theory with $\Aut(W)$, the theory with maximal group embodies theories with smaller groups such as $\mu_d$.
But in general, there are other broad states that we do not know how to treat.
\end{rem}

Interestingly, the counterpart of the convexity ($H^1=0$) of Gromov--Witten theory is the concavity ($H^0=0$) of FJRW theory.
More precisely, the geometric object considered by FJRW theory of $W \colon [\CC^N/\mu_d]\to \CC$
is an $n$-pointed  curve $\cC$ with a line bundle $\cL$ such that $\cL^{\otimes d} \simeq \omega_{\cC,\textrm{log}} :=\omega_\cC(\sigma_1+\dotsb+\sigma_n)$, where $\sigma_1,\dotsc,\sigma_n$ are the marked points. 
Using the weights of $W$, we define line bundles $\cL_1:=\cL^{w_1},\dotsc,$ $\cL_N:=\cL^{w_N}$.
Concavity in genus-$g$ means
\begin{equation*}
H^0(\cC,\cL_j) = 0 ~,~~ \forall j,
\end{equation*}
for every genus-$g$ objects.
\begin{equation*}
\begin{array}{|cc|cc|}
\hline
\forall ~ j ~, ~ w_j \mid d & \quad & \quad & \exists ~ j \textrm{ s.t. } w_j \not \vert d \\[0.05cm]
\hline
\textrm{Concavity holds in genus-zero} & \quad & \quad & \textrm{Concavity always fails in genus-zero} \\[0.05cm]
\textrm{for polynomial singularities with} & \quad & \quad & \textrm{for polynomial singularities with} \\[0.05cm]
\textrm{weights } (w_1,\dotsc,w_N) \textrm{ and degree } d & \quad & \quad & \textrm{weights } (w_1,\dotsc,w_N) \textrm{ and degree } d \\[0.05cm]
\hline
\end{array}
\end{equation*}

\subsection{An invertible limit to the Euler class}\label{0.5b}
The virtual cycle is central to FJRW theory. It is a well-defined cycle whose construction is non-trivial, but, at least when restricted to narrow states, it provides in a very special case the answer to a natural algebro-geometric question which is open in general: what is an Euler class of a push-forward of a vector bundle?

In FJRW theory, we are confronted with the following situation. 
Consider a family of curves $\pi \colon \cC \rightarrow S$ over a smooth and proper base $S$, together with universal line bundles $\cL_1, \dotsc, \cL_N$ satisfying certain algebraic relations. We set $\cE:= \cL_1 \oplus \dotsb \oplus \cL_N$.
Under the concavity assumption, the sheaf $R^1\pi_*\cE$ is a vector bundle and the push-forward $R^0\pi_*\cE$ vanishes.
Therefore we define the virtual cycle as the Poincar\'e dual of the Euler class of the vector bundle $R^1\pi_*\cE$, that is
\begin{equation*}
\cvir := c_\textrm{top}(R^1\pi_*\cE).
\end{equation*}
Without concavity, we have to deal with both $R^0\pi_*\cE$ and $R^1\pi_*\cE$, and none of them is a vector bundle in general.
Remarkably, the condition $\cL^{\otimes d} \simeq \omega_{\pi,\textrm{log}}$ on the universal line bundle gives us the opportunity to extend the definition of Euler class to the K-theoretic element $R^\bullet\pi_*\cE$, in a way that respects the multiplicative property\footnote{The multiplicative property $c_\textrm{top}(V+W)=c_\textrm{top}(V) \cdot c_\textrm{top}(W)$ is the starting point of cohomological field theory, namely the factorization property.} of Euler class.
Namely, we use Polishchuk--Vaintrob's algebraic class, revisited as the cohomology of a recursive complex (see Definition \ref{defrecursiveMF}) and computed as the limit of a specific characteristic class, that we now explain.

For a vector bundle $V$ on $S$ and a parameter $t \in \CC$, let us define the class
\begin{equation}\label{charclass}
\fc_t(V) := \Ch (\lambda_{-t} V^\vee) \cdot \Td (V) \in H^*(S)\left[ t\right],
\end{equation} 
where $\lambda_{-t}$ denotes the $\lambda$-ring structure of K-theory according to \cite{Fult}, that is
\begin{equation*}
\lambda_{t} V := \sum_{k \geq 0} \Lambda^k V \cdot t^k \in K^0(S)\left[ t\right] .
\end{equation*}
The classes $\lambda_{t} V$ and $\fc_t(V)$ are invertible in $K^0(S)[\![ t ]\!]$ and in $H^*(S)[\![ t ]\!]$ and, by \cite[Proposition 5.3]{Fult}, we have
\begin{equation*}
\lim_{t \to 1} \fc_t(V) = c_\textrm{top}(V).
\end{equation*}
For two vector bundles $A$ and $B$, we define the function $\fc_t \colon K^0(S) \rightarrow H^*(S)[\![ t ]\!]$ by
\begin{equation}\label{charclass2}
\fc_t(B-A) := \frac{\fc_t(B)}{\fc_t(A)} = \Ch \biggl( \frac{\lambda_{-t}B^\vee}{\lambda_{-t}A^\vee} \biggr) \cdot \cfrac{\Td B}{\Td A} \in H^*(S)[\![ t ]\!],
\end{equation}
and the limit $t \to 1$ generally diverges, since the class $c_\textrm{top}$ is not invertible.

\begin{rem}
The class $\fc_t(V)$ differs from the equivariant Euler class of $V$ appearing in previous similar approaches to quantum cohomology, as in \cite{Coates2}.
In terms of the roots $\alpha_1,\dotsc,\alpha_v$ of the vector bundle $V$, the class $\fc_t(V)$ is\footnote{Strictly speaking, the class $\fc_t(\cdot)$ is defined for $t \neq 1$ in terms of Chern characters via the formula \eqref{cvirx}. Indeed, the power series \eqref{charclass2} has only a radius of convergence equal to $1$ and the formulas \eqref{charclass2} and \eqref{cvirx} coincide for $\left| t \right| <1$, by Lemma \ref{formaldev}.}
\begin{equation*}
\fc_t(V) = \prod_{k=1}^v \frac{e^{\alpha_k}-t}{e^{\alpha_k}-1} \cdot \alpha_k, \qquad \textrm{for $t \neq 1$.}
\end{equation*}
\end{rem}

We state straight away the main technical result of the paper in the case of a chain polynomial with narrow states.
We refer to Theorem \ref{main} for a complete statement, including broad states and other invertible polynomials.

\begin{thm}[See Theorem \ref{main}]\label{mainb}
In the narrow sector and for a chain polynomial $W=x_1^{a_1}x_2+\dotsb+x_{N-1}^{a_{N-1}}x_N+x_N^{a_N+1}$, the following limit converges and equals the genus-zero virtual class
\begin{equation*}
\cvir =  \lim_{t \to 1} \prod_{j=1}^N \fc_{t_j}(-R \pi_*(\cL_j))
\end{equation*}
with $t_j:=t^{(-a_1)\dotsm (-a_{j-1})}$.
\end{thm}

When broad entries are involved, Theorem \ref{main} still bears information on the computation of Polishchuk and Vaintrob's algebraic class.
In general, the compatibility with FJRW virtual class is still an open question, but we proved it for (almost) every invertible polynomials with the maximal group of symmetries.

\begin{thm}[See Theorem \ref{compat}]
Let us consider invertible polynomials with no monomials of the form $x^ay+y^2$ or $x^ay+y^2x$.
For every such polynomials with maximal group of symmetries and in every genus, Polishchuk and Vaintrob's algebraic class coincides with Fan--Jarvis--Ruan--Witten virtual class, up to a rescaling of the broad sector.
\end{thm}

\subsection{Mirror symmetry}\label{0.5}
We go further and compute a big I-function for chain polynomials with maximal group of symmetries, see Theorem \ref{bigItheorem}.
In particular, we prove mirror symmetry for Calabi--Yau chain polynomials, see Theorem \ref{thmfinal}.

Let $W$ be a chain polynomial
\begin{equation*}
W=x_1^{a_1}x_2+\dotsb+x_{N-1}^{a_{N-1}}x_N+x_N^{a_N+1},
\end{equation*}
with weights $(w_1,\dotsc,w_N)$ and degree $d$.
The mirror polynomial $W^\vee$ is defined by
\begin{equation*}
W^\vee=y_1^{a_1}+y_1y_2^{a_2}+\dotsb+y_{N-1}y_N^{a_N+1},
\end{equation*}
with different weights and degree in general, \cites{Hubsch,Krawitz}.
We look at two local systems.

On the B-side, the local system is given by the primitive cohomology of the fibration $\left[ \{ W^\vee - t \cdot \prod x_j = 0 \} / \SL(W^\vee)\right]_t \rightarrow \Delta^*$ over a pointed disk $\Delta^* \subset \CC$ around $0$, together with the Gauss--Manin connection $\nabla^B$.
Then, we have a sub-local system $(\mathcal{E}^B,\nabla^B)$ attached to the Picard--Fuchs equation
\begin{equation*}
\biggl[ t^d \prod_{j=1}^N \prod_{c=0}^{w_j-1} (\frac{w_j}{d}   \frac{t \partial}{\partial t} + c) - \prod_{c=1}^d ( \frac{t \partial}{\partial t} - c)\biggr] \cdot f(t) = 0.
\end{equation*}
of the polynomial $W^\vee$, see e.g.~\cites{Morrisson,Ebelingstudent}.
Observe that the integers $w_j$ and $d$ in the Picard--Fuchs equation are the weights and the degree of the polynomial $W$.
A fundamental solution of the Picard--Fuchs equation is given by the I-function \eqref{finalformulaI}.

On the A-side, we take the trivial vector bundle given by the state space over itself, equipped with a connection $\nabla^A$ induced by genus-zero FJRW invariants.
This local system can be entirely recovered\footnote{The J-function is a fundamental solution of the scalar differential equation from the local system, see \cite[Definition 4.6.]{Guest}.} by the J-function
\begin{eqnarray*}
J(h,-z) = -z e_\grj + h + 
\sum_{\substack{n \geq 0 \\ l \geq 0}} ~~ \sum_{\substack{\gamma_1,\dotsc,\gamma_n,\widetilde{\gamma} \\ \in \Aut(W)}} \langle e_{\gamma_1},\dotsc,e_{\gamma_n},\tau_{l}(e_{\widetilde{\gamma}})  \rangle_{0,n+1} \cfrac{h^{\gamma_1}\dotsm h^{\gamma_n}}{n!(-z)^{l+1}} e^{\widetilde{\gamma}},
\end{eqnarray*}
where $h = \sum h^\gamma e_\gamma$ is the decomposition of an element of the state space on a basis and $e^{\widetilde{\gamma}}$ is dual to $e_{\widetilde{\gamma}}$ via the pairing of the state space.

\begin{thm}[See\footnote{Precisely, we prove in Theorem \ref{thmfinal} that the I-function is proportional to the restriction of the J-function to $\tau(\Delta^*)$. The embedding $\tau$ is called the mirror map.} Theorem \ref{thmfinal}]
Let $W$ be a Calabi--Yau chain polynomial and $\Delta^*$ be a sufficiently small pointed disk of $\CC$ around $0$.
There exists an explicit embedding $\tau$ of the pointed disk $\Delta^*$ into the state space of FJRW theory, such that we have an isomorphism of local systems
\begin{equation*}
\tau^*(\mathcal{E}^A , \nabla^A) \simeq (\mathcal{E}^B , \nabla^B)
\end{equation*}
over $\Delta^*$, where $(\mathcal{E}^A , \nabla^A)$ is the local system over $\tau(\Delta^*)$ entirely determined by the restriction of the J-function to $h \in \tau(\Delta^*)$.
\end{thm}

\subsection{Structure of the paper}
In Sect.~\ref{section1}, we reformulate Kreuzer's and Krawitz's  explicit description \cites{Kreuzer2,Krawitz} of the state space of invertible polynomials using a natural bookkeeping device of decorated graphs, and we describe the moduli space of $W$-spin curves.
In Sect.~\ref{section2}, we present Polishchuk--Vaintrob's construction \cite{Polish1} of their algebraic class for the quantum singularity theory. 
In Sect.~\ref{contrib}, we introduce the notion of a recursive complex which in several cases allows us to compute Polishchuk--Vaintrob's class, see Theorem \ref{main}.
We also check the compatibility (Theorem \ref{compat}) between this class and FJRW virtual class for (almost) every invertible polynomials in every genus.
In Sect.~\ref{computations}, we cast our computation within Givental's quantization formalism to obtain the big I-function \eqref{bigI} and we prove the mirror symmetry statement (Theorem \ref{thmfinal}) for FJRW theory of chain polynomials.

~~

\noindent
\textbf{Acknowledgement.}
The author is extremely grateful to his Ph.D.~supervisor Alessandro Chiodo for motivating discussions and invaluable advice on this paper.
He thanks Yongbin Ruan and Alexander Polishchuk for encouraging him to look for a compatibility Theorem \ref{compat}.
A special thank to Polishchuk for his illuminating talks on matrix factorizations and the discussions that followed. 
Dimitri Zvonkine's remarks and suggestions have been very helpful. 
It is a pleasure to acknowledge Hiroshi Iritani for a fruitful discussion on mirror symmetry formalism.
He is also extremely grateful to the reviewers for very interesting and helpful comments to enhance the readability of the paper and to suggest him to add Theorem \ref{bigItheorem}.
Finally, he thanks his wife Sol\`ene Molle for her constant support in this work.


\section{Quantum singularity theory}\label{section1}
Any Landau--Ginzburg (LG) orbifold carries a cohomological field theory, first introduced by Fan--Jarvis--Ruan \cites{FJRW,FJRW2} inspired by Witten \cite{Witten}.
We have two motivations to focus on invertible polynomials with diagonal automorphisms: there is a complete description of the state space and a definition for the (conjectural) mirror LG orbifold.
In Sect.~\ref{section1.1}, we encode any invertible polynomial into an oriented graph; in Sect.~\ref{sectionstatespace}, we write a basis of the state space by means of decorations on such graph.
In Sect.~\ref{sectionmodulispace}, we introduce the moduli space of $W$-spin curves and the quantum singularity theory.

\subsection{Conventions and notations}\label{QST}\label{term}
With respect to FJRW theory, two cohomological classes have been recently introduced: the FJRW virtual class by Fan, Jarvis, and Ruan \cites{FJRW,FJRW2} on one side and an algebraic class by Polishchuk and Vaintrob \cite{Polish1} on the other side.
Fan, Jarvis, and Ruan view this class analytically and follow Witten's initial sketched idea \cite{Witten} formalized for A-singularities by Mochizuki \cite{Moch}.
Polishchuk and Vaintrob provide an algebraic construction generalizing their previous construction and that of Chiodo \cite{ChiodoJAG} in the A-singularity case.
So far, little is known on the compatibility between these two approaches: Faber, Shadrin, and Zvonkine's work \cite{FSZ} may be regarded as a check of compatibility of all approaches in the A-singularities case, Polishchuk and Vaintrob push forward this check to all simple singularities in \cite{Polish1}.
Chang, Li, and Li prove the match when only narrow entries occur in \cite[Theorem 1.2]{Li2}.
Theorem \ref{compat} establishes the compatibility for (almost) every invertible polynomials with maximal group of symmetries.

Both cycles lead to cohomological field theories, see \cite[Theorem 4.1.8]{FJRW} for the FJRW virtual cycle and \cite[Sect.~5]{Polish1} for Polishchuk and Vaintrob's cycle.
Therefore, we also call Polishchuk and Vaintrob's class a virtual class and denote it by $\cvirPV$, while $\cvirFJRW$ stands for the FJRW virtual class.
Quantum invariant with the upper-script PV (resp.~FJRW) refers to an intersection number against $\cvirPV$ (resp.~$\cvirFJRW$).

We work in the algebraic category and over $\CC$.
All stacks are proper Deligne--Mumford stacks; we use also the term ``orbifold'' for this type of stacks.
We denote orbifolds by curly letters, e.g.~$\cC$ is an orbifold curve and the scheme $C$ is its coarse space.
We recall that vector bundles are coherent locally free sheaves and that the symmetric power of a two-term complex is the complex
\begin{equation*}
\cS^k\([ A \rightarrow B ]\)  =  [ \cS^k A \rightarrow \cS^{k-1} A \otimes B \rightarrow \dotso \rightarrow A \otimes \Lambda^{k-1} B \rightarrow \Lambda^k B ]
\end{equation*}
with morphisms induced by $A \rightarrow B$.

All along the text, the index $i$ varies from $1$ to $n$ and refers exclusively to the marked points of a curve whereas the index $j$ varies from $1$ to $N$ and corresponds to the variables of the polynomial.
We represent tuples by overlined notations, e.g.~$\overline{\gamma}=(\gamma(1),\dotsc,\gamma(n))$, or by underlined notations, e.g.~$\underline{p}=(p_1,\dotsc,p_N)$.

\subsection{Landau--Ginzburg orbifold}\label{section1.1}
Let $w_1,\dotsc,w_N$ be coprime positive integers and let $W$ be a quasi-homogeneous polynomial of degree $d$ with weights $w_1,\dotsc,w_N$.
For any $\lambda,x_1,\dotsc,x_N \in \CC$, we have
\begin{equation*}
W(\lambda^{w_1} x_1,\dotsc,\lambda^{w_N} x_N) = \lambda^d  W(x_1,\dotsc,x_N).
\end{equation*}
We call charges of the polynomial $W$ the rational numbers $\fq_j := w_j/d$ for all $j$.
The group of diagonal automorphisms of the polynomial $W$ consists of diagonal matrices $\textrm{diag}(\lambda_1,\dotsc,\lambda_N)$ satisfying
\begin{equation*}
W(\lambda_1 x_1,\dotsc,\lambda_N x_N)=W(x_1,\dotsc,x_N) \quad \textrm{for every } (x_1,\dotsc,x_N) \in \CC^N.
\end{equation*}
We write $\textrm{Aut}(W)$ for this group;
it contains the grading element: the matrix
\begin{equation}\label{gradingelement}
\grj := \textrm{diag}(e^{2 \ii \pi \fq_1},\dotsc,e^{2 \ii \pi \fq_N})~, \quad \fq_j := \frac{w_j}{d}
\end{equation}
of order $d$. Any subgroup containing this matrix is called admissible.

We say that a quasi-homogeneous polynomial $W$ is non-degenerate if it has an isolated singularity at the origin and if its weights are uniquely defined.
Then the dimension of the Jacobian ring
\begin{equation*}
\cQ_W := \CC \left[ x_1,\dotsc,x_N \right] / \left( \partial_1 W,\dotsc, \partial_N W \right)
\end{equation*}
is finite over $\CC$ and so is the group $\textrm{Aut}(W)$.
As a stack, the zero locus $\left\lbrace W=0 \right\rbrace$ of a non-degenerate polynomial $W$ is a smooth hypersurface within the weighted projective space $\PP(w_1,\dotsc,w_N)$. By the adjunction formula, its canonical bundle vanishes when $d=w_1+\dotsb+w_n$.
Under this condition we refer to $W$ as a Calabi--Yau polynomial.

\begin{rem}\label{SL}
Under the Calabi--Yau condition, the group $\mathrm{SL}(W)$ of diagonal automorphisms with determinant $1$ is admissible and we set $\SL(W):=\mathrm{SL}(W)/\langle \grj \rangle$.
\end{rem}

\begin{dfn}
A Landau--Ginzburg (LG) orbifold is a pair $(W,G)$ with $W$ a non-degenerate (quasi-homogeneous) polynomial and $G$ an admissible group.
We regard $(W,G)$ as a morphism $W \colon \left[ \CC^N/G \right] \rightarrow \CC$ where $\left[ \CC^N/G \right]$ is a quotient stack.
\end{dfn}

Consider a non-degenerate polynomial
\begin{equation*}
W(x_1,\dotsc,x_N) = c_1 x_1^{m_{1,1}} \dotsm x_N^{m_{1,N}} + \dotsb + c_R x_1^{m_{R,1}} \dotsm x_N^{m_{R,N}},
\end{equation*}
and write $E_W:=(m_{k,j})$ for the matrix of exponents.
For $R=N$, this matrix is invertible (the weights are uniquely defined), hence we have the following definition.

\begin{dfn}[Berglund--H\"ubsch, \cite{Hubsch}]\label{invertiblematrix}
An invertible polynomial is a non-degenerate (quasi-homogeneous) polynomial with as many variables as monomials.
\end{dfn}

Up to a change of coordinates of $\CC^N$, we may assume $c_1 = \dotsb = c_N = 1$, so that an invertible polynomial is determined by the matrix of exponents.
According to Kreuzer--Skarke \cite{Kreuzer},
every invertible polynomial is a Thom--Sebastiani (TS) sum of invertible polynomials, with disjoint sets of variables, of the following three types
\begin{equation}\label{ThomSebastiani}
\begin{array}{lll}
\textrm{Fermat:} & \qquad x^{a+1} & \\
\textrm{chain of length } c: & \qquad x_1^{a_1}x_2+\dotsb+x_{c-1}^{a_{c-1}} x_c+x_c^{a_c+1} & (c \geq 2), \\
\textrm{loop of length } l: & \qquad x_1^{a_1}x_2+\dotsb+x_{l-1}^{a_{l-1}} x_l+x_l^{a_l}x_1 & (l \geq 2). \\
\end{array}
\end{equation}
Let us point out to the reader our non-standard choice for the exponent of Fermat polynomial and for the last exponent of chain polynomials.
This choice will become clear  in the next paragraph.
Once for all, we assume 
\begin{equation}\label{restrictions}
\textrm{every diagonal entry of } E_W \textrm{ is greater or equal to } 2.
\end{equation}
This is a slight restriction, especially for Calabi--Yau polynomials, where only the polynomials $xy+y^k$ are excluded.

We attach to the invertible polynomial $W$ an oriented graph $\Gamma_W$ (possibly containing loops, i.e.~oriented edges starting and ending at the same vertex), whose vertices $v_1,\dotsc,v_N$, in one-to-one correspondence with the variables, are decorated by a positive integer via $f_W \colon v_j \mapsto a_j$ as follows.
Given an index $j$, there is a unique index $t(j)$, possibly equal to $j$, such that
\begin{equation*}
x_j^{a_j}x_{t(j)} \textrm{ is a monomial of } W;
\end{equation*}
we draw an arrow from $v_j$ to $v_{t(j)}$, set $f_W(v_j) := a_j$ and say that $v_{t(j)}$ follows $v_j$.
There is a bijection between connected components and the terms of the TS sum.

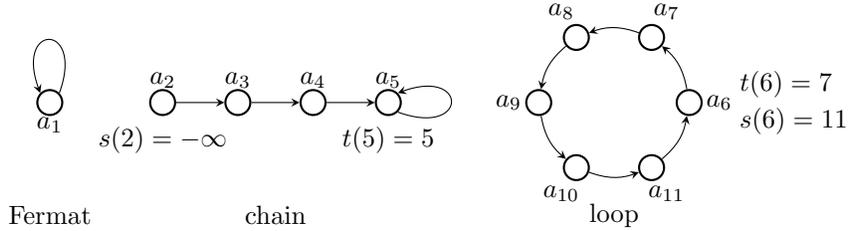
\begin{center}
\begin{tikzpicture}[scale=0.5]
\tikzstyle{sommet}=[circle,draw,thick]
\draw (0:2) node[sommet](A1){}
 (60:2) node[sommet](A2){}
 (120:2) node[sommet](A3){}
 (180:2) node[sommet](A4){}
 (240:2) node[sommet](A5){}
 (300:2) node[sommet](A6){};

\draw (0:2.8) node{$a_6$}
 (60:2.8) node{$a_7$}
 (120:2.8) node{$a_8$}
 (180:2.8) node{$a_9$}
 (240:2.8) node{$a_{10}$}
 (300:2.8) node{$a_{11}$};

\draw[->,>=stealth] (A1) to[bend right=20] (A2);
\draw[->,>=stealth] (A2) to[bend right=20] (A3);
\draw[->,>=stealth] (A3) to[bend right=20] (A4);
\draw[->,>=stealth] (A4) to[bend right=20] (A5);
\draw[->,>=stealth] (A5) to[bend right=20] (A6);
\draw[->,>=stealth] (A6) to[bend right=20] (A1);

\draw (-6:4.8) node{$s(6)=11$}
 (6:4.6) node{$t(6)=7$};

\draw (0,-3) node{loop};

\draw (-12,0) node[sommet](A1){}
 (-10,0) node[sommet](A2){}
 (-8,0) node[sommet](A3){}
 (-6,0) node[sommet](A4){};

\draw (-12,0.6) node{$a_2$}
 (-10,0.6) node{$a_3$}
 (-8,0.6) node{$a_4$}
 (-6,0.6) node{$a_5$};

\draw[->,>=stealth] (A1) to[bend right=0] (A2);
\draw[->,>=stealth] (A2) to[bend right=0] (A3);
\draw[->,>=stealth] (A3) to[bend right=0] (A4);
\draw[->,>=stealth](A4.south east)..controls +(-20:2cm) and +(20:2cm)..(A4.north east);

\draw (-12,-1) node{$s(2)=-\infty$}
 (-6,-1) node{$t(5)=5$};

\draw (-9,-3) node{chain};

\draw (-15,-0.6) node{$a_1$};
\draw (-15,0) node[sommet](A){};
\draw[->,>=stealth](A.north east)..controls +(70:2cm) and +(110:2cm)..(A.north west);

\draw (-15,-3) node{Fermat};
\end{tikzpicture}
\end{center}

\noindent
As a consequence, it is convenient to write any invertible polynomial as
\begin{equation}\label{invertchain}
W = x_1^{a_1} x_{t(1)} +\dotsb+ x_N^{a_N} x_{t(N)}
\end{equation}
and to define a sort of going-back function
\begin{equation}\label{previousnext}
\begin{array}{lccl}
s \colon & \left\lbrace 1,\dotsc,N \right\rbrace & \rightarrow & \left\lbrace -\infty, 1,\dotsc,N \right\rbrace \\
         & k & \mapsto & \left\lbrace \begin{array}{lcl}
         j & \qquad    \qquad   &\textrm{ if } t(j)=k \textrm{ and } j \neq k, \\
         -\infty & \qquad \qquad &\textrm{ otherwise.}
         \end{array}\right. 
\end{array}
\end{equation}


\subsection{State space}\label{sectionstatespace}
We focus on the LG orbifold $(W,\mathrm{Aut}(W))$ where $W$ is an invertible polynomial, and we give a basis of the state space.
This case embodies the relevant information for the LG mirror symmetry Theorem \ref{thmfinal} for Calabi--Yau hypersurfaces of chain-type.
In genus zero, it also determines the $\mathrm{Aut}(W)$-invariant part of the LG orbifold $(W,G)$ for an arbitrary admissible group $G$.

Following the established FJRW terminology, for any $\gamma \in \mathrm{Aut}(W)$ there is a dichotomy between broad and narrow variables (or vertices), where the set of broad variables is
\begin{equation}\label{bgamma}
\bB_\gamma = \left\lbrace x_j \left| \right. \gamma_j =1 \right\rbrace.
\end{equation}
In the graph $\Gamma_W$, any broad vertex is followed by a broad vertex.
The restriction $W_\gamma$ of the polynomial $W$ to the invariant space $(\aA^N)^\gamma$ under the action of $\gamma$ yields
\begin{equation*}
\st_\gamma:=(\cQ_{W_\gamma} \otimes d\underline{x}_\gamma)^{\mathrm{Aut}(W)} \qquad \textrm{with } d\underline{x}_\gamma:=\bigwedge_{x_j \in \bB_\gamma} dx_j.
\end{equation*}
This is the invariant part of $\cQ_{W_\gamma} \otimes d\underline{x}_\gamma$ under the induced action of $\mathrm{Aut}(W)$.

\begin{dfn}
The A-state space for the LG orbifold $(W,\mathrm{Aut}(W))$ is the vector space
\begin{equation*}
\st = \bigoplus_{\gamma \in \mathrm{Aut}(W)} \st_{\gamma}
\end{equation*}
equipped with a bidegree and a natural non-degenerate symmetric bilinear pairing, an orbifolded version of the residue pairing.
We refer to \cite[Equation (4)]{LG/CY} or \cite[Equation (5.12)]{Polish1} for further details, and to \eqref{gradingst} and \eqref{pairingcomput} for expressions with respect to a chosen basis.
\end{dfn}

We now show how certain decorations of $(\Gamma_W,\bB_\gamma)$ can be used as a bookkeeping device classifying vectors of the A-state space.
A decoration of the graph $\Gamma_W$ is simply a subset $\rR_\gamma$ of $\bB_\gamma$. 
On the graph, we represent the variables contained in $\rR_\gamma$ by a crossed vertex, see the figure below.
To fix ideas, a variable (or a vertex) can be
\begin{enumerate}
\item[(a)] narrow,
\item[(b)] broad and crossed, or
\item[(b')] broad and uncrossed.
\end{enumerate}
We lighten terminology by omitting ``broad'' and writing ``crossed'' and ``uncrossed''.

\begin{dfn}
A decoration $\rR_\gamma$ is admissible if every uncrossed vertex is followed by a crossed vertex and every crossed vertex is followed by itself or by an uncrossed vertex (every vertex with a loop is crossed).
A decoration $\rR_\gamma$ is balanced if in each connected component of the graph there are as many uncrossed vertices as crossed vertices.
\end{dfn}

For an admissible decoration, one alternates between uncrossed and crossed vertices, and if the decoration is balanced, the number of broad vertices in each connected component is even.

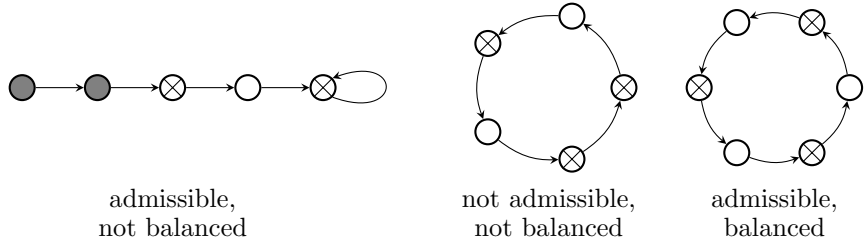
\begin{center}
\begin{tikzpicture}[scale=0.5]
\tikzstyle{sommet}=[circle,draw,thick]
\draw (6,0) ++(0:2) node[sommet](A1){}

(6,0) ++ (60:2) node[sommet](A2){}
(6,0) ++ (60:2) ++(-0.25,0.25)--++(0.5,-0.5)
(6,0) ++ (60:2) ++(-0.25,-0.25)--++(0.5,0.5) 

(6,0) ++ (120:2) node[sommet](A3){}
 
(6,0) ++ (180:2) node[sommet](A4){}
(6,0) ++ (180:2) ++(-0.25,0.25)--++(0.5,-0.5)
(6,0) ++ (180:2) ++(-0.25,-0.25)--++(0.5,0.5) 
 
(6,0) ++ (240:2) node[sommet](A5){}
 
(6,0) ++ (300:2) node[sommet](A6){}
(6,0) ++ (300:2) ++(-0.25,0.25)--++(0.5,-0.5)
(6,0) ++ (300:2) ++(-0.25,-0.25)--++(0.5,0.5);

\draw[->,>=stealth] (A1) to[bend right=20] (A2);
\draw[->,>=stealth] (A2) to[bend right=20] (A3);
\draw[->,>=stealth] (A3) to[bend right=20] (A4);
\draw[->,>=stealth] (A4) to[bend right=20] (A5);
\draw[->,>=stealth] (A5) to[bend right=20] (A6);
\draw[->,>=stealth] (A6) to[bend right=20] (A1);

\draw (0,-3) node{not admissible,};
\draw (0,-3.7) node{not balanced};

\draw (0:2) node[sommet](A1){}
(0:2) ++(-0.25,0.25)--++(0.5,-0.5)
(0:2) ++(-0.25,-0.25)--++(0.5,0.5)
 
(72:2) node[sommet](A2){}

(144:2) node[sommet](A3){}
(144:2) ++(-0.25,0.25)--++(0.5,-0.5)
(144:2) ++(-0.25,-0.25)--++(0.5,0.5)
 
(216:2) node[sommet](A4){}

(288:2) node[sommet](A5){}
(288:2) ++(-0.25,0.25)--++(0.5,-0.5)
(288:2) ++(-0.25,-0.25)--++(0.5,0.5);

\draw[->,>=stealth] (A1) to[bend right=20] (A2);
\draw[->,>=stealth] (A2) to[bend right=20] (A3);
\draw[->,>=stealth] (A3) to[bend right=20] (A4);
\draw[->,>=stealth] (A4) to[bend right=20] (A5);
\draw[->,>=stealth] (A5) to[bend right=20] (A1);

\draw (6,-3) node{admissible,};
\draw (6,-3.7) node{balanced};

\draw (-14,0) node[sommet,fill=gray](A0){}
 (-12,0) node[sommet,fill=gray](A1){}
 (-10,0) node[sommet](A2){}
 ++(-0.25,0.25)--++(0.5,-0.5)
 (-10,0) ++(-0.25,-0.25)--++(0.5,0.5)
 (-8,0) node[sommet](A3){}
 (-6,0) node[sommet](A4){}
 ++(-0.25,0.25)--++(0.5,-0.5)
 (-6,0) ++(-0.25,-0.25)--++(0.5,0.5);

\draw[->,>=stealth] (A0) to[bend right=0] (A1);
\draw[->,>=stealth] (A1) to[bend right=0] (A2);
\draw[->,>=stealth] (A2) to[bend right=0] (A3);
\draw[->,>=stealth] (A3) to[bend right=0] (A4);
\draw[->,>=stealth](A4.south east)..controls +(-20:2cm) and +(20:2cm)..(A4.north east);

\draw (-10,-3) node{admissible,};
\draw (-10,-3.7) node{not balanced};
\end{tikzpicture}
\end{center}

\noindent
To any admissible and balanced decoration $\rR_\gamma$, we associate the element
\begin{equation}\label{basisH}
\begin{split}
e(\rR_\gamma) := & \prod_{k=1}^m \biggl( \prod_{x_j \in (\bB_\gamma \backslash \rR_\gamma)\cap L_k} a_j x_j^{a_j-1} - \prod_{x_j \in \rR_\gamma \cap L_k} -x_j^{a_j-1} \biggr) \cdot \\
                 & \prod_{k=1}^{m'} \biggl( \prod_{x_j \in (\bB_\gamma \backslash \rR_\gamma)\cap C_k} a_j x_j^{a_j-1} \biggr) \cdot \bigwedge_{x_j \in \bB_\gamma} d x_j \\
\end{split}
\end{equation}
of the space $\st_\gamma$, with $a_j:=f_W(v_j)$.
In this formula, the sets of vertices $L_1,\dotsc,L_m$ correspond to the loop-type (connected) components of the graph $\Gamma_W$, the sets $C_1,\dotsc,C_{m'}$ correspond to the chain-type components and the order of the variables in the wedge product is taken in the direction of the arrows of the graph, always starting with an uncrossed vertex.
Notice the absence of Fermat components in \eqref{basisH}, as the vertex of a Fermat component is narrow for every admissible and balanced decorations.
A straightforward argument using the anti-symmetric property of the wedge product shows that \eqref{basisH} is well-defined (it does not matter which uncrossed vertex, nor which connected component, one starts from).

By \cite[Lemma 1.7]{Krawitz}, the set of all the elements $e(\rR_\gamma)$, with a diagonal automorphism $\gamma$ and an admissible and balanced decoration $\rR_\gamma$, forms a basis of the state space,
\begin{equation}\label{basisst}
\st = \bigoplus_{\gamma \in \mathrm{Aut}(W)} \bigoplus_{\substack{\rR_\gamma \subset \bB_\gamma \\
\textrm{admissible, balanced}}} \CC \cdot e(\rR_\gamma).
\end{equation}
By convention, we take $e(\rR_\gamma)=0$ for every non-balanced and admissible decoration; we never consider non-admissible decorations.

\begin{rem}
In \cite[Lemma 1.7]{Krawitz}, the basis is given by the elements
\begin{equation*}
\prod_{x_j \in \bB_\gamma \backslash \rR_\gamma} x_j^{a_j-1} \cdot \bigwedge_{x_j \in \bB_\gamma} d x_j,
\end{equation*}
with a diagonal automorphism $\gamma$ and an admissible and balanced decoration $\rR_\gamma$.
As for every loop-type components $L_1,\dotsc,L_m$ we have
\begin{equation*}
\prod_{x_j \in L_k} a_j \neq 1,
\end{equation*}
then the set of all the elements $e(\rR_\gamma)$ forms a basis of the state space.
The reason why we prefer the basis given by \eqref{basisH} will become clear in Sect.~\ref{2.3}, where we construct some matrix factorizations with Chern characters equal to \eqref{basisH} (see \eqref{ChernKKK}).
Another reason is that the matrix of the bilinear pairing of $\st$ is easy to compute in the basis given by \eqref{basisH}.
\end{rem}

The grading of the state space is given by
\begin{equation}\label{gradingst}
\deg (e(\rR_\gamma)) = \textrm{card} \left\lbrace 1 \leq j \leq N \left| \right. \gamma_j=1 \right\rbrace  + 2 \sum_{j=1}^N (\Gamma_j-\fq_j),
\end{equation}
where $\Gamma_j$ is determined by $\gamma_j = \exp (2 \ci \pi \Gamma_j)$, $\Gamma_j \in \left[ 0,1\right[$.
The bilinear pairing is given by
\begin{equation}\label{pairingcomput}
\bigl( e(\rR_\gamma) , e(\rR'_{\gamma'}) \bigr) = \left\lbrace \begin{array}{cl}
\displaystyle \prod_{x_j \in \bB_\gamma \backslash (\rR_\gamma \cup \rR'_{\gamma'})} (-a_j) & \textrm{if } \gamma'=\gamma^{-1}, \\
0 & \textrm{otherwise,}
\end{array}\right. 
\end{equation}
with $\rR'_{\gamma'}$ an admissible and balanced decoration for $\gamma'$ and $\bB_\gamma=\bB_{\gamma^{-1}}$.

\begin{rem}
By \cite[Lemma 6.1.1]{Polish1}, a computation of the three-point Polishchuk and Vaintrob's correlators yields the bilinear pairing, so that we will deduce equation \eqref{pairingcomput} from Theorem \ref{main}.
It seems difficult to obtain this explicit formula directly from the definition via the residue pairing \cite[Equation (4)]{LG/CY} or via the canonical pairing on matrix factorizations \cite[Equation (2.24)]{Polish1}.
\end{rem}



\subsection{Moduli space}\label{sectionmodulispace}
A genus-$g$ orbifold (or twisted) curve $\cC$ with marked points is a connected, proper, and one-dimensional Deligne--Mumford stack whose coarse space $C$ is a genus-$g$ nodal curve, and such that the morphism $\rho \colon \cC \rightarrow C$ is an isomorphism away from the nodes and the marked points.
Any marked point or node can have a non-trivial stabilizer equal to a finite cyclic group.
We focus on smoothable orbifold curves: orbifold curves whose local picture at the node is $\left[ \left\lbrace xy=0 \right\rbrace / \mu_r \right]$ with
\begin{equation*}
\zeta_r \cdot (x,y) = (\zeta_r x , \zeta_r^{-1} y).
\end{equation*}
Throughout this section we set $r$ to be the smallest integer $l$ such that $\gamma^l=1$ for every element $\gamma \in \textrm{Aut}(W)$,
and further restrict to $r$-stable curves, i.e.~smoothable orbifold curves whose stabilizers (at the nodes and at the markings) have fixed order $r$ and whose coarse nodal pointed curve is stable.

\begin{dfn}
The moduli space $\sS_{g,n}$ classifies all $W$-spin curves
\begin{equation*}
(\cC; \sigma_1,\dotsc,\sigma_n;\cL_1,\dotsc,\cL_N;\phi_1,\dotsc,\phi_N),
\end{equation*}
where $(\cC; \sigma_1,\dotsc,\sigma_n)$ is an $r$-stable genus-$g$ curve, $\cL_1,\dotsc,\cL_N$ are line bundles on the curve $\cC$ and
\begin{equation}\label{spinstruct}
\phi_j \colon \cL_j^{\otimes a_j} \otimes \cL_{t(j)} \longrightarrow \omega_{\cC,\log} := \omega_{\cC} (\sigma_1 + \dotsc + \sigma_n)
\end{equation}
are isomorphisms.
\end{dfn}

\begin{rems*}
The twisted canonical line bundle $\omega_{\cC,\log}$ equals the pull-back via $\rho$ of the line bundle $\omega_C (\sigma_1 + \dotsc + \sigma_n)$ on the coarse curve.
A preliminary less general definition of FJRW moduli objects has been already given in Sect.~\ref{intro4}. It involved the moduli space of the LG orbifold $(W,\mu_d)$, where every line bundle $\cL_j$ comes from the same line bundle $\cL$ satisfying $\cL^d \simeq \omega_{\cC,\log}$; this moduli space is naturally embedded in $\sS_{g,n}$.
The notion of $W$-spin curves can be generalized to families over a base scheme $S$ and the moduli space $\sS_{g,n}$ is therefore a smooth and proper Deligne--Mumford stack. It is finite over the moduli space $\overline{\cM}_{g,n}$ of stable curves.
Moreover, this definition of $\sS_{g,n}$, owing to \cites{FJRW,Vistoli2}, is compatible with the definition via $\Gamma$-spin curves of \cite[Proposition 3.2.2]{Polish1}.
\end{rems*}

Locally at the marked point $\sigma_i$ of a $W$-spin curve, the group of $r$th-roots of unity acts on the line bundle $\cL_j$ by 
\begin{equation}\label{multiplicities}
\zeta_r \cdot (x,\xi) = (\zeta_r x, \zeta_r^{m_j(i)} \xi) ~, \quad \textrm{with } m_j(i) \in \left\lbrace 0 , \dotsc, r-1 \right\rbrace
\end{equation}
called the multiplicity of the line bundle $\cL_j$ at the marked point $\sigma_i$.
We write $\gamma_j(i):= \zeta_r^{m_j(i)}$ and $\gamma(i):=(\gamma_1(i),\dotsc,\gamma_N(i)) \in (\mathbb{U}(1))^N$ to define the type
\begin{equation*}
\overline{\gamma}:=(\gamma(1),\dotsc,\gamma(n)) \in (\textrm{Aut}(W))^n
\end{equation*}
of a $W$-spin curve, yielding a decomposition
\begin{equation*}
\sS_{g,n} = \bigsqcup_{\overline{\gamma} \in (\textrm{Aut}(W))^n} \sS_{g,n}(\gamma(1),\dotsc,\gamma(n)),
\end{equation*}
where $\sS_{g,n}(\overline{\gamma})$ is an empty component when the selection rule
\begin{equation}\label{selecrule}
\gamma(1) \dotsm \gamma(n) = \grj^{2g-2+n}
\end{equation}
is not satisfied (see \cite[Proposition 2.2.8]{FJRW}).
In each component lies a homology cycle whose Poincar\'e dual is called the FJRW virtual class
\begin{equation*}
\cvirFJRW \colon \st^{\otimes n} \rightarrow H^*(\sS_{g,n}).
\end{equation*}
The cohomological degree of $\cvirFJRW(e(\rR_{\gamma(1)}), \dotsc, e(\rR_{\gamma(n)}))_{g,n}$ is
\begin{equation}\label{cohdegvir}
\deg (e(\rR_{\gamma(1)})) + \dotsb + \deg (e(\rR_{\gamma(n)})) + 2 \hat{c}_W \cdot (g-1),
\end{equation}
where $\hat{c}_W := \sum_{j=1}^N (1-2 \fq_j)$ is the central charge of the polynomial $W$.

Besides, Polishchuk and Vaintrob constructed another virtual class
\begin{equation*}
\cvirPV \colon \st^{\otimes n} \rightarrow H^*(\sS_{g,n})
\end{equation*}
with the same cohomological degree.
The relation between $\cvirPV$ and $\cvirFJRW$ is explained in Sect.~\ref{term}.
We do not give a definition for the FJRW virtual class; the interested reader is referred to \cites{FJRW,FJRW2}.
Nevertheless, Sect.~\ref{section2} is entirely devoted to the construction of $\cvirPV$.
As already mentioned in the introduction, in the narrow sector (where $\cvirFJRW=\cvirPV$) and under the genus-zero concavity property, that is
\begin{equation*}
H^0(\cC,\cL_j)=0 \qquad \textrm{for all }  j \textrm{ and all genus-zero $W$-spin curves } \cC,
\end{equation*}
the genus-zero virtual class $\cvirFJRW$ is determined by linearity via
\begin{equation*}
c_\textrm{top}(R^1\pi_* (\cL_1 \oplus \dotsb \oplus \cL_N)) \quad \textrm{(Euler class of a vector bundle),}
\end{equation*}
where $\cL_1, \dotsc,\cL_N$ are the universal line bundles on the universal space of $\sS_{0,n}(\overline{\gamma})$ and where $\pi$ is the projection to $\sS_{0,n}(\overline{\gamma})$.

Each virtual class forms a cohomological field theory, see \cite[Theorem 4.1.8]{FJRW} for $\cvirFJRW$ and \cite[Sect.~5]{Polish1} for $\cvirPV$, and the cohomological field theory for $\cvirFJRW$ is called FJRW theory.
The invariants (also called correlators) are respectively
\begin{equation}\label{invariants}
\langle \tau_{b_1}(u_1) \dotsm \tau_{b_n}(u_n) \rangle^\mathrm{FJRW}_{g,n} := \int_{\overline{\cM}_{g,n}} \psi_1^{b_1} \dotsm \psi_n^{b_n} \frac{r^g}{\mathrm{deg} (\textrm{o})}  \textrm{o}_*\cvirFJRW(u_1,\dotsc,u_n)_{g,n},
\end{equation}
\begin{equation}\label{invariantsPV}
\langle \tau_{b_1}(u_1) \dotsm \tau_{b_n}(u_n) \rangle^\mathrm{PV}_{g,n} := \int_{\overline{\cM}_{g,n}} \psi_1^{b_1} \dotsm \psi_n^{b_n} ~ \frac{r^g}{\mathrm{deg} (\textrm{o})} ~ \textrm{o}_*\cvirPV(u_1,\dotsc,u_n)_{g,n}
\end{equation}
with $\psi_1,\dotsc,\psi_n$ the usual psi-classes in $H^2(\overline{\cM}_{g,n})$, $u_1,\dotsc,u_n \in \st$ and
\begin{equation*}
\textrm{o} \colon \sS_{g,n} \rightarrow \overline{\cM}_{g,n}
\end{equation*}
the morphism forgetting $W$-spin and orbifold structures.
We omit $\tau_{b_i}$ when $b_i=0$.

\section{Virtual class from matrix factorizations}\label{section2}
By focusing on the relevant LG orbifold $(W,\Aut(W))$, we illustrate the algebraic construction of the virtual class from \cite{Polish1}.
In Sect.~\ref{2.2}, we set-up some material for the treatment of broad marked points.
In Sect.~\ref{2.1}, we present matrix factorizations of Koszul type, which are ideally suited to the LG orbifolds.
In Sect.~\ref{polishfundmat}, we introduce Polishchuk--Vaintrob's \MF $\PV$, a universal object over the moduli stack.
In Sect.~\ref{2.3}, we construct Koszul \MFs $\kK(\rR_{\overline{\gamma}})$ which lift every element $e(\rR_{\overline{\gamma}})$ of the state space.
In Sect.~\ref{cvirinvertible}, by coupling  $\PV$ with $\kK(\rR_{\overline{\gamma}})$ we get a two-periodic complex and ultimately a class in cohomology: the virtual class appearing in \eqref{invariantsPV}.
We point out that this two-periodic complex may be regarded as a generalization of Chiodo's complex (see \cite{ChiodoJAG}) yielding the virtual class in the case of Fermat monomial $x^r$, i.e.~$r$-spin curves.

\subsection{Combinatorics of $W$-spin curves}\label{2.2}
Consider a component $\sS_{g,n}(\overline{\gamma})$ of type $\overline{\gamma}=(\gamma(1),\dotsc,\gamma(n)) \in \left( \textrm{Aut}(W) \right)^n$.
We write

~~

\setlength{\unitlength}{1cm}
\begin{picture}(3,2)(-4.5,0)
\thicklines
\put(0,2){\line(1,0){3}}
\put(3,0){\line(0,1){2}}
\put(0,0){\line(0,1){2}}
\put(0,0){\line(1,0){3}}
\put(0,1.5){\line(1,0){3}}
\put(0.5,0){\line(0,1){2}}
\put(0.5,1.5){\line(-1,1){0.5}}

\put(0.09,1.55){$i$}
\put(0.33,1.75){$j$}

\put(0.1,1.1){$\sigma_1$}
\put(0.2,0.5){$\vdots$}
\put(0.1,0.1){$\sigma_n$}

\put(0.6,1.63){$\cL_1$}
\put(1.5,1.73){$\dots$}
\put(2.4,1.63){$\cL_N$}

\put(0.6,1.1){$\gamma_{1}(1)$}
\put(0.6,0.1){$\gamma_{1}(n)$}
\put(2.07,1.1){$\gamma_{N}(1)$}
\put(2.07,0.1){$\gamma_{N}(n)$}
\put(1.55,1.2){$\dots$}
\put(1.55,0.2){$\dots$}
\put(0.8,0.5){$\vdots$}
\put(2.6,0.5){$\vdots$}
\end{picture}
\\
and keep record of the sets
\begin{equation*}
\bB_{\overline{\gamma}} = \left\lbrace (\sigma_i,x_j) \left| \right. \gamma_j(i) = 1 \right\rbrace , ~~
\bB_{\overline{\gamma}_j} = \left\lbrace \sigma_i \left| \right.  (\sigma_i,x_j) \in \bB_{\overline{\gamma}} \right\rbrace , ~~ \bB_{\gamma(i)} = \left\lbrace x_j \left| \right.  (\sigma_i,x_j) \in \bB_{\overline{\gamma}} \right\rbrace.
\end{equation*}
The set $\bB_{\overline{\gamma}}$ is an analog of \eqref{bgamma} with marked points and labels the coordinates $\left\lbrace x_j(i)\right\rbrace_{(\sigma_i,x_j) \in \bB_{\overline{\gamma}}}$ of the affine space
\begin{equation}\label{aAffinespace}
\aA^{\overline{\gamma}}:= \prod_{i=1}^n \left( \aA^N \right)^{\gamma(i)},
\end{equation}
where $\left( \aA^N \right)^{\gamma(i)}$ denotes the invariant part under $\gamma(i)$.
We consider the invertible polynomial $W_{\overline{\gamma}}$ on $\aA^{\overline{\gamma}}$ given by
\begin{equation*}
W_{\overline{\gamma}} := W_{\gamma(1)}(x_1(1),\dotsc,x_N(1)) + \dotsb + W_{\gamma(n)}(x_1(n),\dotsc,x_N(n)),
\end{equation*}
where $W_{\gamma(i)}$ stands for the invariant part of $W$ under the action of $\gamma(i)$.
Observe that we obtain the graph $\Gamma_{W_{\gamma(i)}}$ from $\Gamma_W$ by erasing all narrow vertices for $\gamma(i)$ and all edges starting or ending at a narrow vertex.
Then the graph $\Gamma_{W_{\overline{\gamma}}}$ is the disjoint union of $\Gamma_{W_{\gamma(1)}}, \dotsc, \Gamma_{W_{\gamma(n)}}$.
For any scheme $S$, we set
\begin{equation*}
\cO_S^{\overline{\gamma}}:= \bigoplus_{(\sigma_i,x_j) \in \bB_{\overline{\gamma}}} \cO \quad = \quad \bigoplus_{j=1}^N \cO_S^{\overline{\gamma}_j} \quad \textrm{with canonical basis } \left\lbrace e_j(i) \right\rbrace_{(\sigma_i,x_j) \in \bB_{\overline{\gamma}}} .
\end{equation*}

\subsection{Category of matrix factorizations}\label{2.1}
Let $\bw$ be a function on a stack $X$.
A matrix factorization $E:=(E,\delta_E)$ of potential $\bw$ is a $\ZZ/2$-graded vector bundle $E = E_0 \oplus E_1$ on $X$ together with an endomorphism $\delta_E$ satisfying $E_0 \leftrightarrows E_1$ and $\delta_E \circ \delta_E = \bw \cdot \textrm{id}_E$.
The category $\mathrm{MF}(X,\bw)$ of \MFs of $\bw$ on $X$ is a dg-category.
The tensor product of a \MF of $\bw$ with a \MF of $\bw'$ yields a \MF of $\bw+\bw'$;
the dual of a \MF of $\bw$ is a \MF of $-\bw$.
A \MF of $\bw=0$ is a two-periodic complex and it makes sense to look at its cohomology.


Let $V$ be a vector bundle on a stack $X$ and $\alpha \in H^0(X,V)$, $\beta \in H^0(X,V^\vee)$ be global sections whose pairing equals $\beta(\alpha)=\bw$.

\begin{dfn}
The Koszul \MF $\left\lbrace \alpha , \beta \right\rbrace$ of $\bw$ on $X$ consists of a $\ZZ/2$-graded vector bundle $\left\lbrace \alpha , \beta \right\rbrace_0 \oplus \left\lbrace \alpha , \beta \right\rbrace_1$,
\begin{displaymath}
\begin{array}{lclcl}
\left\lbrace \alpha , \beta \right\rbrace_0 &=& \bigwedge_{\textrm{even}} V &=& \cO_X \oplus \bigwedge^2 V \oplus \dotsb \\
\left\lbrace \alpha , \beta \right\rbrace_1 &=& \bigwedge_{\textrm{odd}} V &=& V \oplus \bigwedge^3 V \oplus \dotsb,
\end{array}
\end{displaymath}
together with the morphism $\delta_{\alpha,\beta} := \iota(\beta) + \alpha \wedge \cdot$, where the notation $\iota(\beta)$ stands for the contraction by $\beta$.
\end{dfn}

These objects behave well under tensor products.
For global sections $\alpha, \alpha', \beta, \beta'$ of vector bundles $V,V',V^\vee,(V')^\vee$, we have
\begin{equation*}
\left\lbrace \alpha , \beta \right\rbrace \otimes \left\lbrace \alpha' , \beta' \right\rbrace = \left\lbrace \alpha \oplus \alpha' , \beta \oplus \beta' \right\rbrace
\end{equation*}
where $\alpha \oplus \alpha' \in H^0(X,V \oplus V')$ and $\beta \oplus \beta' \in H^0(X,(V \oplus V')^{\vee})$.

In \cite{Polish3}, Polishchuk and Vaintrob provide an explicit description of the Hochschild homology of \MFs for the affine space $X = \aA^N$ with an invertible polynomial $W$ and show
\begin{equation*}
HH_*(\mathrm{MF}(\CC^N,W)) \simeq \cQ_W \cdot dx_1 \wedge \dotsm \wedge dx_N.
\end{equation*}
In particular, for any $\overline{\gamma} \in \Aut(W)^n$, we have
\begin{equation}\label{sthoch}
\st_{\gamma(1)} \otimes \dotsm \otimes \st_{\gamma(n)} \subset HH_*(\mathrm{MF}(\aA^{\overline{\gamma}},W_{\overline{\gamma}})).
\end{equation}

\begin{rem}
In \cite{Polish1}, Polishchuk and Vaintrob work in the category of equivariant matrix factorizations and define a more general cohomological field theory than we need here. Their state space is
\begin{equation*}
\bigoplus_{\gamma \in \Aut(W)} HH_*(\mathrm{MF}_{\textrm{Aut}(W)}((\aA^N)^\gamma,W_\gamma)) \simeq \bigoplus_{\gamma \in \Aut(W)} \bigoplus_{\gamma' \in \Aut(W)} \cQ_{(W_\gamma)_{\gamma'}}^{\Aut(W)}.
\end{equation*}
Following Polishchuk and Vaintrob (see \cite[Equation (5.15)]{Polish1}), we consider in this paper the specialization of PV's theory which consists in keeping only
\begin{equation}\label{special}
\gamma'=1
\end{equation}
in the expression of the PV's state space and we observe that this specialization gives the state space $\st$.
\end{rem}

In the way any \MF on $\aA^{\overline{\gamma}}$ is a matrix $U$ with coefficients in $\CC \left[ x_j(i)\right]_{(\sigma_i,x_j) \in \bB_{\overline{\gamma}}}$ of the form
\begin{equation*}
\begin{pmatrix}
0 & U_1\\
U_0 & 0
\end{pmatrix}, \quad U^2=W_{\overline{\gamma}} \cdot \textrm{Id},
\end{equation*}
we can figure out (see for instance \cite[Theorem 3.2.3]{Polish3}) the Chern character
\begin{equation}\label{Chernmf}
\Ch (U) = \mathrm{str} \biggl( \prod_{(\sigma_i,x_j) \in \bB_{\overline{\gamma}}} \frac{\partial U}{\partial x_j(i)} \biggr) \cdot \bigwedge_{(\sigma_i,x_j) \in \bB_{\overline{\gamma}}} d x_j(i) \in \st_{\gamma(1)} \otimes \dotsm \otimes \st_{\gamma(n)}.
\end{equation}
We take the same order for the variables of $\bB_{\overline{\gamma}}$ in the multiplication of matrices (from right to left) and in the wedge product (from left to right). The result is well-defined because super-trace and wedge product are both anti-commutative.
Moreover we notice that $\Ch (U)$ vanishes whenever the cardinal of $\bB_{\overline{\gamma}}$ is odd.


\subsection{Polishchuk--Vaintrob's matrix factorization}\label{polishfundmat}
Take a family $\pi \colon \cC \rightarrow S$ of $W$-spin curves over a base scheme $S$ and consider a resolution of any $R\pi_*(\cL_j)$ by a complex $[A_j \rightarrow B_j]$ of vector bundles.
For any geometric point $s \in S$, observe that
\begin{equation*}
\ker (A_j \rightarrow B_j)_s = H^0(\cC_s,\cL_{j,s}) \quad \textrm{and} \quad \textrm{coker} (A_j \rightarrow B_j)_s = H^1(\cC_s,\cL_{j,s}).
\end{equation*}
Denote by $A$ and $B$ the vector bundles
\begin{equation*}
A:= A_1 \oplus \dotsb \oplus A_N \quad \textrm{and} \quad B:= B_1 \oplus \dotsb \oplus B_N
\end{equation*}
on $S$ and by $X$ the total space of $A$ with projection $p$ to $S$,
\begin{equation*}
X:=\Spec (\mathrm{Sym} A^\vee) \quad \textrm{and} \quad p \colon X \longrightarrow S.
\end{equation*}
In \cite[Sect.~4.2]{Polish1}, Polishchuk and Vaintrob construct a morphism $Z \colon X \rightarrow \aA^{\overline{\gamma}}$ and two sections
\begin{equation*}
\alpha \in H^0(X,p^*B^\vee) \quad \textrm{and} \quad \beta \in H^0(X,p^*B), \qquad \textrm{with} \quad \alpha(\beta) = Z^*W_{\overline{\gamma}}.
\end{equation*}
These sections are sums $\alpha:=\alpha_1+\dotsb+\alpha_N$ and $\beta:=\beta_1 + \dotsb+ \beta_N$, where $\beta_j$ is induced by the differential of 
$[A_j \rightarrow B_j]$. By a slight abuse of notation, we write
\begin{equation}\label{alphabeta2}
\beta_j  \colon  A_j \rightarrow B_j.
\end{equation}
The section $\alpha_j$ is a sum of two morphisms, of which we give a rough idea in \eqref{Serreduality},
\begin{equation}\label{alphabeta}
\begin{array}{lcl}
\alpha'_j & \colon & \cS^{a_j} A_j \rightarrow  B_{t(j)}^\vee, \\
\alpha''_j & \colon & \cS^{a_j-1} A_j \otimes A_{t(j)} \rightarrow B_j^\vee. \\
\end{array}
\end{equation}
This yields Polishchuk--Vaintrob's \MF $\left\lbrace - \alpha , \beta \right\rbrace$ on $X$,
\begin{equation}\label{PVmatfact}
\PV:=\left\lbrace - \alpha , \beta \right\rbrace \in \textrm{MF}(X,-Z^*W_{\overline{\gamma}}).
\end{equation}
\begin{center}
\begin{tikzpicture}
\node (X) at (0.8,1) {$X$};
\node[above] (E) at (1.8,1.6) {$\PV$};
\node (A) at (0,0) {$\aA^{\overline{\gamma}}$};
\node (S) at (1.6,0) {$S$};
\draw[->,>=stealth] (E) to[bend left=10] (X);
\draw[->,] (X) -- (A);
\draw[->] (X) -- (S);
\draw (0.45,0.65) node[left] {$Z$};
\draw (1.15,0.65) node[right] {$p$};
\end{tikzpicture}
\end{center}

The tensor product of $\PV$ with \MFs from the affine space $\aA^{\overline{\gamma}}$ of potential $Z^*W_{\overline{\gamma}}$ produces a two-periodic complex, whose support is included in the zero section $S \hookrightarrow X$ (see \cite[Sect.~4.2, Step 4; Proposition 1.4.2]{Polish1}); we apply the push-forward functor and obtain
\begin{equation}\label{foncteur}
\begin{array}{lcccc}
\Phi \colon & \textrm{MF}(\aA^{\overline{\gamma}},W_{\overline{\gamma}}) & \longrightarrow & \textrm{MF}(S,0) \\
            &                               U                            & \longmapsto     & p_*(Z^*(U) \otimes \PV).    \\
\end{array}
\end{equation}
\begin{rem}
The functor $\Phi$ matches the functor obtained from \cite[Equation (5.5)]{Polish1} once we forget the equivariance.
Indeed, it is a direct application of the projection formula \cite[Proposition 1.5.5]{Polish1}.
As a consequence, the morphism $\Phi_*$ induced on the Hochschild homology coincides with the specialization \eqref{special} of the morphism defined in \cite[Equation (5.8)]{Polish1}.
In Sect.~\ref{cvirinvertible}, we will use the Chern character $\Ch(\Phi(U))$ of the two-periodic complex $\Phi(U)$ to define the virtual class.
\end{rem}

For sake of clarity, let us illustrate the construction of the morphisms $Z$ and $\alpha$ over $S=\Spec (\CC)$.
Consider a $W$-spin curve of type $\overline{\gamma}$.
In the broad case, i.e.~for $\gamma_j(i)=1$, we choose an isomorphism
\begin{equation}\label{rigid}
H^0(\sigma_i,\cL_j |_{\sigma_i}) \simeq \cO_S = \CC,
\end{equation}
while in the narrow case, the space of sections $H^0(\sigma_i,\cL_j\left| \right._{\sigma_i})$ vanishes.
Then we denote by $Z_j(i)$ the evaluation of a section of $\cL_j$ at a marked point $\sigma_i$, with $(\sigma_i,x_j) \in \bB_{\overline{\gamma}}$; this is a morphism
\begin{equation}\label{morphismZ}
Z_j(i) \colon H^0(\cC,\cL_j) \rightarrow \CC
\end{equation}
and we assemble the morphisms $Z_j(i)$ for every $(\sigma_i,x_j) \in \bB_{\overline{\gamma}}$ to get
\begin{equation*}
Z \colon H^0(\cC,\cL_1) \oplus \dotsb \oplus H^0(\cC,\cL_N) \rightarrow \aA^{\overline{\gamma}}.
\end{equation*}
\begin{rem}
The construction of the morphism $Z$ depends on the choice of an isomorphism \eqref{rigid}, called a rigidification of the $W$-spin curve.
The approach of both Fan--Jarvis--Ruan and Polishchuk--Vaintrob is to work on the moduli space parametrizing $W$-spin curves alongside with such rigidification; this amounts to work over a finite \'etale cover of the moduli space considered here. However, the final computation of the virtual class provided in Theorem \ref{main} does not depend on this choice.
\end{rem}

A crucial and yet elementary ingredient of our explicit realization of the virtual class is the twisting of the line bundles by the marked points, that is
\begin{equation*}
\cL'_j:=\cL_j(-\sigma_1-\dotsb-\sigma_n).
\end{equation*}
Using the isomorphism \eqref{spinstruct}, we obtain
\begin{equation*}
{\cL'_j}^{\otimes a_j} \otimes \cL'_{t(j)} \hookrightarrow \omega_{\cC} \qquad (\textrm{orbifold canonical bundle of } \cC).
\end{equation*}
By (orbifold) Serre duality, we get morphisms
\begin{equation}\label{Serreduality}
\mathrm{Sym}^{a_j} H^0(\cC,\cL'_j) \rightarrow H^0(\cC,{\cL'_j}^{\otimes a_j}) \hookrightarrow H^0(\cC,\omega_{\cC} \otimes {\cL'_{t(j)}}^{\vee}) \simeq H^1(\cC,\cL'_{t(j)})^\vee,
\end{equation}
and similarly
\begin{equation*}
\mathrm{Sym}^{a_j-1}(H^0(\cC,\cL'_j)) \otimes H^0(\cC,\cL'_{t(j)}) \rightarrow H^1(\cC,\cL'_j)^\vee.
\end{equation*}
These morphisms are related to $\alpha'_j$ and $\alpha''_j$ of \eqref{alphabeta}.

In \cite[Sect.~4.2]{Polish1}, Polishchuk and Vaintrob lift all these constructions over a base scheme $S$, with appropriate vector bundles $A$ and $B$.
In particular, for each monomial we have a commutative diagram
\begin{equation*}
\xymatrix{
    \cS^{a_j} A_j \otimes A_{t(j)} \ar[r] \ar[d] & (\cS^{a_j} A_j \otimes B_{t(j)}) \oplus (\cS^{a_j-1} A_j \otimes A_{t(j)} \otimes B_j) \ar[d] \\
    \cO_S^{\overline{\gamma}_{j}} \ar[r] & \cO_S,
  }
\end{equation*}
the morphisms are defined as follows.
The vertical arrow on the left is induced by $Z$ and by the algebra structure on the sheaf $\cO_S^{\overline{\gamma}_{j}}$, the vertical arrow on the right is induced by $\alpha_j$, the first horizontal arrow is induced by $\beta_j$ and by $\beta_{t(j)}$ and the second horizontal arrow is the trace.

\subsection{Lifting up the state space}\label{2.3}
Let $\overline{\gamma} \in \textrm{Aut}(W)^n$ and fix admissible decorations $\rR_{\gamma(1)}, \dotsc, \rR_{\gamma(n)}$, one for each marked point.
We assemble them in an admissible decoration of the graph $\Gamma_{W_{\overline{\gamma}}}$
\begin{equation*}
\rR_{\overline{\gamma}} := \left\lbrace (\sigma_i,x_j) \in \bB_{\overline{\gamma}} \left| \right. x_j \in \rR_{\gamma(i)} \right\rbrace \subset \bB_{\overline{\gamma}},
\end{equation*}
and take
\begin{equation*}
e(\rR_{\overline{\gamma}}) := e(\rR_{\gamma(1)}) \otimes \dotsm \otimes e(\rR_{\gamma(n)}) \in \st_{\gamma(1)} \otimes \dotsm \otimes \st_{\gamma(n)} .
\end{equation*}
Consider on $\aA^{\overline{\gamma}}$ the free sheaf  
\begin{equation*}
\bigoplus_{j=1}^N \cO^{\rR_j} =:  \cO^{\rR_{\overline{\gamma}}} \subset \cO^{\overline{\gamma}} \quad \textrm{ with basis } \left\lbrace e_j(i) \right\rbrace_{(\sigma_i,x_j) \in \rR_{\overline{\gamma}}}.
\end{equation*}
For any $(\sigma_i,x_j) \in \rR_{\overline{\gamma}}$, form the sections
\begin{equation}\label{sections}
\begin{array}{rcll}
\mathfrak{a}_j(i) & := & \left( x_{s(i,j)}(i)^{a_{s(i,j)}} + x_j(i)^{a_j-1} x_{t(i,j)}(i) \right) \cdot e_j(i) & \in \cO^{\rR_j} \subset \cO^{\rR_{\overline{\gamma}}}, \\[0.2cm]
\mathfrak{b}_j(i) & := &    x_j(i) \cdot e_j(i)^\vee & \in {\cO^{\rR_j}}^{\vee} \subset {\cO^{\rR_{\overline{\gamma}}}}^\vee.
\end{array}
\end{equation}
Here, the notations $s(i,j)$ and $t(i,j)$ are applied to the vertices of $\Gamma_{W_{\overline{\gamma}}}$ (see \eqref{previousnext} for the definition) and we take the convention $x_{-\infty}(i)=0$ and $a_{-\infty}=1$.
Since the decoration $\rR_{\overline{\gamma}}$ is admissible, the Koszul \MF given by
\begin{equation*}
\mathfrak{a}_{\rR_{\overline{\gamma}}} := \sum_{(\sigma_i,x_j) \in \rR_{\overline{\gamma}}} \mathfrak{a}_j(i) \qquad \textrm{and} \qquad \mathfrak{b}_{\rR_{\overline{\gamma}}} := \sum_{(\sigma_i,x_j) \in \rR_{\overline{\gamma}}} \mathfrak{b}_j(i)
\end{equation*}
is in $\textrm{MF}(\mathbb{A}^{\overline{\gamma}},W_{\overline{\gamma}})$; we denote it by $\kK(\rR_{\overline{\gamma}})$ and compute its Chern character via \eqref{Chernmf}.

\begin{lem}
For any $\overline{\gamma} \in \textrm{Aut}(W)^n$ and any admissible decoration $\rR_{\overline{\gamma}}$ of the graph $\Gamma_{W_{\overline{\gamma}}}$, we have
\begin{equation}\label{ChernKKK}
\Ch (\kK(\rR_{\overline{\gamma}})) = e(\rR_{\overline{\gamma}}) \in \st_{\gamma(1)} \otimes \dotsm \otimes \st_{\gamma(n)}.
\end{equation}
In particular, the Chern character vanishes when the decoration is not balanced.
\end{lem}

\begin{proof}
Equation \eqref{ChernKKK} is a direct computation from \eqref{Chernmf}, so we only give the main steps.
First, we derivate the differential of $\kK(\rR_{\overline{\gamma}})$ along $x_j(i)$; and denote it
\begin{equation*}
\partial_j(i) := \frac{\partial (\mathfrak{a}_{\rR_{\overline{\gamma}}} \wedge \cdot + \mathfrak{b}_{\rR_{\overline{\gamma}}}(\cdot))}{\partial x_j(i)}.
\end{equation*}
There are two cases:
if $(\sigma_i,x_j) \in \rR_{\overline{\gamma}}$, we get
\begin{equation*}
\partial_j(i) = \left\lbrace \begin{array}{rl}
(a_j-1) ~ x_j(i)^{a_j-2} ~ x_{t(i,j)}(i) ~ e_j(i) \wedge \cdot + e_j(i)^\vee(\cdot) & \textrm{if } x_{t(i,j)}(i) \neq x_j(i), \\[0.2cm]
a_j ~ x_j(i)^{a_j-2} ~ x_{t(i,j)}(i) ~ e_j(i) \wedge \cdot + e_j(i)^\vee(\cdot) & \textrm{otherwise,} \\
\end{array}\right. 
\end{equation*}
and if $(\sigma_i,x_j) \in \bB_{\overline{\gamma}}-\rR_{\overline{\gamma}}$, we get
\begin{equation*}
\partial_j(i) = x_{s(i,j)}(i)^{a_{s(i,j)}-1} ~ e_{s(i,j)}(i) \wedge \cdot  + a_j ~ x_j(i)^{a_j-1} ~ e_{t(i,j)}(i) \wedge \cdot
\end{equation*}
To get the supertrace of \eqref{Chernmf}, we have to compute
\begin{equation}\label{tocompute}
{e_K}^\vee \circ \biggl( \bigcirc_{(\sigma_i,x_j) \in \bB_{\overline{\gamma}}} \partial_j(i) \biggr) (e_K),
\end{equation}
for every element
\begin{equation*}
e_K:= \bigwedge_{(\sigma_i,x_j) \in K} e_j(i)~, \quad K \subset \rR_{\overline{\gamma}}.
\end{equation*}
Observe that, for any $(\sigma_i,x_j) \in \rR_{\overline{\gamma}}$, the only non-zero contribution from $\partial_j(i)$ to \eqref{tocompute} is given by the contraction $e_j(i)^\vee$ (if we take the wedge product by $e_j(i)$, then we cannot contract anymore via $e_j(i)^\vee$ and we obtain zero by ${e_K}^\vee$).
Thus, there are two possibilities.
If $(\sigma_i,x_j) \in K$, then we apply $\partial_{t(i,j)}(i) \circ \partial_j(i)$ and the only non-zero contribution to \eqref{tocompute} is
\begin{equation*}
x_j(i)^{a_j-1} e_j(i) \wedge \bigl( e_j(i)^\vee(e_K) \bigr) = x_j(i)^{a_j-1} ~ e_K.
\end{equation*}
If $(\sigma_i,x_j) \notin K$, then we apply $\partial_j(i) \circ \partial_{s(i,j)}(i)$ and the non-zero contribution is
\begin{equation*}
e_j(i)^\vee \bigl( a_{s(i,j)} ~ x_{s(i,j)}(i)^{a_{s(i,j)}-1} ~ e_j(i) \wedge e_K \bigr) = a_{s(i,j)} ~ x_{s(i,j)}(i)^{a_{s(i,j)}-1} ~ e_K.
\end{equation*}
Finally, the only non-zero values for \eqref{tocompute} are given by the elements $e_K$ where $K$ is the empty set or the set of all the crossed vertices of some loop-type components of the graph.
Then it is straightforward to recover the formula \eqref{basisH}.
\end{proof}


\subsection{Two-periodic complex and virtual class}\label{cvirinvertible}
The tensor product of the \MF $\PV$ with the pull-back $Z^*\kK(\rR_{\overline{\gamma}})$ yields a two-periodic complex on $X$.
More precisely, we write
\begin{equation*}
\begin{array}{lcl}
\alpha'_{s(j)}+\alpha''_j & \colon & \cS^{a_{s(j)}} A_{s(j)} \oplus (\cS^{a_j-1} A_j \otimes A_{t(j)}) \rightarrow  B_j^\vee, \\
\beta_j  & \colon &  A_j \rightarrow B_j, \\
\end{array}
\end{equation*}
with the convention $A_{-\infty}=0$ and $a_{-\infty}=1$,
we add these morphisms to
\begin{eqnarray*}
Z^*\mathfrak{a}_j & \colon & \textrm{Sym}^{a_{s(j)}} A_{s(j)} \oplus (\textrm{Sym}^{a_j-1} A_j \otimes A_{t(j)}) \rightarrow \cO_S^{\rR_j}, \\
Z^*\mathfrak{b}_j & \colon & A_j \rightarrow (\cO_S^{\rR_j})^\vee,
\end{eqnarray*}
and we obtain
\begin{equation}\label{widealpha}
\begin{array}{lcl}
\widetilde{\alpha}_j & \colon & \cS^{a_{s(j)}} A_{s(j)} \oplus (\cS^{a_j-1} A_j \otimes A_{t(j)}) \rightarrow \widetilde{B}_j^\vee, \\
\widetilde{\beta}_j & \colon & A_j \rightarrow \widetilde{B}_j,
\end{array}
\end{equation}
where $\widetilde{B}_j$ is the vector bundle $B_j \oplus (\cO_S^{\rR_j})^\vee$.
Consider the direct sum
\begin{equation*}
\widetilde{B}:=\widetilde{B}_1 \oplus \dotsb \oplus \widetilde{B}_N
\end{equation*}
and the morphisms
\begin{equation*}
\begin{array}{l}
\widetilde{\alpha}:=\widetilde{\alpha}_1 + \dotsb + \widetilde{\alpha}_N, \\
\widetilde{\beta}:=\widetilde{\beta}_1 + \dotsb + \widetilde{\beta}_N. \\
\end{array}
\end{equation*}
We end with the two-periodic complex
\begin{equation}\label{anticomm}
\{ \widetilde{\alpha} , \widetilde{\beta} \} = \PV \otimes Z^*\kK(\rR_{\overline{\gamma}}) \in \mathrm{MF}(X,0).
\end{equation}

According to \eqref{foncteur}, it is well-defined to push-forward this two-periodic complex via the projection $p$, then we get
\begin{equation*}
p_* \{ \widetilde{\alpha} , \widetilde{\beta} \} = \Phi(\kK(\rR_{\overline{\gamma}})).
\end{equation*}
By \cite[Remark 1.5.1]{Polish1}, in the case of two-periodic complexes we have a quasi-isomorphism
\begin{equation*}
p_* \{ \widetilde{\alpha} , \widetilde{\beta} \}  ~ \simeq ~ p^{\textrm{naive}}_* \{ \widetilde{\alpha} , \widetilde{\beta} \}
\end{equation*}
where $p_*^\textrm{naive}$ is the naive push-forward, i.e.~the push-forward for quasi-coherent sheaves instead of matrix factorizations.
We denote by $T$ the two-periodic complex
\begin{equation}\label{TTT}
T:=p^\textrm{naive}_* \{ \widetilde{\alpha} , \widetilde{\beta} \}
\end{equation}
of quasi-coherent sheaves on $S$, with
\begin{equation*}
\begin{array}{lcl}
T^+ &:=& \mathrm{Sym}(A^\vee) \otimes \bigwedge_{\textrm{even}} \widetilde{B}^\vee, \\
T^- &:=& \mathrm{Sym}(A^\vee) \otimes \bigwedge_{\textrm{odd}} \widetilde{B}^\vee, \\
\end{array}
\end{equation*}
and the differential $\delta$ induced by \eqref{widealpha}.

\begin{dfn}\label{cvircomplex}
The virtual class evaluated at $e(\rR_{\overline{\gamma}})$ is
\begin{equation}\label{virtualcomplex}
\cvirPV(e(\rR_{\overline{\gamma}})) = \Ch \left( H^+(T) - H^-(T) \right) \frac{\Td (\widetilde{B})}{\Td (A)} \in H^*(S,\CC)
\end{equation}
and it extends in a linear map
\begin{equation*}
\cvirPV \colon \st^{\otimes n} \longrightarrow H^*(S;\CC).
\end{equation*}
\end{dfn}
By \cite[Lemma 1.2.1]{Polish3}, any \MF $U$ in $\textrm{MF}(\aA^{\overline{\gamma}},W_{\overline{\gamma}})$ satisfies
\begin{equation}\label{commut}
\Phi_* (\Ch (U)) = \Ch (\Phi(U)),
\end{equation}
so that the following diagram is commutative.
\begin{equation*}
\xymatrix{
    \mathrm{MF}(\aA^{\overline{\gamma}},W_{\overline{\gamma}}) \ar[r]^{\Phi} \ar[d]_\Ch & \mathrm{MF}(S,0) \ar[d]^\Ch \\
    \otimes_{i=1}^n \st_{\gamma(i)} \ar[r]_{\Phi_*} & H^*(S)
    }
\end{equation*}

\noindent
Thus Definition \ref{cvircomplex} is compatible with \cite[Equation (5.15)]{Polish1} on any base scheme $S$, and yields a morphism
\begin{equation*}
\cvirPV \colon \st^{\otimes n} \longrightarrow H^*(\sS_{g,n};\CC).
\end{equation*}


\section{Recursive complexes}\label{contrib}
This section has six parts. The first four parts introduce a new structure called the recursive complex. They may be read independently from the rest of the paper, but we can keep in mind the two-periodic complex we have just constructed.
These four parts are the setup of the recursive complex (Sect.~\ref{3.1}), the main technical property about its cohomology (Sect.~\ref{3.2}, Theorem \ref{genus0theo}), its proof (Sect.~\ref{proof}), and its application yielding the formula \eqref{formulelim} for the virtual class in terms of characteristic classes.  
In Sect.~\ref{appli}, we get back to the quantum singularity theory of invertible polynomials and compute Polishchuk--Vaintrob's class in Theorem \ref{main}.
In Sect.~\ref{compatsec}, we prove Theorem \ref{compat} on the compatibility between Polishchuk--Vaintrob and Fan--Jarvis--Ruan--Witten theories for (almost) every invertible polynomials.

\subsection{Definition and non-degeneracy}\label{3.1}
Consider two vector bundles $A$ and $B$ on a smooth scheme $S$, and a two-periodic complex
\begin{equation*}
T = \left( \dotsb \rightarrow T^+ \xrightarrow{\delta} T^- \xrightarrow{\delta} T^+ \rightarrow \dotsb \right)
\end{equation*}
of quasi-coherent sheaves on the base space $S$, with
\begin{equation*}
T^+ := \cS ~ A^\vee \otimes \bigwedge\nolimits_{\textrm{even}} B^\vee
\end{equation*}
and $T^-$ replacing ``even'' by ``odd''.

\begin{dfn}\label{defrecursiveMF}
Let $a_1,\dotsc,a_N$ be positive integers greater than or equal to $2$.
We say that $(T,\delta)$ is a recursive complex with respect to $(a_1,\dotsc,a_N)$ if the sheaves $A$ and $B$ can be decomposed into a direct sum of coherent locally free sheaves
\begin{equation*}
A := A_1 \oplus \dotsb \oplus A_N \qquad \textrm{and} \qquad B:= B_1 \oplus \dotsb \oplus B_N
\end{equation*}
such that there are morphisms
\begin{displaymath}
\begin{array}{lclcl}
\alpha_j & \colon & \cO_S & \rightarrow & \left( \cS^{a_j} A_j^\vee \otimes B_{j+1}^\vee \right) \oplus \left( \cS^{a_j-1} A_j^\vee \otimes A_{j+1}^\vee \otimes B_j^\vee \right), \\
\beta_j & \colon & B_j^\vee & \rightarrow & A_j^\vee, \\
\end{array}
\end{displaymath}
whose sum $\alpha_1+\dotsb+\alpha_N+\beta_1+\dotsb+\beta_N$ induce the differential $\delta$.
We use the cyclic convention $A_{N+1}=A_1$ and $B_{N+1}=B_1$.
\end{dfn}

\begin{rem}
The data of a recursive complex embodies a $\ZZ^{N+1}$-grading, as we explain later in Sect.~\ref{proof}. In particular, we recover the two-periodic complex from this $\ZZ^{N+1}$-graded complex in the usual way.
\end{rem}

\begin{exa*}
Let $W=x_1^{a_1}x_2+\dotsb+x_N^{a_N}x_1$ be a loop polynomial.
For any genus, any type $\overline{\gamma} \in \Aut(W)^n$ and any decoration $\rR_{\overline{\gamma}}$, the naive push-forward \eqref{TTT}
\begin{equation*}
p^\textrm{naive}_*(\PV \otimes \kK(\rR_{\overline{\gamma}}))
\end{equation*}
over a base scheme $S$ is a recursive complex.
\end{exa*}

\begin{exa*}
Let $W=x_1^{a_1}x_2+\dotsb+x_{N-1}^{a_{N-1}}x_N+x_N^{a_N+1}$ be a chain polynomial.
In genus $g=0$, for any type $\overline{\gamma} \in \Aut(W)^n$ and any decoration $\rR_{\overline{\gamma}}$, the naive push-forward
\begin{equation*}
p^\textrm{naive}_*(\PV \otimes \kK(\rR_{\overline{\gamma}}))
\end{equation*}
over a base scheme $S$ is a recursive complex, because the vector bundle $A_N$, the morphism $\alpha_N$ and the morphism $\beta_N$ vanish (see \eqref{quasiconcave} and  \cite[Proposition 3.1]{LG/CY} for precisions on this vanishing condition).
\end{exa*}

\begin{rem}
For a chain polynomial in higher genus, the two-periodic complex \eqref{TTT} is not recursive, but satisfies a similar definition, where we adopt the convention $A_{N+1}=0$ and $B_{N+1}=B_N$ instead of the cyclic convention.
\end{rem}

Consider a recursive complex $T$ with integers $a_1,\dotsc,a_N \geq 2$.
Remark that each morphism $\alpha_j$ is a sum of two morphisms $\alpha'_j$ and $\alpha''_j$ with
\begin{equation}\label{alpha'}
\begin{array}{lclcl}
\alpha'_j & \colon & \cO_S & \rightarrow & \cS^{a_j} (A_j)^\vee \otimes (B_{j+1})^\vee, \\
\alpha''_j & \colon & \cO_S & \rightarrow & \cS^{a_j-1} (A_j)^\vee \otimes (A_{j+1})^\vee \otimes (B_j)^\vee. \\
\end{array}
\end{equation}

\begin{dfn}\label{non-degenera}
For any $j$, let us denote by $\overline{A}_{j,s}$ the kernel of the map $\beta^\vee_{j,s}$ over a geometric point $s$ in $S$, and by $\overline{B}_{j,s}$ its cokernel.
Each morphism $\alpha'_{j,s}$ over the point $s$ induces a morphism
\begin{equation*}
\overline{\alpha}'_{j,s} \colon \CC \rightarrow \cS^{a_j} (\overline{A}_{j,s})^\vee \otimes (\overline{B}_{j+1,s})^\vee
\end{equation*}
and we say that the morphism $\alpha'_j$ is non-degenerate if for every geometric point $s$ and every $v \in \overline{A}_{j,s}$, the condition
\begin{equation*}
\forall w \in \overline{B}_{j+1,s}~,~~ (\overline{\alpha}'_{j,s})^\vee  (v^{a_j} \otimes w) = 0
\end{equation*}
implies $v=0$.
A recursive complex is non-degenerate if each morphism $\alpha'_j$ is so. 
\end{dfn}

This definition is equivalent to \cite[Definition 3.1.5]{ChiodoJAG} where one requires the morphism $\alpha'_j$ to induce a base-point free linear system
\begin{equation*}
\overline{B}_{j+1,s} \rightarrow \cS^{a_j} (\overline{A}_{j,s})^\vee = H^0(\PP (\overline{A}_{j,s}),\cO_{\PP (\overline{A}_{j,s})}(a_j)).
\end{equation*}

\subsection{Computation of cohomology}\label{3.2}
A recursive complex is multi-graded; the term of multi-degree $(\underline{p},\underline{q}) \in \NN^{2N}$ is
\begin{equation*}
T^{(\underline{p},\underline{q})} := \left( \cS^{p_1} (A_1)^\vee \otimes \dotsb \otimes \cS^{p_N} (A_N)^\vee \right)  \otimes \left( \Lambda^{q_1}(B_1)^\vee \otimes \dotsb \otimes \Lambda^{q_N} (B_N)^\vee \right)
\end{equation*}
and we have
\begin{equation*}
T^+ := \cS ~ A^\vee \otimes \bigwedge\nolimits_{\textrm{even}} B^\vee = \bigoplus_{\substack{(\underline{p},\underline{q}) \in \NN^{2N} \\ q_1+\dotsb+q_N \textrm{ even}}} T^{(\underline{p},\underline{q})},
\end{equation*}
and $T^-$ replacing ``even'' by ``odd''.
To state Theorem \ref{genus0theo} on the cohomology $H^\pm(T,\delta)$, we need to introduce for any positive intergers $R_1,\dotsc,R_N$ the polytope
\begin{equation}\label{polytope}
\cP(R_1,\dotsc,R_N) := \cP^+ \cap \bigcap_{j=1}^N \cP_j(R_j),
\end{equation}
where 
\begin{equation*}
\cP^+ := \left\lbrace (\underline{p},\underline{q}) \in \NN^{2N} \left| \right. q_j \leq \textrm{rank}(B_j) \right\rbrace \subset \ZZ^{2N} \qquad \textrm{and}
\end{equation*}
\begin{equation*}
\cP_j(R_j) := \biggl\lbrace  (\underline{p},\underline{q}) \in \ZZ^{2N} \left| \right. (p_j+q_j) + \sum_{k=j+1}^N (-a_j) \dotsm (-a_{k-1}) (p_k+q_k) \leq R_j   \biggr\rbrace .
\end{equation*}
Observe that the vector bundle $T^{(\underline{p},\underline{q})}$ vanishes (when $(\underline{p},\underline{q})$ is) outside of $\cP^+$ and the polytope $\cP(R_1,\dotsc,R_N)$ is a finite subset of $\ZZ^{2N}$.

\begin{thm}\label{genus0theo}
Let $T$ be a non-degenerate recursive complex with 
\begin{equation}\label{qqconcave}
A_N=0 \qquad \textrm{(vanishing condition)}.
\end{equation}
Then the cohomology groups $H^+(T)$ and $H^-(T)$ are coherent and, for any sufficiently large integers $R_1 \gg \dotsb \gg R_N \gg 1$, the evaluation of the two-periodic complex $T$ in $K_0(S)$ equals
\begin{equation}\label{sumpolytope}
\sum_{(\underline{p},\underline{q}) \in \cP(R_1,\dotsc,R_N)} (-1)^{q_1+\dotsb+q_N} \bigl[ T^{(\underline{p},\underline{q})}\bigr].
\end{equation}
\end{thm}

We emphasize that the result does not depend on the values of the integers $R_1,\dotsc,R_N$, provided the integers $R_1-R_2,\dotsc,R_{N-1}-R_N$, and $R_N$ are large enough.
We leave the proof to Sect.~\ref{proof}.

\begin{dfn}\label{cvirrecursivdfn}
For a two-periodic complex $(T,\delta)$ with coherent cohomology, we define the virtual class as in \eqref{virtualcomplex} by
\begin{equation}\label{cvirrecursiv}
\cvir(T,\delta) = \Ch \left( H^+(T,\delta) - H^-(T,\delta) \right) \frac{\Td (B)}{\Td (A)} \in H^*(S,\CC).
\end{equation}
\end{dfn}

\begin{cor}\label{maincoro}
Let $T$ be a non-degenerate recursive complex, with the vanishing condition $A_N=0$, as in the above theorem.
Denote by $\AB_j$ the class of $\left[ A_j \rightarrow B_j \right]$ in the derived category $\cD(S)$.
Then the virtual class is
\begin{equation*}
\cvir(T)= \sum \Ch^\vee(\cS^{z_1}\AB_1) \dotsm \Ch^\vee(\cS^{z_N}\AB_N) \prod_{j=1}^N \frac{1}{\Td(\AB_j)},
\end{equation*}
where the sum is taken over all $(z_1,\dotsc,z_N) \in \mathbb{N}^N$ such that for each $j$ we have
\begin{equation}\label{domainofsum}
z_j - a_j z_{j+1} + \dotsb + (-a_j) \dotsm (-a_{N-1}) z_N \leq R_j
\end{equation}
and where $\Ch_l^\vee$ stands for $(-1)^l \Ch_l$.

In particular, the result is independent of the choice of a representative $\left[ A_j \rightarrow B_j \right]$ for the derived element $\AB_j$.
\end{cor}

\begin{proof}
By Theorem \ref{genus0theo}, the virtual class equals
\begin{equation*}
\cvir(T)= \sum_{(\underline{p},\underline{q}) \in \bigcap_{j=1}^N \cP_j(R_j)} (-1)^{q_1+\dotsb+q_N} \Ch ( T^{(\underline{p},\underline{q})} ) \prod_{j=1}^N \frac{1}{\Td(\AB_j)},
\end{equation*}
where the sum is taken over $\bigcap_{j=1}^N \cP_j(R_j)$ because the vector bundle $T^{(\underline{p},\underline{q})}$ vanishes outside of $\cP^+$.
Consider the function $f \colon \NN^{2N} \rightarrow \NN^N$ which sends $(\underline{p},\underline{q})$ to $(\underline{z})$ with $z_j=p_j+q_j$,
and observe that $f(\cP_j(R_j))$ is exactly the subset delimited by 
\begin{equation*}
z_j - a_j z_{j+1} + \dotsb + (-1)^{N-j} a_j \dotsm a_{N-1} z_N \leq R_j.
\end{equation*}
Thus the virtual class is
\begin{equation*}
\cvir(T)  =  \sum_{(\underline{z}) \in f(\bigcap_{j=1}^N \cP_j(R_j))} \sum_{(\underline{p},\underline{q}) \in f^{-1}(\underline{z})} (-1)^{q_1+\dotsb+q_N} \Ch ( T^{(\underline{p},\underline{q})} ) \prod_{j=1}^N \frac{1}{\Td(\AB_j)}.
\end{equation*}
The corollary follows from the equality (see Sect.~\ref{term})
\begin{eqnarray*}
\textrm{Ch}^\vee(\cS^k\AB_j) & = & \textrm{Ch}^\vee(\cS^k\left[ A_j \rightarrow B_j \right]) \\
& = & \sum_{p+q=k} (-1)^q \textrm{Ch}(\cS^p(A_j)^\vee)\textrm{Ch}(\Lambda^q(B_j)^\vee).
\end{eqnarray*}
\end{proof}

\begin{cor}[Concave case]
Assume that all the vector bundles $A_1, \dotsc, A_N$ vanish.
Then the virtual class reduces to
\begin{equation*}
\cvir(T) = \sum_{q \geq 0} (-1)^q \Ch(\Lambda^q B^\vee) \Td(B) = \textrm{c}_{\textrm{top}}(B).
\end{equation*}
\end{cor}

\subsection{Proof of Theorem \ref{genus0theo}}\label{proof}
Our strategy is to write the quasi-coherent sheaf $T$ as a direct sum (of sheaves and not of two-periodic complexes)
\begin{equation*}
T = T_0 \oplus T_1 \oplus \dotsb \oplus T_N
\end{equation*}
with $T_0$ a coherent locally free sheaf.
For any $j \geq 1$, the subsheaf $T_j$ satisfy a stability condition (see Definition \ref{stable}) with respect to the morphisms $\alpha_j$ and $\beta$, and vanishes in cohomology, that is,
\begin{equation*}
H(T_j,\alpha_j+\beta)=0.
\end{equation*}
We take advantage of this situation to remove each piece $T_j$ but $T_0$. The remaining evaluation in K-theory is exactly \eqref{sumpolytope}.
The proof is in four steps.

In Step $1$, we change the natural $\ZZ^{2N}$-grading of the sheaf $T$ into a $\ZZ^{N+1}$-grading so that the multi-degrees of the morphisms $\alpha_1,\dotsc,\alpha_N$ and $\beta$ form the canonical basis of $\ZZ^{N+1}$.
Although only the grading is changing and not the sheaf, we prefer to write $K$ instead of $T$ when we deal with the $\ZZ^{N+1}$-grading.

In Step $2$, we use \cite[Theorem 2]{Green} to find for each $j$ some planes in $\ZZ^{N+1}$ with the following property.
The sheaf of elements with multi-degrees in these planes, together with the morphisms $\beta$ and $\alpha_j$, is a double complex whose cohomology vanishes.
We call these planes exact and illustrate this in Example \ref{Ndeux} for $N=2$.

In Step $3$, we cut $\ZZ^{N+1}$ into a puzzle with pieces $\cQ_0,\dotsc,\cQ_N$ where, for $j \geq 1$, each $\cQ_j$ is made of some exact planes with respect to the morphisms $\beta$ and $\alpha_j$.
Altogether, the pieces cover $\ZZ^{N+1}$ and do not overlap. The subsheaf whose element have multi-degree in $\cQ_j$ is $K_j$ (or $T_j$ if we work with the $\ZZ^{2N}$-grading).
We illustrate the case $N=3$ (see Example \ref{Ntrois}).

In Step $4$, we use an argument relying on spectral sequences (see Lemma \ref{spectral2}) to remove the subsheaf $T_1$. Then we remove the subsheaf $T_2$ and so on, until there remains only $T_0$ and $T_N$.
To conclude, we need the vanishing condition $A_N=0$ which implies that $T_N$ is empty.

\noindent
\textit{Step $1$: a change of grading.}
Let $\du_1,\dotsc,\du_N,\dv_1,\dotsc,\dv_N$ be the canonical basis of $\ZZ^{2N}$ with the coordinates $(\underline{p},\underline{q})$ and let  $\dw_1,\dotsc,\dw_{N+1}$ be the canonical basis of $\ZZ^{N+1}$ with the coordinates $(\underline{k},l)$.
The quasi-coherent sheaf $T$ is $\ZZ^{2N}$-graded and the degrees of the morphisms $\beta_j$, $\alpha'_j$ and $\alpha''_j$ are
\begin{eqnarray*}
\deg (\beta_j) & = & \du_j - \dv_j, \\
\deg (\alpha'_j) & = & a_j \du_j + \dv_{j+1}, \\
\deg (\alpha''_j) & = & (a_j-1)\du_j + \du_{j+1} + \dv_j,
\end{eqnarray*}
where we use the cyclic convention $\du_{N+1}=\du_1$ and $\dv_{N+1}=\dv_1$.
The lattice $L$ generated by these $3N$ vectors is a sub-lattice of $\ZZ^{2N}$ and we fix a finite subset $\cE$ of $\ZZ^{2N}$ such that
\begin{equation*}
\ZZ^{2N} = \bigsqcup_{\ccE \in \cE} (\ccE + L) \qquad \textrm{and} \qquad (p,q) \in \cE \implies q \textrm{ is even}.
\end{equation*}

For simplicity, we prefer to work with a $\ZZ^{N+1}$-grading on $T$ such that each morphism $\beta_j$ has degree $\dw_{N+1}$ and each morphism $\alpha_j=\alpha'_j+\alpha''_j$ has degree $\dw_j$.
Notice that a grading assigning the same degree to $\alpha'_j$ and $\alpha''_j$ should assign the same degree to $\beta_j$ and $\beta_{j+1}$.
This happens because
\begin{equation*}
\deg (\alpha'_j) + \deg (\beta_{j+1}) = \deg (\alpha''_j) + \deg (\beta_j).
\end{equation*}
For any $\ccE \in \cE$ and $(\underline{k},l) \in \ZZ^{N+1}$, we consider the direct sum
\begin{equation}\label{Kdef}
(K_{\ccE})^{\underline{k},l} := \bigoplus_{(\underline{\lambda}) \in \ZZ^{N-1}}   T^{\ccE+\mathcal{A} \cdot (\underline{k},l,\underline{\lambda})}
\end{equation}
where $\mathcal{A} \cdot (\underline{k},l,\underline{\lambda})$ is (with only non-zero entries represented)

\begin{center}
\begin{tikzpicture}[scale=2.5]
\node (a1) at (0.25,1.525) {};
\draw (0.27,1.525) node{$a_1$};
\node (aN) at (0.85,0.925) {};
\draw (0.82,0.925) node{$a_N$};
\node (b1) at (0.25,0.7) {$1$};
\node (bN) at (0.75,0.2) {$1$};
\node (bNN) at (0.85,0.8) {$1$};

\node (l1) at (0.97,1.525) {$1$};
\node (lN) at (0.99,0.8) {$1$};
\draw (0.91,0.8)--(0.95,0.8);

\node (aa1) at (1.1,1.525) {$1$};
\draw (1.02,1.525)--(1.06,1.525);
\node (aaa1) at (1.1,1.425) {$1$};
\node (aaN) at (1.6,1.025) {$1$};
\draw (1.52,1.025)--(1.56,1.025);
\node (aaaN) at (1.6,0.925) {$1$};
\node (bb1) at (1.1,0.8) {$1$};
\node (bbb1) at (1.1,0.7) {$1$};
\draw (1.02,0.7)--(1.06,0.7);
\node (bbN) at (1.6,0.3) {$1$};
\node (bbbN) at (1.6,0.2) {$1$};
\draw (1.52,0.2)--(1.56,0.2);

\draw[dotted] (a1) -- (aN);
\draw[dotted] (b1) -- (bN);
\draw (bNN);
\draw (l1);
\draw (lN);
\draw[dotted] (aa1) -- (aaN);
\draw[dotted] (bb1) -- (bbN);
\draw[dotted] (aaa1) -- (aaaN);
\draw[dotted] (bbb1) -- (bbbN);

\draw (0.2,0.15) -- (0.2,1.6) -- (1.65,1.6) -- (1.65,0.15) -- cycle;
\draw (0.2,0.875) -- (1.65,0.875);
\draw (0.9,1.6) -- (0.9,0.15);
\draw (1.02,1.6) -- (1.02,0.15);

\draw (1.835,0.15) -- (1.835,1.6) -- (2.155,1.6) -- (2.155,0.15) -- cycle;
\draw (1.835,0.875) -- (2.155,0.875);
\draw (1.835,0.76) -- (2.155,0.76);
\draw (2.4,0.15) -- (2.4,1.6) -- (2.6,1.6) -- (2.6,0.15) -- cycle;
\draw (2.4,0.875) -- (2.6,0.875);

\node (k1) at (2,1.525) {$k_1$};
\node (kN) at (2,0.945) {$k_N$};
\draw (2,0.814) node{$l$};
\node (la1) at (2,0.697) {$\lambda_1$};
\node (laN) at (2,0.215) {};
\draw (2,0.215) node{$\lambda_{N-1}$};

\node (p1) at (2.5,1.525) {$p_1$};
\node (pN) at (2.5,0.925) {$p_N$};
\node (q1) at (2.5,0.795) {$q_1$};
\node (qN) at (2.5,0.2) {$q_N$};

\draw (1.74,0.875) node{$\cdot$};
\draw (2.275,0.875) node{$=$};
\draw (2.75,0.875) node{$- \ e$};

\draw[dotted] (k1) -- (kN);
\draw[dotted] (la1) -- (laN);
\draw[dotted] (p1) -- (pN);
\draw[dotted] (q1) -- (qN);
\end{tikzpicture}
\end{center}

\noindent
Observe that the determinant of the matrix $\mathcal{A}$ is
\begin{equation*}
\det \mathcal{A} = (-1)^N a_1 \dotsm a_N - 1 \neq 0, \qquad \textrm{(because $a_1,\dotsc,a_N \geq 2$).}
\end{equation*}
By looking at the columns of $\mathcal{A}$, it is straightforward to check that for each $j$
\begin{equation*}
\beta_j \colon (K_{\ccE})^{\underline{k},l} \rightarrow (K_{\ccE})^{(\underline{k},l)+\dw_{N+1}} \qquad \textrm{and} \qquad \alpha_j \colon (K_{\ccE})^{\underline{k},l} \rightarrow (K_{\ccE})^{(\underline{k},l)+\dw_j}.
\end{equation*}
Notice that the direct sum \eqref{Kdef} is finite because the sheaf $T^{(\underline{p},\underline{q})}$ is nonzero only when $0 \leq q_j \leq \rank B_j$ for each $j$ and because we have
\begin{eqnarray*}
q_1 & = & k_N-l+\lambda_1 + e_{N+1} \\
q_2 & = & k_1-\lambda_1+\lambda_2 + e_{N+2} \\
& \dotsb & 
\end{eqnarray*}
Thus the locally free sheaf $(K_{\ccE})^{\underline{k},l}$ is coherent.
The total complex $K_{\ccE}^{\textrm{tot}}$ is
\begin{equation}\label{totalcomplex}
\left( K_{\ccE}^{\textrm{tot}} \right)^m := \bigoplus_{k_1+\dotsb+k_N+l=m} (K_{\ccE})^{k_1,\dotsc,k_N,l}, \qquad \delta \colon \left( K_{\ccE}^{\textrm{tot}} \right)^m \rightarrow \left( K_{\ccE}^{\textrm{tot}} \right)^{m+1}
\end{equation}
and the associated two-periodic complex
\begin{equation*}
\bigoplus_{\substack{m \in \ZZ \\ \ccE \in \cE}} \left( K_{\ccE}^{\textrm{tot}} \right)^{2m} \leftrightarrows \bigoplus_{\substack{m \in \ZZ \\ \ccE \in \cE}} \left( K_{\ccE}^{\textrm{tot}} \right)^{2m+1} \qquad \textrm{equals} \qquad (T,\delta).
\end{equation*}

We define a surjective function
\begin{equation*}
\Psi \colon \ZZ^{2N} \rightarrow \ZZ^{N+1} \times \cE
\end{equation*}
which sends $(\underline{p},\underline{q})$ to the unique $(\underline{k},l,\ccE)$ such that there is $\underline{\lambda} \in \ZZ^{N-1}$ with
\begin{equation*}
\mathcal{A}\cdot (\underline{k},l,\underline{\lambda}) = (\underline{p},\underline{q}) - \ccE.
\end{equation*}
For any $(\underline{p},\underline{q}) \in \ZZ^{2N}$ and $(\underline{k},l,\ccE) \in \ZZ^{N+1} \times \cE$ such that $\Psi(\underline{p},\underline{q})=(\underline{k},l,\ccE)$, we have
\begin{equation}\label{Psi map}
\begin{array}{rcl}
k_{j-1}+a_j k_j & = & p^\ccE_j+q^\ccE_j \qquad \qquad \textrm{for any } j, \\
k_1 + \dotsb + k_N - l & = &  q^\ccE_1 + \dotsb + q^\ccE_N,
\end{array}
\end{equation}
where $(\underline{p}^\ccE,\underline{q}^\ccE):=(\underline{p},\underline{q})-\ccE$ and with the convention $k_0=k_N$.
We define the subset $\cQ$ of $\ZZ^{N+1}$ by
\begin{equation}\label{cQ}
\cQ := \biggl\lbrace (\underline{k},l) \left| \right. \forall j, ~~  0 \leq k_{j-1}+a_j k_j \textrm{ and } 0 \leq \sum_{j=1}^N k_j -l \leq \sum_{j=1}^N \rank (B_j) \biggr\rbrace.
\end{equation}
Since the set $\Psi (\cP^+)$ is included in $\cQ \times \cE$ by \eqref{Psi map}, the vector bundle $(K_{\ccE})^{\underline{k},l}$ vanishes outside $\cQ$.

\begin{lem}\label{finitesum}
Since we have $a_j \geq 2$ for every $j$, the direct sum \eqref{totalcomplex} is finite.
\end{lem}

\begin{proof}
Fix the integer $m$.
We show there is a finite number of $(\underline{k},l) \in \cQ$ satisfying
\begin{equation}\label{summ}
k_1+\dotsb+k_N+l=m.
\end{equation}
For any $(\underline{k},l) \in \cQ$, we have
\begin{equation*}
0 \leq k_1+\dotsb+k_N-l \leq \textrm{rank}(B_1 \oplus \dotsb \oplus B_N),
\end{equation*}
so that the number of possible values for $l$ under \eqref{summ} is finite.

For any $(\underline{k},l) \in \cQ$, consider the sum $k^+$ of all positive integers among $k_1,\dotsc,k_N$ and the sum $k^-$ of all negative integers among them.
The condition \eqref{summ} becomes $k^++k^-+l=m$, and if there is a finite number of possible values for $k^+$ or for $k^-$ under \eqref{summ}, then it is clear that the subset of $\cQ$ satisfying \eqref{summ} is finite.
For any $(\underline{k},l) \in \cQ$, observe
\begin{equation*}
k_{j-1} \geq -a_j k_j \geq -2 k_j,
\end{equation*}
so that if $k_j$ contributes to $k^-$ (i.e.~$k_j$ is negative), then $k_{j-1}$ contributes to $k^+$ (i.e.~$k_{j-1}$ is positive).
Thus we get $k^+ \geq - 2 k^-$ and
\begin{equation*}
k_1+\dotsb+k_N := k^+ + k^- \geq - k^-.
\end{equation*}
Consequently, the set of values for $k^-$ under \eqref{summ} is finite.
\end{proof}

\noindent
\textit{Step $2$: exact sequences.}
When we fix all the coordinates of $\ZZ^{N+1}$ but $k_j$ and $l$, we get a double complex $(K_\ccE^{\underline{k},l}; \alpha_j, \beta)$.
Over a geometric point $s \in S$, denote its cohomology along $\beta_s$ and then along $\alpha_{j,s}$ by
\begin{equation*}
H(H((K_{\ccE}^{\underline{k},l})_s,\beta_s),\alpha_{j,s}).
\end{equation*}

\begin{lem}\label{exactseq}
For each $j$, there is a constant $R_j \geq 0$ such that for every geometric point $s$ in $S$ and for every $(\underline{k},l,\ccE) \in \ZZ^{N+1} \times \cE$, we have
\begin{equation}\label{exactseq2}
H(H((K_\ccE^{\underline{k},l})_s,\beta_s),\alpha_{j,s})=0,
\end{equation}
whenever $k_{j-1}-a_j a_{j+1} k_{j+1} >R_j$. Moreover, we take $R_N \geq \rank(B_N)$.
\end{lem}

\begin{proof}
We omit the index $\ccE$ in $(K_\ccE^{\underline{k},l})_s$ to simplify notations.
Following the proof of \cite[Theorem 3.3.1]{ChiodoJAG}, we see that 
\begin{equation}\label{pointwise}
H_{\beta_s}(K^{\underline{k},l}_s)=\overline{K}_s^{(\underline{k},l)},
\end{equation}
where $\overline{K}_s^{(\underline{k},l)}$ is the subspace of $K_s^{(\underline{k},l)}$ obtained when we replace the vector space $A_{j,s}$ (resp. $B_{j,s}$) by $\overline{A}_{j,s}$ (resp. $\overline{B}_{j,s}$) in the construction of $K_s^{(\underline{k},l)}$, with $\overline{A}_{j,s}$ the kernel of $\beta^\vee_{j,s}$ and $\overline{B}_{j,s}$ its cokernel.
Still as in the proof of \cite[Theorem 3.3.1]{ChiodoJAG}, we use the non-degeneracy condition, i.e.~
\begin{equation*}
\alpha'_{j,s} \colon \overline{B}_{j+1,s} \rightarrow \cS^{a_j} (\overline{A}_{j,s})^\vee = H^0(\PP(\overline{A}_{j,s}), \cO_{\PP(\overline{A}_{j,s})}(a_j)
\end{equation*}
is a base-point free linear system, to deduce from \cite[Theorem 2]{Green} that the complex
\begin{equation*}
(\cS^{p_j+a_j \cdot t} (\overline{A}_{j,s})^\vee \otimes \Lambda^t (\overline{B}_{j+1,s})^\vee)_t \qquad \textrm{with differential } \alpha'_{j,s}
\end{equation*}
is exact when $p_j$ is larger than $R'_{j,s} := a_j+\mathrm{rank}(\cS^{a_j}\overline{A}_{j,s})$.

By \eqref{Psi map}, the complex $((\overline{K}_s^{(\underline{k},l)+ t \cdot \dw_j})_t,\alpha'_{j,s})$ is bounded below by
\begin{equation*}
t \geq \ccE_{j+1}+\ccE_{N+j+1}-(k_j+ a_{j+1} k_{j+1}) \qquad \textrm{with $\ccE = (\ccE_1,\dotsc, \ccE_{2N}) \in \cE \subset \ZZ^{2N}$}
\end{equation*}
because otherwise we have $p_{j+1}+q_{j+1}<0$. For this value of $t$, we have
\begin{eqnarray*}
p_j+q_j & = & \ccE_j+\ccE_{N+j} + k_{j-1}+ a_j (k_j +\ccE_{j+1}+\ccE_{N+j+1}-(k_j+ a_{j+1} k_{j+1})) \\
& = & (k_{j-1}- a_j a_{j+1} k_{j+1}) + \ccE_j+\ccE_{N+j} + a_j(\ccE_{j+1}+\ccE_{N+j+1}).
\end{eqnarray*}
As a consequence, the complex $((\overline{K}_s^{(\underline{k},l)+ t \cdot \dw_j})_t,\alpha'_{j,s})$ is exact whenever
\begin{equation}\label{inef}
k_{j-1} - a_j a_{j+1} k_{j+1} > R_{j,s,\ccE},
\end{equation}
with $R_{j,s,\ccE} := R'_{j,s,\ccE} + \rank(\overline{B}_{j,s})-\ccE_j-\ccE_{N+j}-a_j(\ccE_{j+1}+\ccE_{N+j+1})$.

Now, we consider the double complex
\begin{equation*}
C^{u,v}:=\bigoplus_{(\lambda_1,\dotsc,\hat{\lambda}_j,\dotsc,\lambda_{N-1}) \in \ZZ^{N-2}}   \overline{T}_s^{\ccE+\mathcal{A} \cdot (\underline{k},l,\underline{\lambda})} \qquad \textrm{with }k_j:=u+v~,~~\lambda_j:=v,
\end{equation*}
where the horizontal differential is $f:=\alpha'_{j,s}$ and the vertical differential is $g:=\alpha''_{j,s}$. 
Assume the inequality \eqref{inef}.
We have just proved that the total cohomology of the double complex $(C;f,0)$ vanishes,
\begin{equation}\label{moche}
H^{\bullet}(C^\textrm{tot},f)=0
\end{equation}
and we want to prove that the total cohomology of $(C;f,g)$ also vanishes,
\begin{equation}\label{preuvtor}
H^{\bullet}(C^\textrm{tot},f+g)=0.
\end{equation}
Take an element
\begin{equation*}
c = \sum_{v \in \ZZ} c_v \qquad \textrm{with } c_v \in C^{k_j-v,v},
\end{equation*}
such that $f(c)+g(c)=0$.
Since the direct sum \eqref{Kdef} is finite, then there is an integer $w$ such that $c_v=0$ for $v < w$.
Observe that $f(c)+g(c)=0$ reads
\begin{equation*}
f(c_v)+g(c_{v-1})=0 ~, \qquad \textrm{for all } v.
\end{equation*}
Thus, we obtain $f(c_w)=0$, and by \eqref{moche}, there is $c'_w$ such that $f(c'_w) = c_w$.
Then,
\begin{equation*}
\begin{array}{lcl}
0 & = & f(c_{w+1})+g(c_w) \\
  & = & f(c_{w+1})+g(f(c'_w)) \\
  & = & f(c_{w+1})-f(g(c'_w)) \\
  & = & f(c_{w+1} - g(c'_w)), \\
\end{array}
\end{equation*}
and by \eqref{moche}, there is $c'_{w+1}$ such that $f(c'_{w+1}) = c_{w+1} - g(c'_w)$.
By induction we construct $c'_v$ for all $v \geq w$ and we set
\begin{equation*}
c'=\sum_{v \geq w} c'_v
\end{equation*}
to get $(f+g)(c')=c$. This proves \eqref{preuvtor} and we conclude with the following value for $R_j$, which is independent from the geometric point $s$ and from $\ccE \in \cE$,
\begin{equation*}
R_j := \textrm{max}_{\ccE \in \cE} \left\lbrace a_j+\mathrm{rank}(\cS^{a_j} A_j) + \rank(B_j) -\ccE_j-\ccE_{N+j}-a_j\ccE_{j+1}-a_j\ccE_{N+j+1}\right\rbrace,
\end{equation*}
and we eventually increase $R_N$ to have $R_N \geq \rank(B_N)$.
\end{proof}

\begin{exa}\label{Ndeux}
We illustrate Lemma \ref{exactseq} when $N=2$.
Since the vanishing condition \eqref{qqconcave} imposes $A_2=0$, there are three vector bundles $A_1,B_1$, and $B_2$ together with two morphisms $\beta_1$ and $\alpha_1$.
Here, $T$ decomposes into a direct sum over $r \in \QQ$ of double complexes
\begin{equation*}
\bigoplus_{p_1+q_1-a_1q_2=r} (T^{p_1,q_1,q_2} , \alpha_1,\beta_1).
\end{equation*}
given by planes directed by the vectors $(a_1,0,1)$ and $(1,-1,0)$ (see the next figure).
Since the integers $p_1,q_1$ and $q_2$ must be non-negative, the set of rational numbers $r$ contributing to this sum is bounded below.

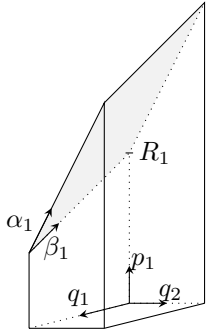
\begin{center}
\begin{tikzpicture}[x= {(-1.333cm,-0.333cm)}, y={(1cm,0cm)}, z={(0cm,1cm)}]
\draw (0,1,0) -- (1,1,0) -- (1,0,0);
\draw[dotted] (1,0,0) -- (0,0,0) -- (0,1,0);
\fill[gray,opacity=0.1] (1,0,1) -- (0,0,2) -- (0,1,4) -- (1,1,3) -- cycle;
\draw (1,1,0)--(1,1,3)--(1,0,1)--(1,0,0);
\draw (1,1,3)--(0,1,4)--(0,1,0);
\draw[dotted] (1,0,1)--(0,0,2)--(0,1,4);

\draw (0,-0.05,2)--(0,0.05,2);
\draw (0,0,2) node[right] {$R_1$};

\draw[->,>=stealth] (1,0,1) -- (1,0.3,1.6);
\draw (1,0.2,1.4) node[left] {$\alpha_1$};

\draw[->,>=stealth] (1,0,1) -- (0.7,0,1.3);
\draw (0.95,0,1.05) node[right] {$\beta_1$};

\draw[->,>=stealth] (0,0,0) -- (0,0,0.5);
\draw[->,>=stealth] (0,0,0) -- (0,0.5,0);
\draw[->,>=stealth] (0,0,0) -- (0.5,0,0);
\draw (0.5,0,0) node[above] {$q_1$};
\draw (0,0.55,-0.1) node[above] {$q_2$};
\draw (0,-0.1,0.55) node[right] {$p_1$};
\draw[dotted] (0,0,0) -- (0,0,2);
\end{tikzpicture}
\end{center}

\noindent
Lemma \ref{exactseq} claims that for any $r>R_1$ the corresponding two-periodic complex is exact above any $s \in S$.
Since it is a two-periodic complex of coherent locally free sheaves, it is exact over the base scheme $S$.
Thus, the cohomology of $T$ is given by
\begin{equation*}
\bigoplus_{r \leq R_1} H\Bigl(\bigoplus_{p_1+q_1-a_1q_2=r} T^{p_1,q_1,q_2} , \alpha_1+\beta_1\Bigr).
\end{equation*}
Since the direct sum is finite, the evaluation in K-theory of $T$ equals
\begin{equation*}
\bigoplus_{r \leq R_1} \bigoplus_{p_1+q_1-a_1q_2=r} (-1)^{q_1+q_2} \left[ T^{p_1,q_1,q_2} \right]
\end{equation*}
and this coincides with \eqref{sumpolytope}.
\end{exa}

\noindent
\textit{Step $3$: the puzzle.}
In general, the two-periodic complex $T$ is not a direct sum of two-periodic complexes.
To understand how to deal with it, we treat an example.

\begin{exa}\label{Ntrois}
When $N=3$, the vanishing condition \eqref{qqconcave} implies $A_3=0$ and there remains five vector bundles, $A_1, A_2, B_1, B_2$, and $B_3$, with four morphisms.
Consider the coordinates $z_j=p_j+q_j$ for $j=1,\dotsc,3$ and assume for simplicity that the morphism $\alpha''_j$ in \eqref{alpha'} is zero (this is not the case in general).
Since the coordinate $z_j$ corresponds to
\begin{equation*}
\bigoplus_{p_j+q_j=z_j} \cS^{p_j} A_j^\vee \otimes \bigwedge\nolimits^{q_j} B_j^\vee,
\end{equation*}
we impose the condition $z_3 \leq R_3:= \rank (B_3)$.
We interpret Lemma \ref{exactseq} as follows.

We have two kind of exact lines, in the sense that the corresponding double complex is exact:
\begin{enumerate}
\item for all $z_1 > R_1$ and for all $z_3$, the line $(z_1,0,z_3)+ t \cdot (a_1,1,0)$ is exact,
\item for all $z_2 > R_2$ and for all $z_1$, the line $(z_1,z_2,0)+ t \cdot (0,a_2,1)$ is exact
\end{enumerate}
The idea is to get rid of a maximal set of exact and disjoint lines.
First, we fix $\widetilde{R}_1 \geq R_1$ and we consider all the exact lines of the first kind with $z_1 > \widetilde{R}_1$ and $z_3=0$.
Then we accept only the exact lines of the second kind with the extra condition $z_1 - a_1 z_2 \leq \widetilde{R}_1$, in order to avoid the overlaps.
Finally we consider all the lines of the first kind with $z_3>0$ and $z_1 + a_1 a_2 z_3 > \widetilde{R}_1$, in order to stick with the previous lines.
Notice that if $\widetilde{R}_1$ is greater than $R_1+a_1a_2R_3$, then the last lines we have drawn are exact.
We sum up this construction by defining two sets
\begin{displaymath}
\cQ_1 := \left\{\begin{array}{ll}
(z_1,0,z_3)+ t \cdot (a_1,1,0) \in \NN^3 \left| \right. &  z_1 + a_1 a_2 z_3 > R_1+a_1a_2R_3 ~, \\
&  z_3 \leq R_3 \textrm{ and } t \in \NN \\
\end{array} \right\},
\end{displaymath}
\begin{displaymath}
\cQ_2 := \left\{\begin{array}{ll}
(z_1,z_2,0)+ t \cdot (0,a_2,1) \in \NN^3 \left| \right. & z_1-a_1z_2 \leq R_1+a_1a_2R_3 ~, \\
& z_2 > R_2 ~,~~ z_3 \leq R_3 \textrm{ and } t \in \NN \\
\end{array} \right\}.
\end{displaymath}
The complement of the subsets $\cQ_1$ and $\cQ_2$ is the finite polytope represented in

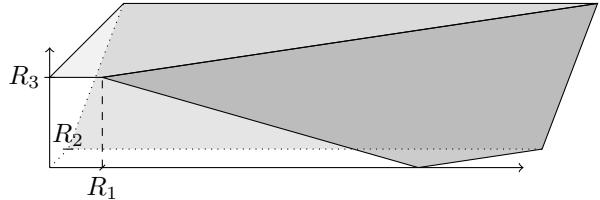
\begin{center}
\begin{tikzpicture}[x= {(0.35cm,0cm)}, y={(0.123cm,0.123cm)}, z={(0cm,0.4cm)}]
\fill[gray,opacity=0.1] (0,0,3)--(2,0,3)--(18,8,3)--(0,8,3)--cycle; 
\fill[gray,opacity=0.4] (2,0,3)--(14,0,0)--(18,2,0)--(18,8,3)--cycle; 
\fill[gray,opacity=0.2] (0,2,0)--(18,2,0)--(18,8,3)--(0,8,3)--cycle; 

\draw (0,0,3)--(2,0,3)--(18,8,3)--(0,8,3)--cycle;
\draw (2,0,3)--(14,0,0)--(18,2,0)--(18,8,3)--cycle;

\draw[dashed] (2,0,0)--(2,0,3);

\draw[->] (0,0,0) -- (18,0,0);
\draw[dotted] (0,0,0) -- (0,2,0);
\draw[dotted] (18,2,0)--(0,2,0)--(0,8,3);
\draw[->] (0,0,0) -- (0,0,4);

\draw (2,0,0) node[below]{$R_1$} ;
\draw (2,-0.3,0)--(2,0.3,0) ;
\draw (0,2,-0.2) node[above]{$R_2$} ;
\draw (-0.2,2,0)--(0.2,2,0) ;
\draw (0,0,3) node[left]{$R_3$} ;
\draw (-0.2,0,3)--(0.2,0,3) ;
\end{tikzpicture}
\end{center}
\noindent
and it coincides with the subset of Theorem \ref{genus0theo},
\begin{displaymath}
\cQ_0 := \left\{\begin{array}{ll}
(z_1,z_2,z_3) \in \NN^3 \left| \right. & z_1-a_1 z_2+a_1a_2z_3 \leq R_1+a_1a_2R_3  ~,~~  \\
 & z_2-a_2 z_3 \leq R_2 \textrm{ and } z_3 \leq R_3 \\
\end{array} \right\}.
\end{displaymath}
\end{exa}

We go back to the proof of Theorem \ref{genus0theo} and we give a definition of the subsets $\cQ_j$ in terms of the coordinates $(k_1,\dotsc,k_N,l)$.
Let $R_1,\dotsc,R_N$ be the integers obtained in Lemma \ref{exactseq} and define recursively $\widetilde{R}_1,\dotsc,\widetilde{R}_N$ by
\begin{equation*}
\widetilde{R}_N=R_N ~,~~ \widetilde{R}_{N-1}=R_{N-1} \qquad \textrm{and} \qquad \widetilde{R}_k = R_k + a_k a_{k+1} \widetilde{R}_{k+2}.
\end{equation*}
For any $(\epsilon_1,\dotsc,\epsilon_N)$ in $\left\lbrace -1,1\right\rbrace ^N$, we define the subset $\cQ_{(\underline{\epsilon})}$ of $\cQ$ delimited by
\begin{equation}\label{Qepsilon}
\begin{array}{ll}
k_{j-1} + (-1)^{N-j} a_j \dotsm a_N k_N \leq \widetilde{R}_j & \textrm{if } \epsilon_j=1 \textrm{ and} \\
k_{j-1} + (-1)^{N-j} a_j \dotsm a_N k_N > \widetilde{R}_j & \textrm{if } \epsilon_j=-1, \\
\end{array}
\end{equation}
for every $j$ and with the cyclic convention $k_0=k_N$.
These subsets form the pieces of our puzzle
\begin{equation*}
\cQ = \bigsqcup_{(\underline{\epsilon}) \in \left\lbrace -1,1 \right\rbrace^N } \cQ_{(\underline{\epsilon})}.
\end{equation*}

\begin{dfn}\label{stable}
Let $f \colon T \rightarrow T$ be a linear endomorphism and $Z \subset \ZZ^{N+1}$.
We say that $f$ ends in $Z$ if the image of every element with degree in $Z$ has degree in $Z$ or vanishes.
On the contrary, we say that $f$ starts from $Z$ if every pre-image of every element with degree in $Z$ has degree in $Z$.
When a morphism $f$ starts from $Z$ and ends in $Z$, we say that $Z$ is stable under $f$.
\end{dfn}

For instance, for any $j \neq N$, the morphism $\alpha_j$ starts from any $\cQ_{(\epsilon_1,\dotsc,\epsilon_N)}$ with $\epsilon_{j+1}=1$ and ends in any $\cQ_{(\epsilon_1,\dotsc,\epsilon_N)}$ with $\epsilon_{j+1}=-1$.
The morphism $\alpha_N$ plays a special role, because of the asymmetry between $k_N$ and the other coordinates in \eqref{Qepsilon}.
Moreover, any subset $\cQ_{(\epsilon_1,\dotsc,\epsilon_N)}$ is stable under the morphism $\beta$.

\begin{pro}\label{puzzle}
Fix an integer $1 \leq j \leq N$ and choose $(\epsilon_1,\dotsc,\epsilon_N)$ in $\left\lbrace -1,1\right\rbrace^N$ such that
\begin{equation*}
\begin{array}{ll}
 \textrm{if } j \leq N-2, & \left\lbrace  \begin{array}{l}
 \epsilon_j=-1 \\
 \epsilon_{j+2}=1
 \end{array}\right. \\[0.4cm]
\textrm{if } j = N-1 \textrm{ or } j=N, &  \quad \epsilon_j=-1. \\
\end{array}
\end{equation*}
Then for any geometric point $s$ in $S$ and $(\underline{k},l) \in \cQ_{(\epsilon_1,\dotsc,\epsilon_N)}$, we have
\begin{equation*}
H(H((K_\ccE^{\underline{k},l})_s,\beta_s),\alpha_{j,s})=0.
\end{equation*}
\end{pro}

\begin{proof}
For $j=N$, since $\epsilon_N=-1$ implies
\begin{equation*}
p_N + q_N = k_{N-1}+a_Nk_N > R_N \geq \rank(B_N)
\end{equation*}
and since the vanishing condition in Theorem \ref{genus0theo} requires $A_N=0$, then we have directly $K_\ccE^{\underline{k},l}=0$.
For $j=N-1$, we have $\epsilon_{N-1}=-1$, i.e.~
\begin{equation*}
k_{N-2} - a_{N-1} a_N k_N > R_{N-1},
\end{equation*}
and we use Lemma \ref{exactseq}.

Fix an index $j \leq N-2$. The conditions $\epsilon_j=-1$ and $\epsilon_{j+2}=1$ mean
\begin{equation*}
k_{j-1} + (-1)^{N-j} a_j \dotsm a_N k_N > \widetilde{R}_j  \qquad \textrm{and} \qquad k_{j+1} + (-1)^{N-j} a_{j+2} \dotsm a_N k_N \leq \widetilde{R}_{j+2}.
\end{equation*}
Multiplying the second inequality by $-a_j a_{j+1}$ and adding the first inequality yields
\begin{equation*}
k_{j-1} - a_j a_{j+1} k_{j+1} > \widetilde{R}_j - a_j a_{j+1} \widetilde{R}_{j+2} = R_j
\end{equation*}
and we conclude again with Lemma \ref{exactseq}.
\end{proof}

Proposition \ref{puzzle} guides us to the definition of the subsets $\omega_1,\dotsc,\omega_N$ of $\left\lbrace -1,1\right\rbrace ^N$:
\begin{equation*}
\omega_1 = \left\lbrace (\epsilon_1,\dotsc,\epsilon_N) \in \left\lbrace -1,1\right\rbrace^N \left| \right. \epsilon_1=-1 \textrm{ and }\epsilon_3=1 \right\rbrace
\end{equation*}
and for $j \geq 2$, the set $\omega_j$ is the maximal subset of $\left\lbrace -1,1\right\rbrace^N \setminus (\omega_1 \cup \dotsb \cup \omega_{j-1})$ with the properties
\begin{enumerate}
\item $\epsilon_j=-1$,
\item $\epsilon_{j+2}=1$ (not required for $\omega_{N-1}$ and $\omega_N$),
\item $(\epsilon_1,\dotsc,\epsilon_{j+1},\dotsc,\epsilon_N) \in \omega_j \implies (\epsilon_1,\dotsc,-\epsilon_{j+1},\dotsc,\epsilon_N) \in \omega_j$ (not required for $\omega_N$).
\end{enumerate}
The first and the second points are motivated by Proposition \ref{puzzle} and the third point by the stability condition for the morphism $\alpha_j$ (see Lemma \ref{goesout}).
We also define the set $\omega_\Omega$ such that we have
\begin{equation}\label{decoup}
\left\lbrace -1,1\right\rbrace^N = \omega_\Omega \sqcup \omega_1 \sqcup \dotsb \sqcup \omega_N.
\end{equation}
Observe that if $(\epsilon_1,\dotsc,\epsilon_N)$ is in $\omega_k \subset \left\lbrace -1,1\right\rbrace^N$, then in $\left\lbrace -1,1\right\rbrace^{N+1}$ we get
\begin{itemize}
\item $(\epsilon_1,\dotsc,\epsilon_N,1)$ is in $\omega_k$
\item $(\epsilon_1,\dotsc,\epsilon_N,-1)$ is in $\omega_k$ when $k \neq N-1$ and $k \neq \Omega$
\item $(\epsilon_1,\dotsc,\epsilon_N,-1)$ is in $\omega_{N+1}$ when $k=N-1$ or $k=\Omega$
\item $(1,\dotsc,1,-1)$ is in $\omega_{N}$.
\end{itemize}
We illustrate the behavior of the sets $\omega_j$ under the insertion of $\pm 1$ as
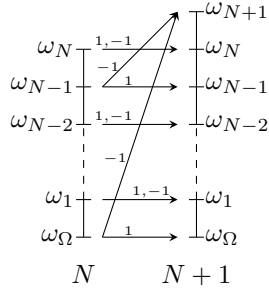
\begin{center}
\begin{tikzpicture}[scale=0.5]
\draw (0,0)--(0,1);
\draw[dashed] (0,1)--(0,3);
\draw (0,3)--(0,5);
\draw (3,0)--(3,1);
\draw[dashed] (3,1)--(3,3);
\draw (3,3)--(3,6);

\draw (-0.2,0)--(0.2,0);
\draw (2.8,0)--(3.2,0);
\draw (-0.2,1)--(0.2,1);
\draw (2.8,1)--(3.2,1);
\draw (-0.2,3)--(0.2,3);
\draw (2.8,3)--(3.2,3);
\draw (-0.2,4)--(0.2,4);
\draw (2.8,4)--(3.2,4);
\draw (-0.2,5)--(0.2,5);
\draw (2.8,5)--(3.2,5);
\draw (2.8,6)--(3.2,6);

\draw (0,0) node[left] {$\omega_\Omega$};
\draw (3,0) node[right] {$\omega_\Omega$};
\draw (0,1) node[left] {$\omega_1$};
\draw (3,1) node[right] {$\omega_1$};
\draw (0,3) node[left] {$\omega_{N-2}$};
\draw (3,3) node[right] {$\omega_{N-2}$};
\draw (0,4) node[left] {$\omega_{N-1}$};
\draw (3,4) node[right] {$\omega_{N-1}$};
\draw (0,5) node[left] {$\omega_N$};
\draw (3,5) node[right] {$\omega_N$};
\draw (3,6) node[right] {$\omega_{N+1}$};

\draw (0,-0.5) node[below] {$N$};
\draw (3,-0.5) node[below] {$N+1$};

\draw[->,>=stealth] (0.5,0)--(2.5,0);
\draw[->,>=stealth] (0.5,1)--(2.5,1);
\draw[->,>=stealth] (0.5,3)--(2.5,3);
\draw[->,>=stealth] (0.5,4)--(2.5,4);
\draw[->,>=stealth] (0.5,5)--(2.5,5);
\draw[->,>=stealth] (0.5,0)--(2.5,6);
\draw[->,>=stealth] (0.5,4)--(2.5,6);

\draw (1.2,-0.2) node[above] {$\scriptscriptstyle 1$};
\draw (1.8,0.8) node[above] {$\scriptscriptstyle 1,-1$};
\draw (0.8,2.8) node[above] {$\scriptscriptstyle 1,-1$};
\draw (1.2,3.8) node[above] {$\scriptscriptstyle 1$};
\draw (0.8,4.8) node[above] {$\scriptscriptstyle 1,-1$};
\draw (1.4,2.1) node[left] {$\scriptscriptstyle -1$};
\draw (1.2,4.5) node[left] {$\scriptscriptstyle -1$};
\end{tikzpicture}
\end{center}

\noindent
and it becomes easy to prove
\begin{equation}\label{omegaO}
\omega_\Omega = \left\lbrace (1,\dotsc,1) \right\rbrace.
\end{equation}

\noindent
At last, we define the subsets

\begin{equation*}
\cQ_j := \bigsqcup\limits_{(\underline{\epsilon}) \in \omega_j} \cQ_{(\underline{\epsilon})} \subset \ZZ^{N+1} \qquad \textrm{with } 1 \leq j \leq N,
\end{equation*}
and $\cQ_\Omega$ stands for the finite subset $\cQ_{(1,\dotsc,1)}$.

\begin{lem}\label{goesout}
For $j \leq N-1$, the subset $\cQ_j$ is stable under the morphisms $\alpha_j$, $\beta$ and $\alpha_k$ with $j+1 < k < N$.
For $j \leq N-2$, the morphism $\alpha_{j+1}$ starts from $\cQ_j$.
\end{lem}

\begin{proof}
The stability condition is clear for the morphisms $\alpha_j$ and $\beta$.
Let us prove it for $\alpha_k$ with $j+1 < k <N$.
Fix $(\epsilon_1,\dotsc,\epsilon_N)$ in $\omega_j$.
By the figure above, we see that $(\epsilon_1,\dotsc,\epsilon_j)$ is also in $\omega_j$ and
\begin{equation*}
(\epsilon_1,\dotsc,\epsilon_j, \pm 1, 1 , \pm 1 , \dotsc, \pm 1) \in \omega_j.
\end{equation*}
Since $k \neq N$, the morphism $\alpha_k$ can only change the sign of $\epsilon_{k+1}$ with $k+1 \geq j+3$, hence the stability condition holds for $\alpha_k$.

Since the morphism $\alpha_{j+1}$ can only change the sign of $\epsilon_{j+2}$ and since $\epsilon_{j+2}=1$ for the set $\cQ_j$,
the morphism $\alpha_{j+1}$ starts from $\cQ_j$.
\end{proof}

Recall that when we fix all the coordinates of $\ZZ^{N+1}$ but $k_j$ and $l$, we get a double complex $(K_\ccE^{\underline{k},l}; \alpha_j, \beta)$.
By Lemma \ref{goesout}, the subset $\cQ_j$ is stable with respect to these morphisms.
Thus those double complexes can be of two kinds:
\begin{enumerate}
\item all the multi-degrees of the sheaves are in $\cQ_j$,
\item all the multi-degrees of the sheaves are in the complement of $\cQ_j$.
\end{enumerate}
In the first case, Proposition \ref{puzzle} implies that the total cohomology over any geometric point vanishes.
As the sheaves are coherent and locally free, the total cohomology over the base scheme $S$ vanishes also, which we write
\begin{equation}\label{vanishingforQj}
H(K_\ccE^{\underline{k},l},\alpha_j + \beta) =0 \quad \textrm{on } \cQ_j.
\end{equation}

\noindent
\textit{Step $4$: spectral sequences.}
For each $j$, we define the sheaves
\begin{equation}\label{inductcomplex}
 (K_\ccE^{\underline{k},l})_j =  \left\lbrace \begin{array}{ll}
0 & \textrm{if } (\underline{k},l) \in \cQ_1 \sqcup \dotsb \sqcup \cQ_j, \\
K_\ccE^{\underline{k},l} & \textrm{otherwise,}
\end{array}\right. 
\end{equation}
yielding the total complexes $K_1,\dotsc,K_N$; the notation $K_0$ stands for $K$.
By Lemma \ref{finitesum}, they are complexes of coherent sheaves.

\begin{lem}\label{spectral2}
Let $(C;f,g)$ be a double complex and $\cE$ be a subset of $\ZZ^2$.
We define a double complex $(C_0;f_0,g_0)$ by
\begin{equation*}
C_0^{r,s} =  \left\lbrace \begin{array}{ll}
0 & \textrm{if } (r,s) \in \cE, \\
C^{r,s} & \textrm{otherwise,}
\end{array}\right. 
\end{equation*}
with the induced morphisms.

If the morphisms $f$ and $g$ start from the subset $\cE$, then the natural injection of sheaves $C_0^{r,s} \hookrightarrow C^{r,s}$ induces an injective morphism $(C_0;f_0,g_0) \hookrightarrow (C;f,g)$ of double complexes.

If the morphisms $f$ and $g$ end in the subset $\cE$, then the natural surjection of sheaves $C^{r,s} \twoheadrightarrow C_0^{r,s}$ induces a surjective morphism $(C;f,g) \twoheadrightarrow (C_0;f_0,g_0)$ of double complexes.

Moreover, in both cases, if, for any $n \in \ZZ$, there is an integer $r_0$ such that
\begin{equation}\label{complete}
C^{n-r,r}=0 \quad \textrm{for any } r \geq r_0
\end{equation}
and if we have
\begin{equation}\label{spectral}
H^r(C^{\bullet,s},f) = H^r(C_0^{\bullet,s},f) \qquad \textrm{for any } (r,s) \in \ZZ^2,
\end{equation}
then the total complexes are quasi-isomorphic.
\end{lem}

\begin{rem}
If the subset $\cE$ is stable by the morphism $f$ and if
\begin{equation*}
H^r(C^{\bullet,s},f) = 0 \qquad \textrm{for any } (r,s) \in \cE,
\end{equation*}
then the equation \eqref{spectral} follows.
\end{rem}

\begin{proof}
It is easy to check the injection and the surjection between the double complexes.
To prove the quasi-isomorphism, we treat the case of the injection; the other case is similar.
Consider the spectral sequences given by the filtration induced by the rows
\begin{equation*}
F^pC^n = \oplus_{r \geq p} C^{n-r,r} \quad \textrm{and} \quad F^pC_0^n = \oplus_{r \geq p} C_0^{n-r,r}.
\end{equation*}
The injection induces a morphism from the spectral sequence $E_\bullet (C_0;f_0,g_0)$ to the spectral sequence $E_\bullet(C;f,g)$, and by \eqref{spectral}, we have
$E_1(C_0;f_0,g_0)=E_1(C;f,g)$.
Then the two spectral sequences coincide,
\begin{equation*}
E_r(C_0;f_0,g_0)=E_r(C;f,g) ~,~~ \textrm{for any } r \geq 1.
\end{equation*}
By \eqref{complete}, the spectral sequences
abut to the total cohomology of the complexes,
\begin{equation*}
H(C_0^\textrm{tot},f_0+g_0) = H(C^\textrm{tot},f+g).
\end{equation*}
\end{proof}

For any $1 \leq j \leq N-1$, Lemma \ref{goesout} implies
\begin{equation*}
(K_j,\alpha_j+\beta,\alpha_{j+1}+\dotsb+\alpha_{N-1}) \hookrightarrow (K_{j-1},\alpha_j+\beta,\alpha_{j+1}+\dotsb+\alpha_{N-1}),
\end{equation*}
with the convention $K_0:=K$.
For any $(\underline{k},l)$ in $\cQ_j$, we have
\begin{equation*}
\begin{array}{lcll}
H((K_\ccE^{\underline{k},l})_{j-1},\alpha_j+\beta) & = & 0 & \textrm{by \eqref{vanishingforQj},} \\
H((K_\ccE^{\underline{k},l})_j,\alpha_j+\beta) & = & 0 & \textrm{by construction.} \\
\end{array}
\end{equation*}
By Lemma \ref{finitesum}, the condition \eqref{complete} is satisfied and we apply Lemma \ref{spectral2} to get
\begin{equation}\label{egalcompl}
H(K_j,\alpha_j+\dotsb+\alpha_{N-1}+\beta) = H(K_{j-1},\alpha_j+\dotsb+\alpha_{N-1}+\beta)~, ~~ \textrm{for } 1 \leq j \leq N-1.
\end{equation}
By the vanishing condition of Theorem \ref{genus0theo}, the vector bundle $A_N$ is zero, thus
\begin{equation*}
\begin{array}{lcl}
\bigoplus_{(\underline{k},l,\ccE) \in \cQ_N \times \cE} K_\ccE^{\underline{k},l} & = & 0, \\
\alpha_N & = & 0, \\
K_{N-1} & = & K_N.
\end{array}
\end{equation*}
Since the morphisms $\alpha_1, \dotsc, \alpha_{N-1}$ and $\beta$ start from the subset $\cQ_\Omega$, there is a surjection
\begin{equation*}
(K_j,\beta+\alpha_{j+1}+\dotsb+\alpha_N,\alpha_j) \twoheadrightarrow (K_N,\beta+\alpha_{j+1}+\dotsb+\alpha_N,\alpha_j),
\end{equation*}
for any $j \leq N-1$.
Consequently, following Lemma \ref{finitesum}, we need
\begin{equation}\label{hyp}
H(K_j,\alpha_{j+1}+\dotsb+\alpha_N+\beta) = H(K_N,\alpha_{j+1}+\dotsb+\alpha_N+\beta)
\end{equation}
to show
\begin{equation}\label{res}
H(K_j,\alpha_j+\dotsb+\alpha_N+\beta) = H(K_N,\alpha_j+\dotsb+\alpha_N+\beta).
\end{equation}
The equality \eqref{egalcompl} for $j=N-1$ coincides with \eqref{hyp} for $j=N-2$ (because $K_N=K_{N-1}$), so that we get \eqref{res} for $j=N-2$.
The following implications are
\begin{displaymath}
\left. \begin{array}{l}
\eqref{res} \textrm{ for } j \\
\eqref{egalcompl} \textrm{ for } j \\
\end{array} \right\}
= (\eqref{hyp} \textrm{ for } j-1) \implies (\eqref{res} \textrm{ for } j-1).
\end{displaymath}
By a decreasing induction on the index $j$, we end with \eqref{hyp} for $j=0$,
\begin{equation*}
H(K,\alpha_1+\dotsb+\alpha_N+\beta) = H(K_N,\alpha_1+\dotsb+\alpha_N+\beta).
\end{equation*}
By construction, $(K_N,\delta)$ is a bounded complex, as it vanishes outside of $\cQ_{(1,\dotsc,1)}$.
Its evaluation in K-theory gives
\begin{equation*}
\begin{split}
\Bigl[ \left( K^{\textrm{tot}} \right)^\bullet,\delta \Bigr] = \left[ \left( K_N^{\textrm{tot}} \right)^\bullet,\delta \right]& = \sum\limits_{(\underline{k},l) \in \cQ_{(1,\dotsc,1)}} (-1)^{k_1+\dotsb+k_N+l} \left[ K^{\underline{k},l} \right]\\
& =  \sum\limits_{(\underline{k},l) \in \cQ_{(1,\dotsc,1)}} (-1)^{k_1+\dotsb+k_N-l} \left[ K^{\underline{k},l} \right] \\
& = \sum\limits_{(\underline{p},\underline{q}) \in \cP(R_1,\dotsc,R_N)} (-1)^{q_1+\dotsb+q_N} \left[ T^{\underline{p},\underline{q}} \right], \\
\end{split}
\end{equation*}
by the change of variables
\begin{equation*}
\Psi(\cP(R_1,\dotsc,R_N))=\cQ_{(1,\dotsc,1)} \times \cE \ \ \textrm{and} \ \ \Bigl( \Psi^{-1}(\cQ_{(1,\dotsc,1)}) \times \cE \Bigr) \cap \cP^+=\cP(R_1,\dotsc,R_N).
\end{equation*}
This ends the proof of Theorem \ref{genus0theo}. \qed

\subsection{Explicit formula for the virtual class}\label{leadgiven}
Let $(T,\delta)$ be a non-degenerate recursive complex with $A_N=0$.
We want to derive from Theorem \ref{genus0theo} a formula suitable to Givental's theory (see Sect.~\ref{computations}).
First, we take the formal power series
\begin{eqnarray}
\mathbf{F}^\mathrm{form}_j(x) & := & \sum_{z \in \NN} \Ch^\vee(\cS^{z}\AB_j) \cdot x^z  \frac{1}{\Td(\AB_j)} \nonumber \\
\mathbf{F}^\mathrm{form}(x_1,\dotsc,x_N) & := & \prod_{j=1}^N \mathbf{F}^\mathrm{form}_j(x_j), \label{formalcvir}
\end{eqnarray}
and by Corollary \ref{maincoro}, we observe that we get the virtual class by the following rules:
\begin{enumerate}
\item choose sufficiently large integers $R_1 \gg \dotsb \gg R_N \gg 1$ (we could estimate lower bounds for these values),
\item develop the formal power series \eqref{formalcvir} with respect to the variables $x_1, \dotsc, x_N$,
\item eliminate every monomial $x_1^{h_1} \dotsm x_N^{h_N}$ with
\begin{equation*}
h_j+\sum_{k=j+1}^N (-a_j) \dotsm (-a_{k-1}) h_k > R_j \quad \textrm{for some index } j,
\end{equation*}
\item evaluate the remaining polynomial at $(x_1,\dotsc,x_N)=(1,\dotsc,1)$.
\end{enumerate}

Fix once for all sufficiently large integers $R_1 \gg \dotsb \gg R_N \gg 1$.
We set the polynomial functions with coefficients in $H^*(S)$
\begin{equation}\label{dfnGj}
\mathbf{G}^{(R_j,\dotsc,R_N)}_j(x_j,\dotsc,x_N):=\sum  \prod_{k=j}^N \Ch^\vee(\cS^{z_k}\AB_k) \cdot x_k^{z_k} \frac{1}{\Td(\AB_k)},
\end{equation}
where the sum is taken over all indexes $(z_j,\dotsc,z_N) \in \NN^{N-j+1}$ such that
\begin{equation*}
z_l+\sum_{k=l+1}^N (-a_l) \dotsm (-a_{k-1}) z_k \leq R_l, \qquad \textrm{for each $j \leq l \leq N$.}
\end{equation*}
We also define the rational functions with coefficients in $H^*(S)$
\begin{equation}\label{cvirx}
\mathbf{F}_j(x) = \exp \Bigl(  \sum\limits_{l \geq 0} s_l(x) \Ch_l(\AB_j) \Bigr) 
\end{equation}
with the functions
\begin{equation}\label{parametresl}
s_l(x) = \left\lbrace 
\begin{split}
&- \ln (1-x)  & \qquad \textrm{if $l=0$,} \\
&\cfrac{B_l(0)}{l} + (-1)^l \sum\limits_{k=1}^l (k-1)! \left( \frac{x}{1-x} \right)^k \gamma(l,k) & \qquad \textrm{if } l \geq 1. \\
\end{split}\right. 
\end{equation}

\noindent
Here, the number $\gamma(l,k)$ is defined by the generating function
\begin{equation*}
\sum_{l \geq 0} \gamma(l,k) \frac{z^l}{l!} := \frac{(e^z-1)^k}{k!}.
\end{equation*}
We notice that $\gamma(l,k)$ vanishes for $k>l$ and that the sum over $l$ in \eqref{cvirx} is finite because $\Ch_l$ vanishes for $l > \dim(S)$.
Finally, we set $\mathbf{F}(x_1,\dotsc,x_N) := \prod_{j=1}^N \mathbf{F}_j(x_j)$.

\begin{lem}\label{formaldev}
For each $j$, the formal series $\mathbf{F}_j^\mathrm{form}$ is a development of the rational function $\mathbf{F}_j$ in the open unit ball $B_{\CC}(0,1)$.
\end{lem}

\begin{proof}
For a vector bundle $E$ of rank $s$ with roots $r_1,\dotsc,r_s$ and for $\left| t \right| <1$, we have
\begin{displaymath}
\begin{split}
\sum_{m\geq 0}(-1)^m\Ch(\Lambda^m(E)) \cdot t^m & =  \prod_{k=1}^s(1-t \cdot e^{r_k}) \\
\sum_{m\geq 0}\Ch(\cS^m(E)) \cdot t^m & =  \frac{1}{\prod_{k=1}^s(1-t \cdot e^{r_k})} \\
\end{split}
\end{displaymath}
and in the ring $\CC \left[ x,y \right] /(y^R)$, we have
\begin{displaymath}
\begin{split}
1-x \cdot e^y  & =  (1-x) \exp \biggl(  - \sum_{l \geq 1} \sum_{k \geq 1} (k-1)! \left(\frac{x}{1-x}\right)^k \gamma(l,k) \frac{y^l}{l!}  \biggr) \\
\frac{y}{1-e^{-y}} & =   \exp \biggl( - \sum_{l \geq 1} \frac{B_l(0)}{l} \frac{y^l}{l!} \biggr).\\
\end{split}
\end{displaymath}
We apply these formulas to $\AB_j$ to conclude.
\end{proof}

For any $\lambda$ in $\CC-\left\lbrace 0,1 \right\rbrace$, we consider the rational function with values in $H^*(S)$
\begin{equation*}
\mathbf{F}(\lambda) := \mathbf{F}(x_1,\dotsc,x_N)\left| \right._{(x_1,\dotsc,x_N)=\bigl(\lambda,\lambda^{-a_1}, \dotsc, \lambda^{(-a_1) \dotsm (-a_{N-1})}\bigr)}.
\end{equation*}

\begin{thm}\label{formula}
Under the hypothesis of Theorem \ref{genus0theo}, the rational function $\mathbf{F}(\lambda)$ is a polynomial in $\lambda,\lambda^{-1}$ and we have
\begin{equation*}
\cvir(T,\delta) = \mathbf{F}(\lambda=1).
\end{equation*}
In particular, we get
\begin{equation}\label{formulelim}
\cvir(T,\delta) = \lim_{\lambda \to 1} \Biggl(\prod_{j=1}^N (1-\lambda_j)^{-\Ch_0 (\AB_j)} \Biggr)  \exp \Biggl( \sum_{j=1}^N \sum_{l \geq 1} s_l(\lambda_j) \Ch_l(\AB_j) \Biggr),
\end{equation}
with $\lambda_j = \lambda^{(-a_1) \dotsm (-a_{j-1})}$.
\end{thm}

\begin{proof}
We proceed by a decreasing induction on $j$ to prove
\begin{equation}\label{recurrence}
\mathbf{F}_j(\lambda_j) \dotsm \mathbf{F}_N(\lambda_N) = \mathbf{G}^{(R_j,\dotsc,R_N)}_j(\lambda_j,\dotsc,\lambda_N), \textrm{ for all $\lambda \in \CC$,}
\end{equation}
with $\lambda_j = \lambda^{(-a_1) \dotsm (-a_{j-1})}$.

\paragraph*{Initialization}
The vanishing property $A_N=0$ implies 
\begin{eqnarray*}
\mathbf{G}^{(R_N)}_N(\lambda_N) & = & \mathbf{F}^\mathrm{form}_N(\lambda_N) \qquad \textrm{ for all $\lambda_N \in \CC$,} \\
& = & \mathbf{F}_N(\lambda_N) \qquad ~~ \textrm{ for all $\left| \lambda_N \right| <1$.}
\end{eqnarray*}
Since the rational function $\mathbf{F}_N$ coincides with the polynomial function $\mathbf{G}^{(R_N)}_N$ in the open unit ball $B_\CC(0,1)$, then they are equal in $\CC$ and 
\eqref{recurrence} is true for $j=N$.

\paragraph*{Heredity}
We assume equation \eqref{recurrence} is true for some index $j+1$.
For any complex number $\lambda$ such that $\left| \lambda_j \right| <1$, we have $\mathbf{F}_j(\lambda_j) = \mathbf{F}^\mathrm{form}_j(\lambda_j)$ by Lemma \ref{formaldev}.
Combining it with the equation \eqref{recurrence} for $j+1$, we obtain
\begin{equation*}
\mathbf{F}_j(\lambda_j) \dotsm \mathbf{F}_N(\lambda_N) = \mathbf{F}^\mathrm{form}_j(\lambda_j) \mathbf{G}^{(R_{j+1},\dotsc,R_N)}_{j+1}(\lambda_{j+1},\dotsc,\lambda_N), \textrm{ when $\left| \lambda_j \right| <1$}.
\end{equation*}
Since the value of $\mathbf{G}^{(R_j,\dotsc,R_N)}_j(1,\dotsc,1)$ does not depend on the choice of sufficiently big integers $R_j \gg \dotsb \gg R_N$ (see Theorem \ref{genus0theo}), then we write, for $h \in \NN^*$,
\begin{equation*}
\mathbf{G}^{(R_j+h,R_{j+1},\dotsc,R_N)}_j(1,\dotsc,1)-\mathbf{G}^{(R_j+h-1,R_{j+1},\dotsc,R_N)}_j(1,\dotsc,1)=0.
\end{equation*}
Consequently, we get
\begin{equation}\label{coeffjj}
\sum  \prod_{k=j}^N \Ch^\vee(\cS^{z_k}\AB_k)  \frac{1}{\Td(\AB_k)} = 0,
\end{equation}
where the sum is taken over all indexes $(z_j,\dotsc,z_N) \in \NN^{N-j+1}$ such that
\begin{eqnarray*}
z_l+\sum_{k=l+1}^N (-a_l) \dotsm (-a_{k-1}) z_k & \leq & R_l \qquad \textrm{for $j+1 \leq l \leq N$,} \\
z_j+\sum_{k=j+1}^N (-a_j) \dotsm (-a_{k-1}) z_k & = & R_j + h.
\end{eqnarray*}
Observe that the left hand side of \eqref{coeffjj} is exactly the coefficient of $\lambda_j^{R_j+h}$ in
\begin{equation*}
\mathbf{F}^\mathrm{form}_j(\lambda_j) \mathbf{G}^{(R_{j+1},\dotsc,R_N)}_{j+1}(\lambda_{j+1},\dotsc,\lambda_N)
\end{equation*}
and since this coefficient is zero, we obtain
\begin{equation*}
\mathbf{F}^\mathrm{form}_j(\lambda_j) \mathbf{G}^{(R_{j+1},\dotsc,R_N)}_{j+1}(\lambda_{j+1},\dotsc,\lambda_N) = \mathbf{G}^{(R_j,\dotsc,R_N)}_j(\lambda_j,\dotsc,\lambda_N).
\end{equation*}
At last, the rational function $\lambda \mapsto \mathbf{F}_j(\lambda_j) \dotsm \mathbf{F}_N(\lambda_N)$ coincides with the rational function $\lambda \mapsto \mathbf{G}^{(R_j,\dotsc,R_N)}_j(\lambda_j,\dotsc,\lambda_N)$ when $\left| \lambda_j \right| <1$. Thus, they are equal for all $\lambda \in \CC^*$ and equation \eqref{recurrence} is true for the index $j$.
\end{proof}

We call virtual degree the integer
\begin{equation*}
\textrm{degvir} := - \sum_{j=1}^N \Ch_0(\AB_j).
\end{equation*}
Since the non-degeneracy condition implies $\textrm{rank}(A_j) \leq \textrm{rank}(B_{j+1})$ for any $j$, then the virtual degree is always a non-negative integer (possibly zero).

For each $l\geq 1$, introduce the polynomial
\begin{equation*}
\widetilde{s}_l(X) := \frac{B_l(0)}{l} + (-1)^l \sum_{k=1}^l (k-1)! X^k \gamma(l,k)
\end{equation*}
of degree $l$ and define a polynomial $\bP_k$ with values in $H^{2k}(S)$ and $\deg \bP_k \leq k$ by
\begin{equation*}
\exp \Biggl( \sum_{j=1}^N \sum_{l \geq 1} \widetilde{s}_l(X_j) \Ch_l(\AB_j) \Biggr) = \sum_{k \geq 0} \bP_k(X_1,\dotsc,X_N).
\end{equation*}

After the change of variables $\epsilon_j := \lambda_j^{-1} - 1$, the degree-$2k$ part of $\mathbf{F}(\lambda)$ is
\begin{equation}\label{partk}
\prod_{j=1}^N \biggl(\frac{\epsilon_j}{1+\epsilon_j}\biggr)^{-\Ch_0\AB_j} \cdot \bP_k(\epsilon_1^{-1},\dotsc,\epsilon_N^{-1}).
\end{equation}
The relation $\lambda_j = \lambda^{(-a_1)\dotsm(-a_{j-1})}$ with $\lambda \to 1$ yields
\begin{equation*}
\epsilon_j^{-1} \underset{\epsilon \rightarrow 0}{=}  \frac{\epsilon^{-1}}{(-a_1)\dotsm(-a_{j-1})} + O(1) \quad \textrm{and} \quad \frac{\epsilon_j}{1+\epsilon_j} \underset{\epsilon \rightarrow 0}{=} (-a_1)\dotsm(-a_{j-1}) \cdot \epsilon + O(\epsilon^2).
\end{equation*}

Let $\mathfrak{d_k} \epsilon^{-\mathfrak{k}}$ be the dominant term in $\bP_k(\epsilon_1^{-1},\dotsc,\epsilon_N^{-1})$, with $\mathfrak{d_k} \in H^{2k}(S)$ non-zero, and write the development of \eqref{partk} near $\epsilon=0$
\begin{equation*}
\mathfrak{h}_k \cdot \mathfrak{d_k} \cdot \epsilon^{\textrm{degvir}-\mathfrak{k}} + O(\epsilon^{\textrm{degvir}-\mathfrak{k}+1}), \qquad \textrm{with $\mathfrak{h}_k \in \CC^*$.}
\end{equation*}
We have $\textrm{degvir} \geq \mathfrak{k}$ because the development of \eqref{partk} near $\epsilon=0$ must converge (see Theorem \ref{formula}) and $\mathfrak{h}_k \cdot \mathfrak{d_k} \neq 0$.
Furthermore, we have already noticed that the degree of the polynomial $\bP_k(X_1,\dotsc,X_N)$ is less than $k$, hence we have $\mathfrak{k} \leq k$.
Finally, when $\textrm{degvir}>\mathfrak{k}$, the degree-$2k$ part of $\mathbf{F}(\lambda)$ tends to $0$ when $\epsilon \to 0$.

\begin{cor}
Under the hypothesis of Theorem \ref{genus0theo}, the class 
$\cvir(T,\delta)$ lies in
\begin{equation*}
\cvir(T,\delta) \in \bigoplus_{k \geq \textrm{degvir}} H^{2k}(S,\QQ).
\end{equation*}
\qed
\end{cor}

\subsection{Computing Polishchuk and Vaintrob's virtual class}\label{appli}
We return to the quantum singularity theory of LG orbifolds $(W,\textrm{Aut}(W)$, with an invertible polynomial $W$, and we will adapt Theorem \ref{formula} to compute the virtual class of the two-periodic complex
\begin{equation*}
(T,\delta):=p^\textrm{naive}_*(\PV \otimes \kK(\rR_{\overline{\gamma}})) \quad \textrm{over $S$ (see Sect.~\ref{cvirinvertible})}.
\end{equation*}
Here, we fix an element $\overline{\gamma} \in \left( \textrm{Aut}(W) \right)^n$, a family $\pi \colon \cC \rightarrow S$ of $W$-spin curves of genus zero with sections $\sigma_1,\dotsc,\sigma_n$ of type $\overline{\gamma}$ over a smooth base scheme $S$ and an admissible decoration $\rR_{\overline{\gamma}}$ of the graph $\Gamma_{W_{\overline{\gamma}}}$.

As the invertible polynomial $W$ is not just a loop polynomial, the two-periodic complex $(T,\delta)$ is not recursive. We have to modify this complex and to assume extra conditions to apply Theorem \ref{formula}; we proceed in three steps.
First, we construct another two-periodic complex $(T^\rR,\delta^\rR)$.
Then we prove
\begin{equation*}
\cvirPV(T^\rR,\delta^\rR) = \cvirPV(T,\delta).
\end{equation*}
Finally we assume extra conditions (see Theorem \ref{main}), so that $(T^\rR,\delta^\rR)$ turns into a non-degenerate and recursive complex with vanishing condition \eqref{qqconcave}; we conclude with Theorem \ref{formula}.

\noindent
\textit{Step $1$:}
Recall that the two-periodic complex $(T,\delta)$ is built on the vector bundles $A_1, \dotsc, A_N$ and $\widetilde{B}_1, \dotsc, \widetilde{B}_N$ on the scheme $S$ and observe that $[A_j \rightarrow \widetilde{B}_j]$ is quasi-isomorphic to $R\pi_*(\cL^\rR_j)$ with
\begin{equation}\label{fibreLR}
\cL^\rR_j:=\cL_j \Biggl(- \sum_{(\sigma_i,x_j) \in \rR_{\overline{\gamma}}} \sigma_i \Biggr).
\end{equation}
As each marked point $\sigma_i$ in $\bB_{\overline{\gamma}_j}$ is in $\rR_{\overline{\gamma}_j} \sqcup \rR_{\overline{\gamma}_{t(j)}}$, the isomorphism \eqref{spinstruct} induces
\begin{equation}\label{ssss}
{\cL^\rR_j}^{\otimes a_j} \otimes \cL^\rR_{t(j)} \hookrightarrow \omega_{\cC}.
\end{equation}
Thus the construction of Polishchuk--Vaintrob \cite[Sect.~4.2]{Polish1} gives vector bundles 
\begin{equation*}
A^\rR := A_1^\rR \oplus \dotsb \oplus A_N^\rR \quad \textrm{and} \quad B^\rR := B_1^\rR \oplus \dotsb \oplus B_N^\rR,
\end{equation*}
with morphisms
\begin{equation*}
\alpha^\rR := \alpha_1^\rR + \dotsb + \alpha_N^\rR \quad \textrm{and} \quad \beta^\rR := \beta_1^\rR + \dotsc + \beta_N^\rR,
\end{equation*}
yielding a Koszul matrix factorization $\left\lbrace - \alpha^\rR , \beta^\rR \right\rbrace$ of potential zero on the total space of $A^\rR$.
The naive push-forward gives a two-periodic complex
\begin{equation*}
p^\textrm{naive}_* \left\lbrace - \alpha^\rR , \beta^\rR \right\rbrace = (T^\rR,\delta^\rR) \quad \textrm{on $S$.}
\end{equation*}

Compare the construction of $(T^\rR,\delta^\rR)$ with the construction of $(T,\delta)$ (see \cite[Section 4.2, Step 2]{Polish1}) and observe that
\begin{equation}\label{comparison}
B_j^\rR = B_j \quad \textrm{and} \quad A_j^\rR = \ker Z^{\rR}_j,
\end{equation}
where $Z^{\rR}_j \colon A_j \rightarrow \cO^{\rR_j}$.
Then the morphisms
\begin{equation}\label{alphabeta3}
\begin{array}{ccl}
(\alpha')^\rR_j & \colon & \cS^{a_j} A^\rR_j \rightarrow  B_{t(j)}^\vee, \\
(\alpha'')^\rR_j & \colon & \cS^{a_j-1} A^\rR_j \otimes A^\rR_{t(j)} \rightarrow B_j^\vee, \\
\beta^\rR_j & \colon & A^\rR_j \rightarrow B_j,
\end{array}
\end{equation}
are naturally induced by the morphisms $\alpha'_1,\dotsc, \alpha'_N$, $\alpha''_1,\dotsc, \alpha''_N$ and $\beta_1,\dotsc, \beta_N$.
To fix ideas, over a geometric point $s \in S$, we have
\begin{equation}\label{Serreduality3}
\begin{split}
(\alpha')^\rR_{j,s} \colon & \mathrm{Sym}^{a_j} H^0(\cC_s,\cL^\rR_{j,s}) \rightarrow H^1(\cC_s,\cL^\rR_{t(j),s})^\vee \qquad \textrm{and} \\
(\alpha'')^\rR_{j,s} \colon & \mathrm{Sym}^{a_j-1}(H^0(\cC_s,\cL^\rR_{j,s})) \otimes H^0(\cC_s,\cL^\rR_{t(j),s}) \rightarrow H^1(\cC_s,\cL^\rR_{j,s})^\vee
\end{split}
\end{equation}
induced by \eqref{ssss} (and similarly to \eqref{Serreduality}) and the morphism $Z_{j,s}^{\rR} \colon H^0(\cC,\cL_{j,s}) \rightarrow \cO^{\rR}_{j,s}$ comes from the exact sequence
\begin{equation*}
0 \rightarrow \cL^\rR_{j,s} \rightarrow \cL_{j,s} \rightarrow \cL_{j,s}\left| \right._{\sum_{(\sigma_i,x_j) \in \rR_{\overline{\gamma}}} \sigma_j} \rightarrow 0.
\end{equation*}

\noindent
\textit{Step $2$:}
We go back to the base scheme $S$ and observe
\begin{equation}\label{twotermegal}
R\pi_*(\cL^\rR_j) = [A_j \rightarrow \widetilde{B}_j] = [ A^\rR_j \rightarrow B_j ] \quad \textrm{in } \mathcal{D}(S).
\end{equation}
Denote by $X^\rR$ the total space of $A^\rR$ and by $i \colon X^\rR \hookrightarrow X$ the inclusion in $X$.
The section $\widetilde{\beta}$ induces the regular section $Z^*\mathfrak{b}_{\rR_{\overline{\gamma}}}$ of the sheaf
\begin{equation*}
p^*\widetilde{B}/p^*B \simeq Z^*(\cO^{\rR_{\overline{\gamma}}})^\vee,
\end{equation*}
whose zero locus is $X^\rR$.
The sections $\widetilde{\alpha}$ and $\widetilde{\beta}$ induce $\alpha^\rR$ and $\beta^\rR$ on $X^\rR$.
The zero loci of $\{\alpha^\rR,\beta^\rR\}$ and $\left\lbrace \alpha,\beta \right\rbrace $ coincide with the zero section $S$ in $X^\rR \subset X$, so they are proper.
By \cite[Proposition 4.3.1]{Polish1}, we obtain $\left\lbrace \alpha , \beta \right\rbrace \simeq i_* \left\lbrace \alpha^\rR , \beta^\rR \right\rbrace$, and via the naive push-forward,
\begin{equation}\label{qis}
(T,\delta) \simeq (T^\rR,\delta^\rR).
\end{equation}

\noindent
\textit{Step $3$:}
According to the FJRW terminology, we say that a variable $x_j$ is concave for the decoration $\rR_{\overline{\gamma}}$ if
\begin{equation*}
H^0(\cC,\cL^\rR_j)=0,
\end{equation*}
for every $W$-spin curve $\cC$ of genus-zero and type $\overline{\gamma}$,
with $\cL^\rR_j$ defined by \eqref{fibreLR}.

By \cite[Proposition 3.1]{LG/CY}, if
\begin{equation}\label{quasiconcave}
w_j \left| \right. d ~,~~ \gamma_j(i) \in \left\langle e^{2\ci \pi \fq_j} \right\rangle ~~ \forall i \quad \textrm{and} \quad \rR_{\overline{\gamma}_j} = \bB_{\overline{\gamma}_j},
\end{equation}
i.e.~if the line bundle $\cL_j$ is a root of $\omega_\textrm{log}$ and all the broad marked points are crossed for it,
then the variable $x_j$ is concave for the decoration $\rR_{\overline{\gamma}}$.

\begin{exa*}
For each index $j$ satisfying $t(j)=j$, the variable $x_j$ is concave for every admissible decoration.
Indeed, the monomial $x_j^{a_j}x_{t(j)}$ of $W$ gives $w_j \left| \right. d$, and by definition of an admissible decoration, any vertex followed by itself is crossed.
\end{exa*}

\begin{thm}\label{main}
Consider an invertible polynomial $W$, an element $\overline{\gamma} \in \textrm{Aut}(W)^n$ and admissible decorations $\rR_{\gamma(1)}, \dotsc, \rR_{\gamma(n)}$ of the graph $\Gamma_W$.
Assume that each connected component of $\Gamma_W$ contains a concave variable for $\rR_{\overline{\gamma}}:=\rR_{\gamma(1)} \sqcup \dotsb \sqcup \rR_{\gamma(n)}$. Then this is unambiguous to define
\begin{equation}\label{deflambda}
\begin{array}{lcll}
\lambda_{t(j)} & = & \lambda_j^{-a_j} & \textrm{if $x_j$ is non-concave for $\rR_{\overline{\gamma}}$, with } a_j:=f_W(v_j), \\
\lambda_j & = & \lambda & \textrm{for every remaining index $j$,}
\end{array}
\end{equation}
and the evaluation of Polishchuk--Vaintrob's virtual class \eqref{virtualcomplex} in genus zero on the state $e(\rR_{\overline{\gamma}})$ equals
\begin{equation}\label{formulelim2}
\begin{split}
\cvirPV(e(\rR_{\overline{\gamma}})) & =  \lim_{\lambda \to 1}  \prod_{j=1}^N \fc_{\lambda_j}(-R \pi_*(\cL_j^\rR)) \\
 & =  \lim_{\lambda \to 1}  \prod_{j=1}^N (1-\lambda_j)^{r_j} \fc_{\lambda_j}(-R \pi_*(\cL_j)) \\
 & =  \lim_{\lambda \to 1}  \Biggl(\prod_{j=1}^N (1-\lambda_j)^{-\Ch_0 (R \pi_*(\cL_j))+r_j} \Biggr) \\
 & \qquad \quad \cdot \exp \Biggl( \sum_{j=1}^N \sum_{l \geq 1} s_l(\lambda_j) \Ch_l(R \pi_*(\cL_j)) \Biggr)\\
 & \in \bigoplus_{k \geq \mathrm{degvir}} H^{2k}(\sS_{0,n}, \QQ),
\end{split}
\end{equation}
with the characteristic class $\fc$ defined by \eqref{charclass2}, the function $s_l(x)$ by \eqref{parametresl} and the virtual degree by
\begin{equation*}
\mathrm{degvir}  =  \sum_{j=1}^N -\Ch_0 (R \pi_*(\cL_j))+r_j \quad \in \NN, \qquad \textrm{with } r_j := \mathrm{card } (\rR_{\overline{\gamma}_j}).
\end{equation*}
\end{thm}

\begin{proof}
First, we observe
\begin{equation*}
\Ch_0 (R \pi_*(\cL_j))-r_j = \Ch_0 (R \pi_*(\cL_j^\rR)).
\end{equation*}
Then we prove that $(T^\rR,\delta^\rR)$ is a non-degenerate recursive complex with vanishing condition.
Indeed, let us alter the graph $\Gamma_W$ as follows.
For each concave variable $x_j$, erase the arrow from $v_j$ to $v_{t(j)}$.
As every connected component contains a concave variable, we get a disjoint union of oriented graphs, which are lines from a tail to a head (in the direction of arrows).
Once we draw an arrow from the head to the tail, each oriented graph corresponds to a recursive complex with vanishing condition, as follows.
The vertex $v_j$ corresponds to the vector bundles $A^\rR_j$ and $B_j$, with the morphism $\beta^\rR_j$; the arrow from $v_j$ to $v_{t(j)}$ corresponds to the morphisms $(\alpha')^\rR_j$ and $(\alpha'')^\rR_j$; the vector bundle $A^\rR_k$ which corresponds to the head is zero.
With this representation, this is clear that $(T^\rR,\delta^\rR)$ decomposes as a tensor product of recursive complexes with vanishing condition, corresponding to the connected components of the modified graph.
If we prove the non-degeneracy conditions for $(T^\rR,\delta^\rR)$, then we apply Theorem \ref{formula} for each term of the tensor product, and the product of the resulting formulas gives the virtual class $\cvirPV(T^\rR,\delta^\rR)$.

Over a geometric point $s \in S$, the morphism $(\alpha')^\rR_{j,s}$ is given by
\begin{equation*}
\mathrm{Sym}^{a_j}H^0(\cC_s,\cL^\rR_{j,s}) \rightarrow H^0(\cC_s,(\cL^\rR_{j,s})^{a_j}) \simeq H^0(\cC_s,\omega_{\cC_s} \otimes (\cL^\rR_{t(j),s})^\vee) \simeq H^1(\cC_s,\cL^\rR_{t(j),s})^\vee,
\end{equation*}
and we see that the morphism $(\alpha')^\rR_j$ is non-degenerate (see Definition \ref{non-degenera}).
\end{proof}

\begin{rem}
Theorem \ref{main}, together with the expression of the Chern character of $R\pi_*\cL_j$ (see \cite[Theorem 1.1.1]{Chiodo1}) and a modified version of the algorithm \cite{Faber}, leads to a computer program expressing Polishchuk--Vaintrob's virtual class in terms of psi-classes and boundary terms, and giving numerical values for the invariants of the cohomological field theory \eqref{invariantsPV}.
When at least one of the decorations at the marked points is not balanced, the formula of Theorem \ref{main} vanishes. 
\end{rem}

\begin{rem}\label{3pointt}
The factorization and index zero properties of \cite[Sect.~5]{Polish1} constitute applications of Theorem \ref{main}.
For instance, we show that every three-point correlator is a product of terms
\begin{equation}\label{3pointcor}
\frac{1-\lambda_{t(j)}}{1-\lambda_j} \underset{\lambda \to 1}{\longrightarrow} -a_j
\end{equation}
and possibly of terms $1-\lambda_j \underset{\lambda \to 1}{\longrightarrow} 0$, happening exactly when $\textrm{degvir}>0$.
Hence, the virtual class $(\cvirPV)_{0,3}$ vanishes if and only if its degree $\textrm{degvir}$ is non-zero. 
\end{rem}

\begin{exa*}
By \cite[Lemma 6.1.1]{Polish1}, the bilinear pairing of the state space satisfies
\begin{equation*}
\left( e(\rR_\gamma) , e(\rR'_{\gamma'}) \right) = \langle e(\rR_\gamma) , e(\rR'_{\gamma'}) , e_\grj  \rangle^\textrm{PV}_{0,3} ,
\end{equation*}
where $e_\grj$ is the unit element corresponding to the unique admissible decoration for the grading element $\grj$, defined in \eqref{gradingelement} (this decoration is empty).
The component $\sS_{0,3}(\gamma,\gamma',\grj)$ is non-empty only if $\gamma'=\gamma^{-1}$, then for any index $j$ we have
\begin{equation*}
- \Ch_0 (R \pi_*(\cL_j)) = \left\lbrace \begin{array}{ll}
-1 & \textrm{if } \gamma_j=1, \\
0 & \textrm{otherwise.}
\end{array}\right. 
\end{equation*}
Finally, we recover \eqref{pairingcomput}.
\end{exa*}

\begin{exa*}
Consider the polynomial $W=x_1^2 x_2+ x_2^3 x_3+x_3^5 x_4 + x_4^{10}x_5+x_5^{11}$ with weights $(4,3,2,1,1)$ and degree $11$.
We let as an exercise for the reader to compute
\begin{equation}\label{threepoint}
\langle e_{\grj^3},e_{\grj^3},e_{\grj^6} \rangle=-2
\textrm{ and }
\langle e_{\grj^2}^4e_{\grj^6} \rangle = -\frac{2}{121},
\end{equation}
where we lighten unambiguously notation as $e_\gamma := e(\rR_\gamma)$ and where $\grj$ is the grading element defined in \eqref{gradingelement}.
We stress that concavity fails in these two situations.
\end{exa*}

\begin{rem}
A virtual class with a broad entry can be non-zero.
For instance, consider the $D_5$-singularity $x_1^2x_2+x_2^4$ with weights $(3,2)$ and degree $8$; we leave as an exercise the computation
\begin{equation*}
\cvirPV (e_{\grj^0}, e_{\grj^2}, e_{\grj^3}, e_{\grj^3}, e_{\grj^3})_{0,5} = \textrm{o}^*\psi_2 \in H^2(\mathcal{S}_{0,5}).
\end{equation*}
\end{rem}


\subsection{Compatibility of virtual classes for invertible polynomials}\label{compatsec}
We prove in this section the compatibility between Polishchuk--Vaintrob's class and FJRW virtual class for (almost) every invertible polynomials in every genus.
In Sect.~\ref{computations}, this compatibility applied to chain polynomials is used to deduce mirror symmetry for FJRW theory from Theorem \ref{main}.

\begin{thm}\label{compat}
Let $W$ be an invertible polynomial together with its maximal group $\mathrm{Aut}(W)$.
Assume that no monomials of $W$ are of the form $x^ay+y^2$ or $x^ay+y^2x$.
Then, we have an isomorphism $\Phi \colon \st \rightarrow \st$ rescaling the broad sector (each narrow state is invariant under $\Phi$ and each broad sector is stable under $\Phi$) such that
\begin{equation}\label{compatequal}
\cvir^\mathrm{FJRW} (u_1,\dotsc,u_n)_{g,n} = \cvir^\mathrm{PV} (\Phi(u_1),\dotsc,\Phi(u_n))_{g,n}
\end{equation}
for every elements $u_1, \dotsc, u_n$ of the state space and for every genus $g$.
This isomorphism preserves the pairing (and the grading), i.e.~
\begin{equation}\label{compatpair}
\Phi^\textrm{T} \cdot \eta \cdot \Phi = \eta,
\end{equation}
where $\eta$ is the inverse matrix of the pairing.
\end{thm}

We sketch the proof of Theorem \ref{compat}:
\begin{enumerate}
\item By the property called ``sums of singularities'' (see \cite[Theorem 4.1.8 (8)]{FJRW} and \cite[Theorem 5.8.1]{Polish1}), we treat separately the cases of Fermat, chain, and loop polynomials.
\item By \cite[Theorem 1.2]{Li2}, equation \eqref{compatequal} is true when $u_1, \dotsc, u_n$ are narrow states. This proves the Fermat case and the loop case with an odd number of variables.
\item For each broad state $u_0$, we look for two narrow states $u_1$ and $u_2$ such that the virtual class $\cvir^\mathrm{PV}(u_0,u_1,u_2)_{0,3}$ is non-zero.
\item Using the factorization property (see \cite[Theorem 4.1.8 (6)]{FJRW}, \cite[Sect.~5.3]{Polish1}) of a cohomological field theory, we define a suitable morphism $\Phi$.
\item Using again the factorization property, we prove that $\Phi$ is an isomorphism and respects the pairing.
\end{enumerate}

For a chain polynomial $W$, there is a unique admissible decoration $\rR_\gamma$ attached to a given automorphism $\gamma \in \Aut(W)$; we lighten notation as $e_\gamma := e(\rR_\gamma)$.
Introduce the diagonal automorphism $\gru := \textrm{diag}(\exp 2 \ci \pi u_1,\dotsc,\exp 2 \ci \pi u_N )$ with
\begin{equation*}
u_j := \frac{1}{(-a_j) \dotsm (-a_N)}.
\end{equation*}
Fixing a broad state $e_\lambda \neq 0$, we check that the states $e_{\gru \cdot \lambda^{-1}}$ and $e_{\grj \cdot \gru^{-1}}$ are narrow.
By Remark \ref{3pointt}, we prove that the genus-zero virtual class
$\cvir^\mathrm{PV}(e_\lambda, e_{\gru \cdot \lambda^{-1}} , e_{\grj \cdot \gru^{-1}})_{0,3}$ is non-zero, since its virtual degree vanishes.
Using the factorization property, the corresponding FJRW class $\cvir^\mathrm{FJRW}(e_\lambda, e_{\gru \cdot \lambda^{-1}} , e_{\grj \cdot \gru^{-1}})_{0,3}$ is also non-zero.
We define
\begin{equation}\label{constc}
\Phi(e_\gamma) = \left\lbrace 
\begin{array}{ll}
e_\gamma & \textrm{if $e_\gamma$ is a narrow state} \\
\frac{\cvir^\mathrm{PV} (e_\gamma, e_{\gru \cdot \gamma^{-1}} , e_{\grj \cdot \gru^{-1}})_{0,3}}{\cvir^\mathrm{FJRW} (e_\gamma, e_{\gru \cdot \gamma^{-1}} , e_{\grj \cdot \gru^{-1}})_{0,3}}  e_\gamma & \textrm{if $e_\gamma$ is a broad state.} \\
\end{array}\right. 
\end{equation}
Using again the factorization property, we prove equalities \eqref{compatequal} and \eqref{compatpair}.

For a loop polynomial $W$ with an even number of variables, there are two broad states in the basis $\left\lbrace e(\rR_\gamma)\right\rbrace$, coming from the identity automorphism $\gamma=1$:
\begin{equation*}
\begin{array}{lclclcl}
e_- & := & e(\rR^-_1) \quad  & \textrm{ with } & \rR^-_1 & := & \left\lbrace x_{2j+1} \right\rbrace_j,  \\
e_+ & := & e(\rR^+_1) \quad & \textrm{ with } & \rR^+_1 & := & \left\lbrace x_{2j} \right\rbrace_j.
\end{array}
\end{equation*}
Introduce two diagonal automorphisms $\gru := \textrm{diag}(\exp 2 \ci \pi u_1,\dotsc,\exp 2 \ci \pi u_N )$ and $\grv := \textrm{diag}(\exp 2 \ci \pi v_1,\dotsc,\exp 2 \ci \pi v_N)$ with
\begin{equation*}
u_j  :=  \left\lbrace 
\begin{array}{ll}
\frac{1}{a_1 \dotsm a_N-1} & \textrm{ if } j=1 \\
\frac{(-a_1) \dotsm (-a_{j-1})}{a_1 \dotsm a_N-1} & \textrm{ if } 2 \leq j \leq N \\
\end{array} \right., \quad 
v_j := \left\lbrace 
\begin{array}{ll}
\frac{(-a_2) \dotsm (-a_N)}{a_1 \dotsm a_N-1} & \textrm{ if } j=1 \\
\frac{1}{a_1 \dotsm a_N-1} & \textrm{ if } j=2 \\
\frac{(-a_2) \dotsm (-a_{j-1})}{a_1 \dotsm a_N-1} & \textrm{ if } 3 \leq j \leq N \\
\end{array} \right. .
\end{equation*}
We consider the matrices
\begin{equation*}
B^\mathrm{PV} := \left( \begin{array}{cc}
\cvirPV(e_-,e_\gru,e_{\grj \cdot \gru^{-1}})_{0,3} & \cvirPV(e_-,e_\grv,e_{\grj \cdot \grv^{-1}})_{0,3} \\
\cvirPV(e_+,e_\gru,e_{\grj \cdot \gru^{-1}})_{0,3} & \cvirPV(e_+,e_\grv,e_{\grj \cdot \grv^{-1}})_{0,3} \\
\end{array}\right)
\end{equation*}
and $B^\mathrm{FJRW}$, where we replace each PV by FJRW.
We check that the states $e_\gru, e_{\grj \cdot \gru^{-1}}$, $e_\grv$, and $e_{\grj \cdot \grv^{-1}}$ are narrow and that the matrix $B^\mathrm{PV}$ is invertible; precisely, we have
\begin{equation*}
B^\mathrm{PV} := \left( \begin{array}{cc}
(-a_1) (-a_3) (-a_5) \dotsm (-a_{N-1}) & 1 \\
1 & (-a_2) (-a_4) (-a_6) \dotsm (-a_N) \\
\end{array}\right).
\end{equation*}
Using the factorization property, the matrix $B^\mathrm{FJRW}$ is also invertible and we define
\begin{equation*}
\Phi(u) := \left\lbrace 
\begin{array}{ll}
u & \textrm{ if $u$ is narrow,} \\
\eta \cdot B^\mathrm{PV} \cdot (B^\mathrm{FJRW})^{-1} \cdot \eta^{-1} (u) & \textrm{ if $u$ is broad,}
\end{array}\right. 
\end{equation*}
where the matrices $B^\mathrm{PV}$ and $B^\mathrm{FJRW}$ are written in the basis $(e_-,e_+)$.
Using again the factorization property, we prove equalities \eqref{compatequal} and \eqref{compatpair}.

\section{Computing the J-function}\label{computations}
In this last section, we focus on genus-zero FJRW theory of chain polynomials with maximal group of symmetries. Thus, we can always apply Theorem \ref{main} about Polishchuk--Vaintrob's class, and by Theorem \ref{compat}, we know all FJRW invariants, up to a rescaling of the broad sectors.
In Sect.~\ref{4.2}, we present Givental's theory, in particular we define the Lagrangian cone and the J-function.
In Sect.~\ref{4.3}, we construct a symplectic operator which connects the so-called untwisted theory to FJRW theory.
In Sect.~\ref{4.4}, we apply this operator to a function lying on the untwisted Lagrangian cone and we obtain a big I-function lying on the FJRW Lagrangian cone, see Theorem \ref{bigItheorem}.
Specializing the argument of the big I-function to a line in the state space and restricting to polynomials of Calabi--Yau type, we get a solution of the Picard--Fuchs equation of the mirror polynomial.
Up to a change of variables and a rescaling of this solution, we end with a part of the small J-function of FJRW theory 
and we prove mirror symmetry Theorem \ref{thmfinal}.

\subsection{Givental's formalism}\label{4.2}
What follows is summary of Givental's theory \cite{Givental}, a way to encode genus-zero invariants in a Lagrangian submanifold of a symplectic space.
We restrict here on the aspects of the theory which we need, but the formalism is more powerful and involves also higher genus invariants.
The progress of the text up to the end is a generalization of \cite[Sect.~4]{LG/CYquintique} or \cite[Sect.~3]{LG/CY}, following the method used in \cite{Coates3} for twisted Gromov--Witten invariants.

From now on to the end, we work on genus-zero FJRW theory with a Calabi--Yau chain polynomial $W$,
\begin{equation*}
W = x_1^{a_1}x_2 + \dotsb + x_{N-1}^{a_{N-1}} x_N + x_N^{a_N+1},
\end{equation*}
with weights $(w_1,\dotsc,w_N)$ and degree
\begin{equation}\label{CYcondd}
d=w_1+ \dotsb + w_N \qquad \textrm{(Calabi--Yau condition)}.
\end{equation}
Observe that for a given automorphism $\gamma$, there is a unique admissible decoration $\rR_\gamma$; we lighten notation as $e_\gamma:=e(\rR_\gamma)$.
According to Theorem \ref{compat}, there is a constant $c_\gamma \in \CC^*$ defined by \eqref{constc} such that $\Phi(e_\gamma) = c_\gamma \cdot e_\gamma$.
We recall that $c_\gamma \cdot c_{\gamma^{-1}} = 1$ and that for a narrow state $e_\gamma$, the constant is $c_\gamma=1$.
We set
\begin{equation*}
\e_\gamma := \frac{1}{c_\gamma} \cdot e_\gamma.
\end{equation*}

The Givental space $\cH$ is the symplectic vector space
\begin{equation*}
\cH := \st \left[ z \right] [\![  z^{-1} ]\!] 
\end{equation*}
of Laurent series in $z^{-1}$ with coefficients in the state space, and we use the bilinear pairing of $\st$ to equip this space with the symplectic form
\begin{equation*}
\Omega (\sum_n a_n z^n , \sum_m b_m z^m) = \sum_{p+q=-1} (-1)^p (a_p,b_q).
\end{equation*}
The sum on the right is finite because the number of positive powers of $z$ is finite for an element of $\cH$.
This symplectic form induces a natural polarization
\begin{equation*}
\cH \simeq \cH^+ \oplus \cH^- \qquad \textrm{where } \cH^+:=\st \left[ z \right] \textrm{ and } \cH^-:=z^{-1}\st [\![ z^{-1} ]\!],
\end{equation*}
and identifies $\cH$ with the cotangent bundle of $\cH^+$.
If the decoration $\rR_\gamma$ is balanced, we denote the dual element of $\e_\gamma$ by $\e^\gamma$ and we recall that 
\begin{equation*}
\e^{\gamma} = \biggl( \prod_{\substack{j \left| \right. \gamma_j = 1 \\ N-j \textrm{ odd}}} \frac{-1}{a_j} \biggr) \cdot \e_{\gamma^{-1}}.
\end{equation*}
If the decoration $\rR_\gamma$ is not balanced, then $\e_\gamma=0$ is not a vector of the basis of $\st$.
The coordinates $\{q^\alpha_k\}$ and $\{ p^\beta_l\}$ of the basis $\{\e_\alpha z^k\}_{\alpha, k \geq 0}$ of $\cH^+$ and the basis $\{\e^\beta (-z)^{-1-l}\}_{\beta, l \geq 0}$ of $\cH^-$ are Darboux coordinates,
\begin{equation*}
\Omega = dp \wedge dq.
\end{equation*}
The genus-zero invariants \eqref{invariants} of FJRW theory generates a function $\cF_0 \colon \cH^+ \rightarrow \CC$
\begin{equation*}
\cF_0:=\sum_{n \geq 0} ~~ \sum_{b_1,\dotsc,b_n \geq 0} ~~ \sum_{\gamma_1,\dotsc,\gamma_n} \langle \tau_{b_1}(\e_{\gamma_1}),\dotsc,\tau_{b_n}(\e_{\gamma_n}) \rangle^\textrm{FJRW}_{0,n} \cfrac{(t_{b_1}^{\gamma_1}) \dotsm (t_{b_n}^{\gamma_n})}{n!},
\end{equation*}
expressed in coordinates $t(z):=q(z)+z$ (this is called dilaton shift), and the graph $\cL$ of its differential
\begin{equation*}
\cL := \left\lbrace (p,q) \textrm{ such that } p = d_q \cF_0\right\rbrace \subset \cH
\end{equation*}
is an exact Lagrangian sub-variety of $\cH$ on which lies the J-function $J \colon \st \rightarrow \cL$,
\begin{eqnarray*}
J(h,-z) = -z \e_\grj + h + 
\sum_{\substack{n \geq 0 \\ l \geq 0}} ~~ \sum_{\gamma_1,\dotsc,\gamma_n,\widetilde{\gamma}} \langle \e_{\gamma_1},\dotsc,\e_{\gamma_n},\tau_{l}(\e_{\widetilde{\gamma}})  \rangle^\textrm{FJRW}_{0,n+1} \cfrac{h^{\gamma_1}\dotsm h^{\gamma_n}}{n!(-z)^{l+1}} \e^{\widetilde{\gamma}},
\end{eqnarray*}
with $h = \sum h^\gamma \e_\gamma$ and $\e_\grj$ the unit element of $\st$ corresponding to the grading element defined by \eqref{gradingelement}.
The FJRW invariants satisfy the string and dilaton equations and the topological recursion relations stated in \cite{FJRW}. 
By \cite[Theorem 1]{Givental}, these relations mean geometrically that $\cL$ is a Lagrangian cone in $\cH$ and that for any point $p \in \cL$, the tangent space satisfies $T_p\cL \cap \cL = z T_p\cL$.
Furthermore, the J-function spans the Lagrangian cone and has the property $J(h,-z)= -z \e_\grj + h + \textrm{o} (z^{-1})$.
The J-function is determined as the unique function on $\cL$ with this property.


\subsection{A symplectic operator}\label{4.3}
The behavior of the J-function is essential for an understanding of the theory in genus-zero.
Any cohomological field theory with a state space $\st$ equipped with a different pairing is encoded in a Lagrangian cone of $\cH$ and distinct Lagrangian cones are related by symplectic operators.
The operator $\Delta$ of Theorem \ref{operateursympl} links an untwisted theory, whose Lagrangian cone is well-understood, to a twisted theory, whose evaluation in special values yields the FJRW theory.

Consider $\overline{\gamma} \in \Aut(W)^n$ and
set the formal virtual class twisted by variables $(s_l^j)_{l \geq 0,j}$ as
\begin{equation*}
\e_{\overline{\gamma}}:=\e_{\gamma_1} \otimes \dotsm \otimes \e_{\gamma_n} \mapsto \exp \biggl( \sum_{j=1}^N \sum_{l \geq 0} s^j_l \Ch_l\bigl(R \pi_*\cL^\rR_j\bigr) \biggr),
\end{equation*}
where the line bundle $\cL^\rR_j$ defined by \eqref{fibreLR} takes the form
\begin{displaymath}
\cL^\rR_j = \left\lbrace 
\begin{array}{ll}
\cL_j(-\sigma_1-\dotsb-\sigma_n) & \textrm{if } N-j \textrm{ is even}, \\
\cL_j & \textrm{if } N-j \textrm{ is odd}. \\
\end{array}\right. 
\end{displaymath}
This gives a twisted cohomological field theory with a Lagrangian cone $\cL_{\text{tw}}$, 
and the twisted bilinear pairing is given by
\begin{equation*}
\left( \e_{\gamma} , \e_{\gamma^{-1}} \right)_\textrm{tw} := \prod_{x_j \in \rR_{\gamma}} \exp( - s_0^j) \cdot \prod_{x_j \in \bB_\gamma \backslash \rR_{\gamma}} \exp( s_0^j).
\end{equation*}
The Lagrangian cone $\cL_{\text{un}} \subset \cH$ of the untwisted theory arises for the specific values
\begin{equation*}
s^j_l=0.
\end{equation*}
Take the specialization of the twisted theory to \eqref{parametresl},
\begin{equation*}
s_l^j:=s_l(\lambda_j)~, \quad \textrm{with } \lambda_1 := \lambda \quad \textrm{and} \quad \lambda_{j+1}:=\lambda^{(-a_1)\dotsm(-a_j)},
\end{equation*}
and let the parameter $\lambda$ tend to $1$.
By Theorems \ref{main} and \ref{compat}, we get the FJRW theory.
Observe that the pairing
\begin{equation*}
\left( \e_{\gamma} , \e_{\gamma^{-1}} \right)^\lambda := \prod_{x_j \in \rR_{\gamma}} (1-\lambda_j) \cdot \prod_{x_j \in \bB_\gamma \backslash \rR_{\gamma}} (1-\lambda_j)^{-1}
\end{equation*}
tends to the pairing \eqref{pairingcomput} of the state space $\st$.

Recall that the Bernouilli polynomials $B_n(x)$ are defined by
\begin{equation*}
\sum_{n=0}^\infty B_n(x) \frac{z^n}{n!} = \frac{z e^{xz}}{e^z-1}
\end{equation*}
and consider the rational number $\Gamma_j(i)$ determined by
\begin{equation*}
\exp(2 \ci \pi \Gamma_j(i))=\gamma_j(i)~, ~~ \Gamma_j(i) \in \left[ 0 , 1 \right[;
\end{equation*}
notice that the multiplicity (see \eqref{multiplicities}) of the line bundle $\cL_j$ at the marked point $\sigma_i$ is $r \cdot \Gamma_j(i)$.
Introduce the notation 
\begin{equation}\label{notationGamma}
\Gamma^\rR_j(i) = \left\lbrace \begin{array}{ll}
1 & \textrm{if } N-j \textrm{ is even and } \Gamma_j(i)=0, \\
\Gamma_j(i) & \textrm{otherwise.} \\
\end{array}\right. 
\end{equation}
First stated in the context of FJRW theory in \cite[Proposition 5.2]{Chiodo0}, the next theorem is a straightforward generalization of \cite[Proposition 4.1.5]{LG/CYquintique} or \cite[Theorem 3.6]{LG/CY}, and an analog of \cite[Theorem 4.1]{Coates3}.
The operator $\Delta$ relies on the formula for the Chern character $\Ch(R\pi_*\cL^\rR_j)$ in terms of psi-classes, which is the formula of Theorem \cite[Theorem 1.1.1]{Chiodo1} where we substitute $\Gamma^\rR_j(i)$ for $m_i/r$ and $\fq_j$ for $s/r$.

\begin{thm}[Chiodo--Zvonkine \cite{LG/CYquintique}]\label{operateursympl}
The transformation $\Delta \colon \cH \rightarrow \cH$ defined by
\begin{equation*}
\Delta = \bigoplus_{\gamma \in \textrm{Aut}(W)} \prod_{j=1}^N \exp \biggl( \sum_{l \geq 0} s_l^j \frac{B_{l+1}(\Gamma^\rR_j)}{(l+1)!} z^l \biggr)
\end{equation*}
gives a linear symplectomorphism between $(\cH,\Omega^{\textrm{un}})$ and $(\cH,\Omega^{\textrm{tw}})$, and
\begin{equation*}
\cL_\textrm{tw} = \Delta (\cL_\textrm{un}).
\end{equation*}
\end{thm}


\subsection{Big I-function, J-function and mirror map}\label{4.4}
By Theorem \ref{operateursympl}, the operator $\Delta$ sends any function $\st \rightarrow \cL_\textrm{un}$ on a function $\st \rightarrow\cL_\textrm{tw}$.
Start with the untwisted J-function
\begin{equation*}
J_\textrm{un}(h,-z)= -z \sum_{n \geq 0} ~~~~~~ \sum_{\overline{\gamma} \in \Aut(W)^n} \cfrac{1}{(-z)^n} \cdot \frac{h^{\gamma(1)} \dotsm h^{\gamma(n)}}{n!} \cdot \e_{\omega(\overline{\gamma})},
\end{equation*}
with $h = \sum_{\gamma} h^\gamma \e_\gamma \in \st$.
Here, $\omega_j(\overline{\gamma})$ is the unique diagonal automorphism such that the component
\begin{equation}\label{virtualcomponent}
\sS_{0,n+1}(\gamma(1),\dotsc,\gamma(n),(\omega(\overline{\gamma}))^{-1})
\end{equation}
of the moduli space is non-empty, and equals by \eqref{selecrule}
\begin{equation}\label{notationpointsuppl}
\omega_j(\overline{\gamma}) :=  \gamma_j(1) \dotsm \gamma_j(n) \exp(2 \ci \pi \fq_j (1-n)).
\end{equation}
As in \cite[Lemma 4.1.10]{LG/CYquintique}, introduce the function
\begin{equation*}
G_y(x,z) := \sum_{m,l \geq 0} s_{l+m-1} \frac{B_m(y)}{m!} \frac{x^l}{l!} z^{m-1} \qquad \textrm{(with $s_{-1}:=0$),}
\end{equation*}
and the untwisted I-function
\begin{equation}\label{Ifunctionuntwisted}
I_\textrm{un}(h,-z) = \prod_{j=1}^N \exp \( - G_{\fq_j} (z \nabla_j,z) \) J_\textrm{un}(h,-z),
\end{equation}
where $\nabla_j$ is the differential operator
\begin{equation*}
\nabla_j := \sum_{\gamma} (\Gamma^\rR_j - \fq_j) h^{\gamma} \frac{\partial}{\partial h^{\gamma}}.
\end{equation*}
This function lies on the untwisted Lagrangian cone $\cL_\textrm{un}$ (see \cite[Lemma 4.1.10]{LG/CYquintique} or \cite[Equation (14)]{Coates3}).
Apply the operator $\Delta$ and get the function
\begin{equation}\label{Itwisted}
I_\textrm{tw}(h,-z)= -z \sum\limits_{\substack{n \geq 0 \\ \overline{\gamma} \in \Aut(W)^n}} \frac{h^{\gamma(1)} \dotsm h^{\gamma(n)}}{n! (-z)^n}  \exp \biggl( -\sum\limits_{\substack{1 \leq j \leq N \\ 0 \leq m <\cD^\rR_j(\overline{\gamma})}} \bs^j ( \omega^\rR_j(\overline{\Gamma}) z +m z ) \biggr) \e_{\omega(\overline{\gamma})},
\end{equation}
where
\begin{equation}\label{notationfin}
\begin{array}{lcl}
\cD^\rR_j(\overline{\gamma}) & := & \bigl(\fq_j + \sum_{i=1}^n (\Gamma^\rR_j(i) - \fq_j)\bigr) - \omega_j^\rR(\overline{\Gamma}) \\[0.2cm]
& = & \lfloor \fq_j + \sum_{i=1}^n (\Gamma^\rR_j(i)-\fq_j) \rfloor - \lfloor \omega^\rR_j(\overline{\Gamma}) \rfloor \\
\end{array}
\end{equation}
and where we adopt the convention
\begin{equation*}
\sum_{0 \leq m < -M} u_m := - \sum_{0 < m \leq M} u_{-m}.
\end{equation*}
In the expressions \eqref{Itwisted} and \eqref{notationfin}, $\omega^\rR_j(\overline{\Gamma})$ is the notation \eqref{notationGamma} used for $\omega_j(\overline{\gamma})$, that is,
\begin{equation*}
\exp (2 \ci \pi \omega^\rR_j(\overline{\Gamma})) = \omega_j(\overline{\gamma}) ~,~~ \omega^\rR_j(\overline{\Gamma}) \left\lbrace \begin{array}{ll}
 = 1 & \textrm{if } N-j \textrm{ is even and } \omega_j(\overline{\gamma})=1, \\
 \in \left[ 0,1 \right[  & \textrm{otherwise,} \\
\end{array}\right. 
\end{equation*}
and $\bs^j(t)$ is the generating series
\begin{equation*}
\bs^j(t):=\sum_{l \geq 0} s^j_l ~ \frac{t^l}{l!}.
\end{equation*}
Besides, we notice the relation
\begin{equation}\label{CH0}
- \Ch_0 (R\pi_*(\cL^\rR_j)) = \cD^\rR_j(\overline{\gamma}) + (-1)^{N-j} \delta_{\omega_j(\overline{\gamma})=1}.
\end{equation}


\begin{thm}\label{bigItheorem}
Let $W$ be a chain polynomial, not necessarily of Calabi--Yau type.
With the notations as above, the big I-function for the Landau--Ginzburg orbifold $(W,\Aut(W))$ lies on the associated Givental Lagrangian cone $\cL$ and
\begin{equation}\label{bigI}
I^\mathrm{big}(h,-z)= -z \sum\limits_{\substack{n \geq 0 \\ \overline{\gamma} \in \Aut(W)^n}} \frac{h^{\gamma(1)} \dotsm h^{\gamma(n)}}{n! (-z)^n} ~ M_1(\overline{\gamma}) \dotsm M_N(\overline{\gamma}) ~ \e_{\omega(\overline{\gamma})},
\end{equation}
where the contribution $M_j(\overline{\gamma})$ is
\begin{equation*}
M_j(\overline{\gamma}) = \left\lbrace
\begin{split}
\prod_{0 \leq m \leq \cD^\rR_j(\overline{\gamma}) - 1}  (\omega^\rR_j(\overline{\Gamma}) + m) z & \qquad \textrm{ when $\cD^\rR_j(\overline{\gamma}) \geq 1$}, \\
1 & \qquad \textrm{ when $\cD^\rR_j(\overline{\gamma}) = 0$}, \\
\prod_{1 \leq m \leq -\cD^\rR_j(\overline{\gamma})}  \cfrac{1}{(\omega^\rR_j(\overline{\Gamma}) - m) z}, & \qquad \textrm{ when $\cD^\rR_j(\overline{\gamma}) \leq -1$}. \\
\end{split}
\right.
\end{equation*}
In the case where $\omega^\rR_j(\overline{\Gamma})=1$ and $\cD^\rR_j(\overline{\gamma}) \leq -1$, we have always $\omega^\rR_{j+1}(\overline{\Gamma})=0$ and $\cD^\rR_{j+1}(\overline{\gamma}) \geq 1$ and we then take the convention
\begin{equation*}
\frac{\omega^\rR_{j+1}(\overline{\Gamma})}{\omega^\rR_j(\overline{\Gamma}) - 1} = -a_j.
\end{equation*}
\end{thm}

\begin{proof}
Look at the twisted I-function \eqref{Itwisted} and specialize the parameters $s^j_l$ to \eqref{parametresl}, observing that
\begin{equation*}
\exp (-\bs (t,x)) = \left\lbrace 
\begin{array}{ll}
\cfrac{e^t-x}{e^t-1} ~ t & \textrm{ if $t \neq 0$} \\
1-x & \textrm{ if $t = 0$}
\end{array} \right. 
\quad \textrm{with } \bs (t,x)= \sum_{l \geq 0} s_l(x)\frac{t^l}{l!}.
\end{equation*}
When we take the limit $\lambda \to 1$ to get
\begin{equation*}
\exp \biggl( -\sum\limits_{\substack{1 \leq j \leq N \\ 0 \leq m <\cD^\rR_j(\overline{\gamma})}} \bs ( \omega^\rR_j(\overline{\Gamma}) z +m z , \lambda_j) \biggr) \longrightarrow M_1(\overline{\gamma}) \dotsm M_N(\overline{\gamma}),
\end{equation*}
we have three cases to face:
\begin{enumerate}
\item if $\omega^\rR_j(\overline{\Gamma})=1$, $\cD^\rR_j(\overline{\gamma}) \leq -1$, and $m=-1$, then
\begin{equation*}
\exp \left(\bs ( \omega^\rR_j(\overline{\Gamma}) z - z , \lambda_j) \right) = \cfrac{1}{1-\lambda_j},
\end{equation*}
\item if $\omega^\rR_j(\overline{\Gamma})=0$, $\cD^\rR_j(\overline{\gamma}) \geq 1$, and $m=0$, then
\begin{equation*}
\exp \left(-\bs ( \omega^\rR_j(\overline{\Gamma}) z , \lambda_j) \right) = 1-\lambda_j,
\end{equation*}
\item otherwise, we have
\begin{eqnarray*}
\exp \left( -\bs ( \omega^\rR_j(\overline{\Gamma}) z + m z , \lambda_j) \right)  & \longrightarrow & (\omega^\rR_j(\overline{\Gamma}) + m) z \qquad \textrm{or} \\
\exp \left( \bs ( \omega^\rR_j(\overline{\Gamma}) z - m z , \lambda_j) \right) & \longrightarrow & \frac{1}{(\omega^\rR_j(\overline{\Gamma}) - m) z}.
\end{eqnarray*}
\end{enumerate}

The main observation is that if case $(1)$ appears for an index $j$, then case $(2)$ appears for the following index $j+1$; in particular, case $(1)$ never appears for $j=N$.
Indeed, as $\omega^\rR_j(\overline{\gamma})=1 \implies \omega^\rR_{j+1}(\overline{\gamma})=1$, then we just have to prove
\begin{equation}\label{degco}
\cD^\rR_j(\overline{\gamma}) \leq -1 \implies \cD^\rR_{j+1}(\overline{\gamma}) \geq 1
\end{equation}
and it is done by a direct computation.
%
%
Thus, when case $(1)$ appears for $j$, we get
\begin{equation*}
\exp \left(\bs ( \omega^\rR_j(\overline{\Gamma}) z - z , \lambda_j) - \bs ( \omega^\rR_j(\overline{\Gamma}) z , \lambda_{j+1})\right) = \cfrac{1-\lambda_{j+1}}{1-\lambda_j} \longrightarrow -a_j.
\end{equation*}

At last, we end with the big I-function \eqref{bigI}.
Since the twisted I-function \eqref{Itwisted} lies on the twisted Lagrangian cone and the FJRW theory is a limit of the twisted theory, then the big I-function \eqref{bigI} lies on the Lagrangian cone $\cL$.
\end{proof}

\begin{rem}\label{powerdisc}
The powers of $z$ in the big I-function are
\begin{equation}\label{powerofzbig}
1-n-N+2 \sum_{j=1}^N \fq_j + \sum_{i=1}^n \frac{1}{2}\deg (\e_{\gamma(i)}) + \frac{1}{2}\deg(\e_{\omega(\overline{\gamma})^{-1}}).
\end{equation}
When the argument $h$ is a sum of states of degrees less than $2$, then the powers of $z$ are less than $1$.
Furthermore, the coefficient of $z$ is of the form $\omega_0(h) \e_\grj$ where $\omega_0$ is a scalar function supported on the subspace of degree-two states.
Thus, we can deduce the so-called small J-function, as in the following theorem.
\end{rem}

\begin{thm}\label{thmfinal}
Let $W$ be a chain polynomial of Calabi--Yau type\footnote{We assume Calabi--Yau condition for the formula \eqref{PicardFuchsequation} to be the Picard--Fuchs equation of the mirror polynomial  $W^\vee$, see \cites{Morrisson,Ebelingstudent}.}, with weights $w_1,\dotsc,w_N$ and degree $d:=w_1+\dotsb+w_N$.
The I-function defined for $t \in \CC^*$ by
\begin{equation}\label{finalformulaI}
I(t,-z) = -z \sum_{k=1}^\infty t^k \frac{\prod_{j=1}^N \prod_{\delta_j < b < \fq_j k, \langle b \rangle = \langle \fq_j k \rangle} bz}{
\prod_{0 < b < k} bz} e_{\grj^k}, \quad \delta_j := - \delta_{\{N-j \textrm{ is odd}\}}
\end{equation}
lies on the Lagrangian cone $\cL$ of the FJRW theory of the Landau--Ginzburg orbifold $(W,\Aut(W))$. 
This function satisfies the Picard--Fuchs equation
\begin{equation}\label{PicardFuchsequation}
\biggl[ t^d \prod_{j=1}^N \prod_{c=0}^{w_j-1} (\fq_j t \frac{\partial}{\partial t} + c) - \prod_{c=1}^d (t \frac{\partial}{\partial t} - c)\biggr] \cdot I(t, -z) = 0
\end{equation}
of the mirror polynomial $W^\vee$.

Furthermore, there is a function $\omega_0(t) \colon \CC^* \to \CC^*$ and some functions $\omega_1(t),\dotsc$, $\omega_{N-2}(t)$ with values in $\st_\mathrm{narrow}$ such that
\begin{equation}\label{powerofz}
I(t, -z) = \omega_0(t) \cdot e_\grj \cdot (-z) + \omega_1(t) + \omega_2(t) (-z)^{-1} + \dotsb + \omega_{N-2}(t) (-z)^{3-N}
\end{equation}
and the J-function is
\begin{equation}\label{finalformulaJ}
J(\tau(t), -z) = e_\grj \cdot (-z) + \tau(t) + \frac{\omega_2(t)}{\omega_0(t)} \cdot (-z)^{-1} + \dotsb + \frac{\omega_{N-2}(t)}{\omega_0(t)} \cdot (-z)^{3-N},
\end{equation}
where the so-called mirror map $\tau$ is
\begin{equation*}
\tau (t) = \frac{\omega_1(t)}{\omega_0(t)}.
\end{equation*}
Restricted to a sufficiently small pointed disk $\Delta^*$ of $\CC$ around $0$, the mirror map $\tau$ is an embedding of $\Delta^*$ into the degree-$2$ part of the state space $\st$.
\end{thm}

\begin{proof}
Restrict the argument of the big I-function \eqref{bigI} to
\begin{equation*}
h = t \cdot \e_{\grj^2} \in \st
\end{equation*}
and consider the function $I(t,-z) := t \cdot I^\mathrm{big}(- t \cdot \e_{\grj^2}, -z)$, which is equal to
\begin{equation}\label{finalformulaIbis}
I(t,-z) = -z \sum_{k=1}^\infty t^k \frac{\prod_{j=1}^N \prod_{\delta_j < b < \fq_j k, \langle b \rangle = \langle \fq_j k \rangle} bz}{
\prod_{0 < b < k} bz} \e_{\grj^k}, \quad \delta_j := - \delta_{\{N-j \textrm{ is odd}\}}.
\end{equation}
By the properties of the Lagrangian cone, this function lies also on the cone $\cL$.

Observe that \eqref{finalformulaIbis} slightly differs from \cite[Equation (40)]{LG/CY}, but only in appearance.
Indeed, in \cite[Equation (40)]{LG/CY}, there are no broad states and in equation \eqref{finalformulaIbis} their contributions equal zero, because the product 
\begin{equation*}
\prod_{-1 < b < \fq_j k, \langle b \rangle = \langle \fq_j k \rangle} bz
\end{equation*}
vanishes when $\langle \fq_j k \rangle=0$ (it happens when $N-j$ is odd) and because when $\langle \fq_N k \rangle=0$ and $\langle \fq_j k \rangle \neq 0$ for $j < N$, the decoration $\rR_{\grj^k}$ is not balanced and we have $\e_{\grj^k}=0$.
Moreover, since every narrow state satisfies $\e_{j^k}=e_{j^k}$, then \eqref{finalformulaIbis} $=$ \eqref{finalformulaI}. 

It is straightforward to check that $I(t, -z)$ satisfies the Picard--Fuchs equation \eqref{PicardFuchsequation}, as in \cite[Equation (57)]{LG/CY}.
Equation \eqref{powerofz} follows from Remark \ref{powerdisc}.
Since the J-function is the unique function on $\cL$ with the property $J(h, -z)= -z e_\grj + h + \textrm{o} (z^{-1})$, then we obtain \eqref{finalformulaJ}.
The mirror map $\tau$ satisfies $\tau(t) = t + o(t)$, hence it defines an embedding of a sufficiently small pointed disk $\Delta^*$ in $\st$. By a simple computation, we check that $\deg (\omega_1(t))=2$ for any $t \in \Delta^*$. 
\end{proof}

\begin{exa*}
Consider the chain polynomial $W=x_1^2x_2+x_2^3x_3+x_3^5x_4+x_4^{10}x_5+x_5^{11}$ with weights $(4,3,2,1,1)$ and degree $11$.
By Theorem \ref{thmfinal}, we can compute explicitly functions $I$ and $J$.
Looking at the coefficient of $z^{-1} e_{\grj^5}$, we show that the following relation between certain non-concave correlators must hold:
\begin{equation}\label{correlation}
\frac{121}{12} \langle e_{\grj^2}^4e_{\grj^6} \rangle - \frac{11}{2} \langle e_{\grj^2}^2e_{\grj^3}e_{\grj^6} \rangle + \frac{5}{3} \langle e_{\grj^2} e_{\grj^4} e_{\grj^6} \rangle + \frac{1}{4} \langle e_{\grj^3}^2 e_{\grj^6} \rangle = 3.
\end{equation}
We have already computed
\begin{equation*}
\langle e_{\grj^3},e_{\grj^3},e_{\grj^6} \rangle=-2 \quad \textrm{and} \quad \langle e_{\grj^2}^4e_{\grj^6} \rangle = -\frac{2}{121} \quad \textrm{(see \eqref{threepoint}.}
\end{equation*}
The correlator $\langle e_{\grj^2} e_{\grj^4} e_{\grj^6} \rangle$ equals $1$ (this is a concave case).
The last correlator to compute is
\begin{equation*}
\langle e_{\grj^2}^2e_{\grj^3}e_{\grj^6} \rangle = -\frac{4}{11}.
\end{equation*}
\end{exa*}

\end{document}